\documentclass[leqno, 12pt]{amsart}
\usepackage[utf8]{inputenc}
\usepackage{lmodern}
\usepackage[english]{babel}
\usepackage{amstext}
\usepackage{amsthm}
\usepackage{amsmath}
\usepackage{amssymb}
\usepackage{graphicx}
\usepackage{mathtools}
\usepackage{ulem}
\usepackage{xcolor}
\DeclareRobustCommand{\erase}{\bgroup\markoverwith{\textcolor{red}{\rule[.5ex]{2pt}{0.4pt}}}\ULon}

\usepackage[svgnames]{xcolor}
\definecolor{color1}{RGB}{27,158,119}
\definecolor{color2}{RGB}{217,95,2}
\definecolor{color3}{RGB}{117,112,179}
\definecolor{color4}{RGB}{231,41,138}
\usepackage[margin=1in]{geometry}
\usepackage{dsfont}
\usepackage{fancyvrb}
\usepackage[unicode=true,pdfusetitle,
 bookmarks=true,bookmarksnumbered=false,bookmarksopen=false,
 breaklinks=false,pdfborder={0 0 1},backref=false,colorlinks=true]
 {hyperref}
 \usepackage{enumitem}

\hypersetup{
  pdftitle={Ground states on a fractured strip and one-dimensional reduction},
  pdfauthor={Stefan Le Coz  and Boris Shakarov},
  pdfsubject={},
  pdfkeywords={}, 
  colorlinks=true,
  linkcolor=DarkBlue  ,          %
  citecolor=DarkRed,        %
  filecolor=DarkMagenta,      %
  urlcolor=DarkGreen,           %
}

\newtheorem{theorem}{Theorem}[section]
\newtheorem{lem}[theorem]{Lemma}
\newtheorem{prop}[theorem]{Proposition}
\newtheorem{proposition}[theorem]{Proposition}

\newtheorem{remark}[theorem]{Remark}
\theoremstyle{definition}
\newtheorem{defin}[theorem]{Definition}

\def\R{\mathbb R}
\def\C{\mathbb C}
\def\N{\mathbb N}

\def\H{\operatorname{H}}

\newcommand{\Strip}{\mathbb S}
\newcommand{\deltastrip}{\delta_0(x)}

\DeclareMathOperator{\sech}{sech}

\DeclarePairedDelimiterX{\dual}[2]{\langle}{\rangle}{#1, #2}

\title[Bifurcation Analysis on a Fractured Strip]{Transverse Instability and Bifurcation Analysis of the line soliton for the NLS equation on a Fractured Strip}

\author[H. Kikuchi]{Hiroaki Kikuchi}
\address{
Department of Mathematics;
Tsuda University; 
2-1-1 Tsuda-machi, Kodaira-shi, Tokyo 187-8577 (JAPAN)
}
\email{hiroaki@tsuda.ac.jp}
\thanks{The work of H. K. was supported 
by JSPS KAKENHI 
Grant Number JP25H01453 and JP25K07089}

\author[B. Shakarov]{Boris Shakarov}
\address{Institut de Mathématiques de Toulouse ; UMR5219, Université de Toulouse ; CNRS, UPS IMT, F-31062 Toulouse Cedex 9 (France)}
\email{boris.shakarov@math.univ-toulouse.fr}
\thanks{The work of B. S. is 
  partially supported by ANR-11-LABX-0040-CIMI and the ANR project NQG ANR-23-CE40-0005}

\author[K. Tomioka]{Kenta Tomioka}
\address{Department of Mathematics;
Tsuda University; 
2-1-1 Tsuda-machi, Kodaira-shi, Tokyo 187-8577 (JAPAN)}
\email{kenta.tomioka@gm.tsuda.ac.jp}
\thanks{
The work of K. T. was supported by JSPS KAKENHI Grant Number 25KJ0325.
}

\date{\today}

\subjclass[2010]{35Q55 (35A15, 35B38)}

\date{\today}
\keywords{nonlinear Schr\"odinger equation, standing waves, action
  ground state, energy ground state, nonlinear quantum graphs, dimensional reduction}

\begin{document}

\begin{abstract}
We consider the nonlinear Schrödinger equation on a two-dimensional strip with an attractive $\delta$ interaction and power nonlinearity. We investigate the transverse stability and bifurcation of line solitons as the width of the strip varies. We first establish local well-posedness in $H^1$, conservation of mass and energy, and global existence in the $H^1$-subcritical regime. We then identify a critical width $L_*$ at which the line soliton undergoes a transverse instability. More precisely, we prove orbital stability for $L<L_*$ and orbital instability for $L>L_*$. At the critical width, a simple eigenvalue of the linearized operator crosses zero, and we construct, via the Lyapunov-Schmidt reduction, a branch of positive nontrivial stationary solutions bifurcating from the line soliton. We determine the direction of this bifurcation by computing the second-order variation of the width along the branch. Finally, we investigate the orbital stability of the bifurcating solitons and obtain a stability criterion which can be evaluated in the regime of sufficiently small interaction strength.
\end{abstract}

\maketitle

\section{Introduction}

Nonlinear Schrödinger equations on waveguides have attracted considerable attention because they combine nonlinear dispersive dynamics with geometric confinement. A particularly interesting situation arises when the waveguide contains a localized defect, which models impurities or fractures and can fundamentally alter both the stationary states and their stability. Understanding the interaction between confinement, localized defects, and transverse instabilities is the main objective of this work. Thus, we consider the nonlinear Schr\"odinger equation on a strip with Neumann boundary conditions and a delta condition on the $x$-axis, modeling a fracture. Here, we denote the strip by $\Strip_L = \R \times [0,L]$ and we consider the following Nonlinear Schrödinger equation with Neumann boundary conditions
\begin{equation}\label{eqDeltaStrip} 
\begin{cases}
     i \partial_t u + \H u - |u|^{p-1} u = 0, \quad (x,y) \in \Strip_L,\\
     \partial_\nu u = 0\, \mbox{ on } \partial\Strip_L= \R \times \{0,L\}, \\
    u(t=0) = u_0 \in H^1(\Strip_L),
\end{cases}  
\end{equation}
where formally
\begin{equation*}
    \H = -\Delta + \gamma \delta_0(x)
\end{equation*}
is a Hamiltonian operator given by a compact perturbation of the Laplacian on the strip, 
 $\gamma\in\R$, the symbol $\deltastrip  
 \in \mathcal{L} 
 (H^{1}(\Strip_L), H^{-1}(\Strip_L))$ 
 is defined 
for $u,v\in H^{1}(\Strip_L)$ by 
\[
(\deltastrip  u, v):= \int_0^L \Re\left(u(0,y)\overline{v(0,y)} \right)dy,
\]
and $\partial_\nu$ denotes the normal derivative on the borders and could be replaced by $\partial_y$ in the present context. By the trace theorem (see e.g. \cite[p. 315]{Brezis2011} or Section \ref{secPreliminar}), the quadratic form associated to $\H$ 
\begin{equation}\label{eqQF}
    Q_{\H}(u) = \| \nabla u \|_{L^2(\Strip_L)}^2 + \gamma \int_0^L |u(0,y)|^2dy
\end{equation}
is well defined on the domain $\mathcal{D}(Q_{\H}) = H^1(\Strip_L)$. The symbol $\delta_0(x)$ can also be written as $\tau^* \tau$, where $\tau$ is the trace operator defined in Section 2.
We will refer to the case $\gamma >0$ as the \textit{repulsive case}, and $\gamma <0$ as the \textit{attractive case}. Solutions to this equation satisfy formally to conservation laws, that is, of the energy and the mass defined as

\begin{align}    
\label{eqEnergy2D}
    E_\gamma(u) &= \frac{1}{2} \| \nabla u \|_{L^2(\Strip_L)}^2 + \frac{\gamma}{2} \int_0^L |u(0,y)|^2 dy - \frac{1}{p+1} \| u \|_{L^{p+1}(\Strip_L)}^{p+1},
\\
    \label{eqMass2D}
     M(u) &= \| u\|_{L^2(\Strip_L)}^2.
\end{align}

The Hamiltonian operator $\H$ models the transverse extension of the classical one-dimensional Laplacian perturbed by a Dirac delta in zero $-\partial_{xx} + \gamma \delta_0$ \cite{CoSh24}. To give more details, let us first introduce the elliptic stationary equation associated with \eqref{eqDeltaStrip}, that is

\begin{equation}\label{eqStationary2D} 
\begin{cases}
     \omega u + \H u  - |u|^{p-1} u = 0, \quad (x,y) \in \Strip_L,\\
     \partial_\nu u = 0\, \mbox{ on } \partial\Strip_L= \R \times \{0,L\}, 
\end{cases}  
\end{equation}

The existence of positive solutions to \eqref{eqStationary2D} was established via variational techniques in \cite{CoSh24}. Specifically, for any $\gamma \geq 0$, $L > 0$, and fixed mass $m > 0$, there exists a constrained energy minimizer $Q \in H^1(\Strip_L)$ satisfying \eqref{eqStationary2D} with a frequency $\omega = \omega(m) > \frac{\gamma^2}{4}$. A key structural feature uncovered by the second author and Le Coz in \cite{CoSh24} is a length-dependent dimensional transition: for sufficiently small domains ($0 < L < \tilde{L}(m)$), these mass-constrained minimizers reduce to $1D$ ground states trivially extended in the transverse variable $y$.  This dimensional reduction breaks down for large $L$. While the variational analysis in \cite{CoSh24} established the existence of this dimensional transition via mass-constrained minimization, the underlying dynamical and spectral mechanisms remained unexplored.

Our primary objective is to identify the precise bifurcation threshold that governs the transition from $1D$ line solitons to genuine $2D$ profiles, and to determine how this geometric shift dictates the orbital stability of the resulting standing waves

Our approach naturally divides into three steps. We first establish the well-posedness theory for the evolution problem. We then identify a critical strip length at which the line soliton undergoes a bifurcation, giving rise to a new branch of positive stationary solutions. Finally, we investigate the orbital stability of both branches and show how the bifurcation influences the nonlinear dynamics.

We now proceed to describe more thoroughly each part. We first establish well-posedness of the model in the following theorem.
\begin{theorem}\label{thmLWP}
    For any $u_0 \in H^1(\Strip_L)$ there exists a unique solution $u \in C([0,T_{max}), H^1(\Strip_L))$ maximal in time. For any $t \in [0,T_{max})$, $E_{\gamma}(u(t)) = E_{\gamma}(u_0)$ and $M(u(t))= M(u_0)$. The blow up alternative holds: either $T_{max} = \infty$ or $T_{max} < \infty$ and 
    \begin{equation}\label{eqBUAlternative}
     \lim_{t \to T_{max}^-} \| \nabla u \|_{L^2(\Strip_L)}  = \infty.
    \end{equation}
    Finally, if $p<3$, then for any $u_0 \in H^1(\Strip_L)$, $T_{max} = \infty$. 
\end{theorem}

The proof relies on a fixed-point argument in the spirit of Kato \cite{Ka87}, combined with Strichartz estimates adapted to the strip geometry. Since dispersion occurs only in the unbounded $x$-direction, we first decompose the problem along the Neumann eigenfunctions in the transverse variable, reducing the propagator to a family of one-dimensional Schrödinger operators with point interactions. Uniform Strichartz estimates for these operators yield the required mixed-space-time estimates on the strip. The singular interaction is therefore incorporated at the level of the one-dimensional Hamiltonians, rather than treated as a perturbative nonlinear term, in the spirit of the one-dimensional case with Dirac potentials (see, e.g., Fukuizumi, Ohta, and Ozawa \cite{zbMATH05355749}). Global well-posedness follows from the Gagliardo-Nirenberg inequality and the conservation of the energy. 

The scope of the second part is to establish a pitchfork bifurcation and a critical length from the line soliton. For this, let us denote by  
\[
    \phi_{\omega,\gamma}^{(L)}(x,y)
    := \phi_{\omega,\gamma}(x),
    \qquad (x,y)\in\mathbb{S}_{L}.
\] 
be the line soliton trivially extended in the $y$ direction to a function on $\mathbb{S}_{L}$. 
This profile $\phi_{\omega, \gamma}$ 
is explicitly known, see \eqref{eqGS1D} below.  We show that it undergoes a pitchfork bifurcation at the length
\begin{equation}\label{eqLCritic}
    L_{*} := \frac{\pi}{\sqrt{\lambda_{1}}} >0,
\end{equation}
where $- \lambda_{1} <0$ is the smallest, simple, and isolated eigenvalue associated with the linearized operator $\mathcal{L}_{+,0}$ defined in Proposition \ref{prpSpectr}.  
We emphasize that while the local well-posedness theory (Theorem \ref{thmLWP}) is established for arbitrary coupling constants $\gamma \in \mathbb{R}$, the bifurcation and stability analysis in the remainder of this work focuses exclusively on the attractive case $\gamma < 0$, as in the repulsive regime, line solitons are known to be unstable already on the line, as established by Le Coz et al. \cite{CoFuFi08}.
The main result is the following.

\begin{theorem}\label{thm1-bi}
Let $p\geq 2$ and $\gamma <0$. 
We define $\chi_{\gamma}$ as the normalized eigenfunction 
of $\mathcal{L}_{+, 0}$ corresponding to $- \lambda_{1}$.
There exist $a_0>0$, a function \( L\in C^{2}\big(( -a_0,a_0);\mathbb R^{+}\big) \) with $L_{*}=L(0)$
and a family of solutions $\varphi(a)\in \mathcal{D}(H_{L(a)})$ of \eqref{eqStationary2D} for each $a\in(-a_0,a_0)$ with the following properties. Denote by
\[
T_a:\mathbb S_{L_*}\to\mathbb S_{L(a)},\qquad T_a(x,y)=(x,\tfrac{L(a)}{L_*}y),
\]
the natural rescaling map. Define the pulled–back profiles
\[
\Phi(a):=\varphi(a)\circ T_a \in \mathcal{D}(H_{L_*}).
\]
Then the map \(a\mapsto\Phi(a)\) belongs to
\(
C^{2}\big((-a_0,a_0); \mathcal{D}(H_{L_*})\big)
\),
and $\varphi(a)>0$ on $\mathbb S_{L(a)}$ for all $a\in(-a_0,a_0)$.

Moreover, there exists $\Psi(a)\in \mathcal{D}(H_{L_*}) $ such that, for every $a\in(-a_0,a_0)$,
\[
\Phi(a)
= \bigl(\phi_{\omega,\gamma}^{(L(a))}\bigr)\circ T_a
  + a\,\chi_{\gamma} \cos\!\Big(\frac{\pi}{L(a)}\,\cdot\; \frac{L(a)}{L_*}y\Big)
  + \Psi(a),
\]
with
\[
\Psi\in C^{2}\big((-a_0,a_0); \mathcal{D}(H_{L_*}) \big),
\qquad
\|\Psi(a)\|_{H^{2}(\mathbb S_{L_*})}=O(a^{2})\ \text{ as }a\to0, 
\]

Finally,
\[
\frac{dL}{da}(0)=0,
\qquad
L(a)=L_{*}+\frac12\frac{d^{2}L}{da^{2}}(0)\,a^{2}+o(a^{2}).
\]
\end{theorem}
To prove this theorem, we follow the robust method of Kirr, Kevrekidis and Pelinovsky \cite{KiKePe11Bi} and Pelinovsky and Phan \cite{zbMATH06092526}  in the context of symmetry-breaking bifurcation of NLS with symmetric potential. We first identify the critical length $L_*$ at which the linearized operator develops a simple kernel generated by the first transverse mode. The Lyapunov-Schmidt reduction then reduces the problem to a scalar bifurcation equation and yields a local branch of nontrivial positive solutions. The symmetry with respect to the bifurcation parameter gives $L'(0)=0$, while the transversality condition ensures that the corresponding kernel eigenvalue crosses zero at $L=L_*$. Thus, the bifurcation is of pitchfork type, and its direction is determined by the sign of $\frac{d^{2}L}{da^{2}}(0)$.

The computation of $\frac{d^{2}L}{da^{2}}(0)$ is one of the main points of the argument. We differentiate the reduced bifurcation equation and exploit the symmetry of the first transverse mode, which forces the lower-order terms to vanish. The second-order correction to the profile is determined by solving the equation on the orthogonal complement of the kernel. To justify the limiting procedure, we use the exponential decay of both the line soliton and the eigenfunction, relying on the classical result of Berestycki and Nirenberg \cite{BeNi90}. 

A technical distinction arises according to the value of $p$. For $p \geq 3$, the required differentiability properties of the nonlinearity allow the argument to proceed directly. When $2<p<3$, the third derivative of the power nonlinearity is singular at zero, requiring a more delicate dominated-convergence argument based on the exponential decay of the profiles. The endpoint case $p=2$ is more delicate still, since the third derivative is no longer available and the argument must instead exploit the explicit structure of the quadratic nonlinearity. 

In all cases, this yields an explicit expression for $\frac{d^{2}L}{da^{2}}(0)$, whose sign determines on which side of $L_*$ the bifurcating branch exists.
We remark that extending this bifurcation analysis to the sub-quadratic regime $1 < p < 2$ appears feasible by adapting the approach of Akahori et al.\ \cite{AkBaIbKi24}. 

A fundamental distinction separates our setting ($\gamma \neq 0$) from the unperturbed case ($\gamma = 0$), such as the framework considered by Yamazaki \cite{Ya15} (see also e.g. Berestycki and Wei \cite{BeWe10}). In the  latter case, the equation is invariant under the scaling  
\begin{equation*}
    u(t,x,y) \mapsto L^{\frac{2}{p-1}} u\left(L^2 t, Lx, Ly\right).
\end{equation*}
This joint spatial dilation maps the problem from the variable strip $\mathbb{R} \times [0,L]$ onto a fixed domain (e.g., $\mathbb{R} \times [0,1]$), fixing both the spatial Laplacian and the power nonlinearity while pushing the dependence on $L$ entirely into a rescaled frequency parameter $\omega = \omega(L)$. Consequently, the problem reduces to standard bifurcation frameworks parametrized by the frequency, as for instance in \cite{KiKePe11Bi}.

In contrast, in the presence of the point defect, this full spatial scaling breaks down. First, an axial dilation $x \mapsto Lx$ transforms the Dirac distribution into $\frac{\gamma}{L} \delta_0(x)$, altering the defect strength. Second, the kinetic, nonlinear, and point-interaction terms exhibit mutually incompatible scaling exponents. 

To circumvent this, we apply a partial scaling exclusively to the transverse coordinate via the pull-back map $T_a(x,y) = (x, \frac{L(a)}{L_*} y)$. Consequently, the frequency $\omega > 0$ remains fixed, and the length parameter $L$ is embedded directly into the differential operator $\mathcal{H}_{L} = -\partial_{xx} - \frac{1}{L^2} \partial_{yy} + \gamma \delta_0(x)$. The bifurcation parameter $L$ therefore modulates the operator structure itself rather than inducing a frequency shift, necessitating a modified Lyapunov--Schmidt reduction to track the $L$-dependent linearized operators.

In the last part, we investigate the orbital stability of the line soliton and the bifurcated one. Since the singular potential $\gamma \delta_0$ explicitly breaks continuous translation invariance in the axial direction, the continuous symmetry group associated with equation \eqref{eqDeltaStrip} is reduced to gauge invariance. Consequently, the appropriate notion of stability for the standing waves is orbital stability modulo phase rotations. The analysis of orbital stability for nonlinear dispersive equations is a classical subject, rooted in the variational approaches of Cazenave and Lions \cite{CaLi82}, the Lyapunov approach of Weinstein \cite{We85,We86}, and later by Grillakis, Shatah, and Strauss \cite{GrShSt87,zbMATH04169705} via the study of the linearized operator.

In the present context, we adopt the following standard definition.

\begin{defin}
We say that 
$\phi \in H^1(\Strip)$ is \textit{orbitally stable} if for every $\varepsilon > 0$ there exists $\delta > 0$ such that for any initial data $u_0 \in H^1(\mathbb{S})$ satisfying
\[
\| u_0 - \phi \|_{H^1(\mathbb{S})} < \delta,
\]
the corresponding solution $u(t)$ exists globally and satisfies
\[
\sup_{t \in \mathbb{R}} \inf_{\theta \in \mathbb{R}} \| u(t) - e^{i \theta} \phi \|_{H^1(\mathbb{S})} < \varepsilon.
\]
\end{defin}

We will show the following.
\begin{theorem}\label{thm-sta}
Let $p \in (1,5)$ and $\gamma <0$ and $\omega > \frac{\gamma^2}{4}$. Then
\begin{enumerate}
    \item
    For $0 < L < L_{*}$, 
    the line soliton 
    $\phi_{\omega, \gamma}$ is orbitally stable.
    \item 
    For $L > L_{*}$, the line soliton 
    $\phi_{\omega, \gamma}$ 
    is not orbitally stable. 
\end{enumerate}    
\end{theorem}
Theorem \ref{thm-sta} addresses the transverse stability and instability of the line soliton, relying on the classical Grillakis-Shatah-Strauss framework \cite{GrShSt87}. In our parameter regime, the slope condition is inherently satisfied. Consequently, the stability of the line soliton is completely dictated by the Morse index of the associated linearized operator. 

The threshold $L_*$ marks a critical geometric bifurcation point. For narrow domains where $0 < L < L_*$, the transverse frequencies are sufficiently large to maintain the Morse index at $1$. Conversely, when the transverse period crosses the threshold ($L > L_*$), a new negative direction appears, increasing the Morse index to $2$ and inducing linear instability. To upgrade this linear instability to full nonlinear orbital instability, we follow the robust, established frameworks developed by Georgiev-Ohta \cite{MR2916078}. Similarly, one can follow the approach in Rousset-Tzvetkov \cite{RoTz10Inst, zbMATH05488047}. The key point is to show the spectral mapping property, which is done in Appendix \ref{sec:spectral-m} following a similar approach to Gesztesy et al. \cite{GeJoLaSt00}. 

Finally, we study the stability of the bifurcating soliton: 
\begin{theorem}\label{thmStab2}
Let $\gamma<0$, $\omega>\gamma^2/4$, and $2\leq p<5$. 
Let $\varphi(a)$ be the bifurcating family given by 
Theorem~\ref{thm1-bi}. Then, for sufficiently small $|a|>0$,
$\varphi(a)$ is orbitally stable if
\[
\frac{d^2L}{da^2}(0)> 0,
\]
and orbitally unstable if
\[
\frac{d^2L}{da^2}(0)<0.
\]
Moreover,  $\phi(a)$ is orbitally stable for sufficiently small $|\gamma|$ and $|a|$.
\end{theorem}

The proof of Theorem \ref{thmStab2} relies on the classical Grillakis-Shatah-Strauss framework \cite{GrShSt87} for orbital stability, which requires analyzing both the Morse index of the linearized operator and the Vakhitov-Kolokolov condition. We perform an asymptotic analysis of the spectral elements along the bifurcating branch for $ a$ close to zero. We show that the second smallest eigenvalue of the linearized operator expands at leading order as a quantity proportional to $\frac{d^2L}{da^2}(0) a^2$. Consequently, the sign of the coefficient $\frac{d^2L}{da^2}(0)$ dictates the Morse index. 

Simultaneously, we ensure that the 
Vakhitov-Kolokolov condition $\partial_\omega \|\varphi(a)\|_{L^2}^2 > 0$ is preserved along the branch. 

Ultimately, the stability question is reduced to evaluating the sign of $\frac{d^2L}{da^2}(0)$. Because this coefficient can be explicitly computed and is strictly negative in the case $\gamma=0$, a Kato perturbation argument ensures that the negative sign, and therefore the orbital stability, persists for all sufficiently small $|\gamma|>0$. We remark that in a different scaling, similar computations in the unperturbed case $\gamma =0$ has been done in Yamazaki~\cite{Ya15}.

The rest of the paper is organized as follows. In Section~\ref{secPreliminar}, we review the one-dimensional line soliton and analyze the spectrum of the corresponding linearized operators. Section~\ref{secWP} is devoted to local well-posedness, establishing Theorem~\ref{thmLWP}. In Section~\ref{secBifurcation}, we investigate the bifurcation structure, develop the relevant calculus of derivatives, and prove the main bifurcation result. Finally, Section~\ref{sec:stab} addresses the stability of the solutions, presenting the key orbital and asymptotic stability theorems.

\section{Preliminaries}\label{secPreliminar}
\subsection{The trace theorem}

We start by defining the trace of the functions projected on the hyperplane $x = 0$. Let us denote by 
\begin{equation*}
    (\tau f)(x,y) = f(0,y) 
\end{equation*}
for $f$ a continuous function in $\Strip_L$. The trace $\tau u$ is well defined as soon as $u \in W^{s,p}(\Strip_L)$ when $s > 1/p$ and the map $\tau: W^{s,p}(\Strip_L) \to  W^{s-1/p,p}(0,L)$ is bounded, see \cite[Theorem 1.5.1.1]{Gr11}. Consequently, choosing $s=1$ and $p = 2$ yields that the quadratic form \eqref{eqQF} is well defined for $u\in H^1(\Strip_L)$.

\subsection{Line ground states}\label{sec1dGs}
The one-dimensional counterpart  of the two-dimensional model \eqref{eqDeltaStrip} is the equation
\begin{equation}\label{eq1DEqIntr}
    (-\partial_{xx}  + \gamma \delta_0 ) u  + \omega u  - |u|^{p-1} u = 0.
\end{equation}
 Here, $\delta_0$ is the Dirac distribution at the origin, namely, $\dual{\delta_0}{ v} = v(0)$ for $v \in H^1(\R)$. The operator $(-\partial_{xx} + \gamma \delta_0)$ is to be intended as the self-adjoint extension of the Laplacian on the domain 
 \begin{equation*}
     \mathcal{D}(-\partial_{xx}  + \gamma \delta_0) = \{u\in H^2(\R)\; : u(0^+)' - u(0^-)' = \gamma u(0) \}
 \end{equation*}
 where $u(0^\pm)' = \lim_{x \to \pm 0} u(x)'$, see \cite{BeKu13}.

In the case $\gamma = 0$, the set of solutions of \eqref{eq1DEqIntr} is given by 
\begin{equation*}
    \{e^{i \alpha} \phi_{\omega, 0}(. - z) : \alpha \in [0,2\pi), \, z \in \R \}
\end{equation*}
where the profile $\phi_{\omega,0}$ is explicitly calculated by direct integration of the equation and is given by
\begin{equation}
\label{eq:explicit}
     \phi_{\omega, 0}(x) = \left( \frac{(p+1)\omega}{2} \sech^2 \left( \frac{(p-1) \sqrt{\omega}}{2} x  \right)\right)^\frac{1}{p-1}.
 \end{equation}
 Note that we are in the framework of Schr\"odinger equations and therefore the functions that we consider are a priori \textit{complex valued}.

 When $\gamma \neq 0$, solutions of \eqref{eq1DEqIntr} and their relations to the nonlinear Schr\"odinger dynamics have been thoroughly investigated, from the initial work of Goodman, Holmes, and Weinstein \cite{GoHoWe04}, and the stability studies of Fukuizumi and co. \cite{FuJe08,FuOhOz08,CoFuFi08}, up to more recent advanced studies such as the classification of global dynamics of even solutions by Gustafson and Inui \cite{GuIn24} or the construction of a minimal blow-up mass solution by Genoud, Le Coz, and Royer \cite{GeLeRo23}.

Most of the results on solutions to \eqref{eq1DEqIntr} that we are going to use in the present paper have been established in \cite{FuJe08,FuOhOz08,CoFuFi08}. Bounded solutions to \eqref{eq1DEqIntr} exist only when 
\(
\omega > \frac{\gamma^2}{4},
\)
 in which case, they can be obtained explicitly by surgery from \eqref{eq:explicit}. Precisely, for $\omega > \frac{\gamma^2}{4}$, there exists a unique positive solution to \eqref{eq1DEqIntr} given by 
    \begin{equation}\label{eqGS1D}
     \phi_{\omega, \gamma}(x) = \left( \frac{(p+1)\omega}{2} \sech^2 \left( \frac{(p-1) \sqrt{\omega}}{2} |x|  - \tanh^{-1}\left( \frac{\gamma}{2 \sqrt{\omega}} \right)  \right)\right)^\frac{1}{p-1}.
  \end{equation}

  The function in \eqref{eqGS1D} can be characterized as a minimizer of certain variational problems. We define the one-dimensional 
  action, the Nehari functional, energy, and mass by (respectively)
\begin{align}
    \label{eqAction1D} 
    S^{1D}_{\omega,\gamma}(u)& = \frac{1}{2} \int_\R |\partial_x u|^2dx  +  \frac{\omega}{2} \int_\R |u|^2 dx+\frac{\gamma}{2} |u(0)|^2 - \frac{1}{p+1} \int_\R |u|^{p+1} dx , \\ 
    \label{eqNehari1D}
    I^{1D}_{\omega,\gamma}(u) &=  \int_\R |\partial_x u|^2dx + \omega \int_\R |u|^2 dx+ \gamma|u(0)|^2- \int_\R |u|^{p+1} dx ,
    \\
    \label{eqEn1D}
    E^{1D}_\gamma(u) &= \frac{1}{2} \int_\R |\partial_x u|^2dx + \frac{\gamma}{2} |u(0)|^2 - \frac{1}{p+1} \int_\R |u|^{p+1} dx , 
    \\ 
    \label{eqMass1D} 
    M^{1D}(u) &= \int_\R |u|^2  dx.
\end{align}
The profile given in \eqref{eqGS1D} was characterized as an action ground state in \cite{FuJe08, FuOhOz08,GoHoWe04}.  Moreover, it has also been characterized as an energy ground state by Adami, Noja, and Visciglia \cite{AdNoVi13} when $\gamma<0$. The case $\gamma>0$ for energy ground states has been treated by Boni and Carlone \cite {BoCa23} in the case of the half-line. Their results can be transferred directly to symmetric functions on the line. 
The results can be summarized as follows.

\begin{prop}\label{prp1DIntr}
    Let $\gamma \in \R$ and $\omega > \gamma^2/4$. 
    \begin{itemize}
   
    \item Let $1<p<5$, $m>0$. There exists $m^*=m^*(\gamma)$, with $m^*(\gamma)=0$ if $\gamma<0$ and $m^*(\gamma)>0$ if $\gamma>0$, such that the following hold. Assume that $m> m^*$.  Then there exists a unique $\omega(m) > \gamma^2/4$ such that the function $\phi_{\omega(m),\gamma}$ defined in \eqref{eqGS1D} is the unique real-valued and positive minimizer of
\begin{equation}
  \label{eqEnMin1D} 
   \begin{cases}
 e^{1D}_{m,\gamma} = \inf\{E^{1D}_\gamma(u)  : u \in H^1(\R),\, M^{1D}(u) = m \}&\text{ if }\gamma\leq0,\\
 e^{1D}_{m,\gamma,sym} = \inf\{E^{1D}_\gamma(u)  : u \in H^1_{rad}(\R),\, M^{1D}(u) = m \}&\text{ if }\gamma>0.\\
      \end{cases}
\end{equation}
    If $m\leq m^*$, then the problems do not admit a minimizer.
    \end{itemize}
\end{prop}

Notice that in the repulsive case $\gamma >0$, symmetry with respect to the origin is required. This originates from the fact that a minimizing sequence may exhibit a \textit{runaway} behavior at infinity on one side of the line, as shown by Fukuizumi and Jeanjean \cite{FuJe08}, which happens due to the absence of translation invariance. 

Notice also that, even in the symmetric case, minimization of the energy at fixed mass can fail, as it becomes energetically favorable for small masses to divide the sequence into two parts traveling away from the origin. 

Finally, the sign of 
$ \partial_\omega M^{1D}(\phi_{\omega,\gamma})$ has been previously determined in \cite{FuJe08,FuOhOz08,CoFuFi08}).
 \begin{lem}\label{lemStabil}
     If $\gamma < 0$, then the following holds.
     \begin{enumerate}
         \item If $1 < p \leq 5$, then $\partial_\omega M^{1D}(\phi_{\omega,\gamma}) >0$.
         \item If $p >5$, then there exists $\omega_1>\gamma^2/4$ such that $\partial_\omega M^{1D}(\phi_{\omega,\gamma}) >0$ for $\gamma^2/4<\omega < \omega_1$ and  $\partial_\omega M^{1D}(\phi_{\omega,\gamma}) < 0$ for $\omega > \omega_1$. 
     \end{enumerate}
      If $\gamma > 0$, then the following holds.
     \begin{enumerate}
         \item If $1 < p \leq 3$, then $\partial_\omega M^{1D}(\phi_{\omega,\gamma}) >0$.
         \item If $3 < p < 5$, then there exists $\omega_2>\gamma^2/4$ such that $\partial_\omega M^{1D}(\phi_{\omega,\gamma}) >0$ for $\omega > \omega_2$ and  $\partial_\omega M^{1D}(\phi_{\omega,\gamma}) < 0$ for $\gamma^2/4<\omega < \omega_2$.
         \item If $p >5$, then $\partial_\omega M^{1D}(\phi_{\omega,\gamma}) < 0$.
     \end{enumerate}
 \end{lem}

\subsection{Spectrum of the linearized 
operator of \eqref{eq1DEqIntr}
around}

To state our bifurcation result, we recall several classical facts concerning the spectrum of the
linearized operator associated with~\eqref{eq1DEqIntr} around the one-dimensional ground state
$\phi_{\omega,\gamma}$ defined in~\eqref{eqGS1D}.  
We begin by considering the one-dimensional analogue of~\eqref{eqDeltaStrip}, namely
\begin{equation}\label{eqNLS1D}
    i \partial_t u - (\partial_{xx} - \gamma \delta_0) u = |u|^{p-1}u, 
    \qquad x \in \mathbb{R}.
\end{equation}
Recall from \S\ref{secPreliminar} that the operator $-\partial_{xx} + \gamma \delta_0$ is well defined and self-adjoint.
Equation~\eqref{eqNLS1D} admits solitary-wave solutions of the form 
$u(t,x) = e^{i\omega t}\phi_{\omega,\gamma}(x)$.  
To analyze the linear stability of this profile, we set
\[
u(t,x) = e^{i\omega t}\bigl(\phi_{\omega,\gamma}(x) + v(t,x)\bigr),
\]
which leads to the evolution equation
\begin{equation}\label{eqLin1}
    i\partial_t v - (\partial_{xx} - \gamma \delta_0)v + \omega v 
    = |\phi_{\omega,\gamma} + v|^{p-1}(\phi_{\omega,\gamma}+v) 
    - |\phi_{\omega,\gamma}|^{p-1}\phi_{\omega,\gamma}.
\end{equation}
Extracting only the linear terms in~\eqref{eqLin1}, and decomposing $v$ into its real and imaginary parts,
one obtains
\[
i\partial_t v 
    + \mathcal{L}_{+,0}\,\Re v 
    + \mathcal{L}_{-,0}\,\Im v 
    = 0,
\]
where the self-adjoint operators $\mathcal{L}_{\pm,0}
=\mathcal{L}_{\pm,0}(\omega,\gamma,p)$ are given by
\begin{equation}\label{eq-01-b9}
    \mathcal{L}_{+,0}
    := -\partial_{xx} + \omega + \gamma \delta_0 
       - p\,\phi_{\omega,\gamma}^{\,p-1}, \qquad   \mathcal{L}_{-,0}
    := -\partial_{xx} + \omega + \gamma \delta_0 
       - \phi_{\omega,\gamma}^{\,p-1}.
\end{equation}
In what follows, we shall exploit the spectral structure of $\mathcal{L}_{+,0}$ in order to locate the
bifurcation point. We define 
\begin{equation*}
    H^{1}_{\mathrm{sym}}(\R) = \{ u \in H^1(\R)\,: u(x) = u(-x)\}.
\end{equation*}

We exploit the following spectral properties established in
\cite[Lemmas~11, 12, 17, and Appendix~B]{CoFuFi08}.
\begin{prop}\label{prpSpectr}
    There exist $\lambda_{1}=\lambda_{1}(\omega,p,\gamma)>0$, 
$\lambda_{2}=\lambda_{2}(\omega,p,\gamma)>0$, and 
$C_{\omega}=C(\omega)>0$ such that:
\begin{enumerate}[label=\alph*)]

\item 
\[
    \sigma\!\left(
        \mathcal{L}_{+,0}\big|_{H^{1}_{\mathrm{sym}}(\R)}
    \right)\footnote{ \ Here the restriction to $H^{1}_{\mathrm{sym}}(\R)$ is intended 
at the level of the quadratic form associated to $\mathcal{L}_{+,0}$.}
    \subset 
    \{-\lambda_{1}\}
    \cup [C_{\omega},\infty).
\]

\item 
If $\gamma \le 0$, then
\begin{equation}\label{eq36-bi}
    \sigma(\mathcal{L}_{+,0})
    \subset 
    \{-\lambda_{1}\}
    \cup [C_{\omega},\infty).
\end{equation}

\item 
If $\gamma>0$, then
\begin{equation}\label{eq35-bi}
    \sigma(\mathcal{L}_{+,0})
    \subset 
    \{-\lambda_{1}\}
    \cup \{-\lambda_{2}\}
    \cup [C_{\omega},\infty).
\end{equation}
\end{enumerate}
\end{prop}
Let $\chi_{\gamma}=
\chi_{\gamma}(\omega,p,\gamma)$ be such that
\begin{equation}
    \label{eqFirstEig}
    \mathcal{L}_{+,0} \chi_{\gamma} 
    = -\lambda_{1}\chi_\gamma, \quad \|\chi_\gamma\|_{L^{2}(\mathbb{R})} =1.
\end{equation}

\section{Strichartz Estimates and Well-Posedness}\label{secWP}
In this section, we prove the well-posedness of the Cauchy problem \eqref{eqDeltaStrip}. The main idea is to apply the contraction principle to Duhamel's formulation 
\begin{equation}
    \label{eqDuhamel}
    u(t) = e^{i\H t} u_0 + i \int_0^t e^{i(t-s)\H} |u(s)|^{p-1} u(s) \, dx
\end{equation}
in the spirit of Kato's method introduced in \cite{Ka87}. In order to set up the contraction principle, we first obtain suitable Strichartz estimates for our setting. We remark that our setting has a distinguished characteristic of being both a product space and having the operator $\H$ as a singular, finite-rank perturbation of the Laplacian in $\R$. 
We introduce the Strichartz admissible pairs: $(q,r)$ are said to be admissible if $4\leq q \leq \infty$ and 
\begin{equation}
    \label{eqStrichAdm}
    \frac{2}{q} + \frac{1}{r} = \frac{1}{2}.
\end{equation}
We also introduce the admissible couple $(\alpha,\beta)$ defined by
\begin{equation}
    \label{eqAlphaBeta}
    \beta= p+1, \quad \alpha = 2\left(\frac{2\beta}{\beta-2}\right).
\end{equation}

\begin{prop}\label{prpStrich1}
    For any admissible pairs $(q,r)$, $(q_1,r_1)$ there exists a constant $C = C(q,r,q_1,r_1) >0$ such that the following estimate holds:
    \begin{align}
        \label{eqStrichEst1}
       \| e^{it \H} f \|_{L^{q}_t L^r_x L^2_y} + \left\| \int_0^t e^{i(t-s) \H} F(s,x,y) \right\|_{L^{q}_t L^r_x L^2_y}  \lesssim C(\| f \|_{L^2_{x,y}} + \| F \|_{L^{q_1'}_t L^{r_1'}_x L^2_y})
    \end{align}
\end{prop}

\begin{proof}
The proof is based on the Strichartz estimates for the one-dimensional operator $\H_0 = -\partial_{xx} + \gamma \delta_0$ where $\delta_0$ is the Dirac delta at zero. In this case, the (local in time) Strichartz estimates are given by 
    \begin{equation}
    \begin{aligned}
        \label{eqStrich1D1}
        &\sup_{m \in \R }\| e^{it (\H_0 + m)} f \|_{L^{q}_t L^r_x} \leq C\| f \|_{L^2_{x}} \\ 
        &\sup_{m \in \R }\left\| \int_0^t e^{i(t-s) (\H_0 + m)} F ds \right\|_{L^{q}_t L^r_x} 
       \leq C \| F \|_{L^{q_1'}_t L^{r_1'}_x}
    \end{aligned}
    \end{equation}
    for any $(q,r)$, $(p_1,q_1)$ admissible pairs, see \cite[Proposition $2.3$]{AdCaFiNo11}. Indeed, it is well known that this operator satisfies the same dispersive estimates as the free Laplacian \cite{AdSa05}. Moreover, the spectral parameter $m$ will appear naturally after a Fourier decomposition in $y$ in what follows.
    Now we consider 
    \begin{equation*}
        u(t,x,y) = e^{i t \H} f + \int_0^t  e^{i(t-s)\H} F(s,x,y) ds
    \end{equation*}
    which is the Duhamel's formulation of
    \begin{equation}\label{eqNLSF1}
        i \partial_t u + \H u = F, \quad u(0,x,y) = f(x,y)
    \end{equation}
    We decompose $u,f$, and $F$ into the Fourier modes in $y$. We introduce the orthonormal basis for $L^{2}[0,L]$ with Neumann boundary conditions
    \begin{equation}
    \label{eq:basis}
\begin{cases}
\lambda_k=\left(\frac{k\pi}{L}\right)^2,\quad \theta_k(y)=\sqrt{\frac2L}\cos\left(\frac{k\pi}{L}y\right),\quad k\geq 1,\\
\lambda_0=0,\quad \theta_0(y)=\sqrt{\frac1L},\quad k=0.
\end{cases}
\end{equation}
and we write
\begin{equation}\label{eqDecoU1}
     u(t,x,y) = \sum_{n \geq 0} u_n(t,x) \theta_n(y), \quad   F(t,x,y) = \sum_{n \geq 0} F_n(t,x) \theta_n(y), \quad   f(x,y) = \sum_{n \geq 0} f_n(x) \theta_n(y)
\end{equation}
    where 
    \begin{align*}
        &u_n(t,x) = \sqrt{\frac{2}{L}} \int_0^L u(t,x,y) \theta_n(y) dy, \quad
         F_n(t,x) = \sqrt{\frac{2}{L}} \int_0^L F(t,x,y) \theta_n(y) dy,\\
          &f_n(x) = \sqrt{\frac{2}{L}} \int_0^L f(x,y) \theta_n(y) dy.
    \end{align*} 
    We remark that this decomposition allows us to write the Hamiltonian operator as a decomposition
    \begin{equation*}
        \H u = \bigoplus_{n\geq0} \H_n u_n \otimes \theta_n(y), 
    \end{equation*}
    where
    \begin{equation*}
        \H_n = - \partial_{xx} + \gamma  \delta_0 + \lambda_n
    \end{equation*}
    is the one-dimensional Laplacian operator with a Dirac delta at zero and $\lambda_n = (n\pi L^{-1})^2$. Indeed, it is enough to observe that the quadratic form associated with $\H$ is diagonalizable in the following way: since
    \begin{equation*}
        \H u = \sum_{n\geq 0} ( -\partial_{xx} + \gamma \delta_0(x) + \lambda_n) u_n(x) \theta_n(y),
    \end{equation*}
    the quadratic form associated to $\H$ can be seen as 
    \begin{equation*}
        \langle \H u, u\rangle = \| \nabla u\|_{L^2}^2 + \gamma \sum_{n \geq 0} |u_n(0)|^2 = \sum_{n \geq 0} \left(\| \partial_x u_n(x)\|_{L^2}^2 + \lambda_n  \| u_n(x)\|_{L^2}^2 + \gamma |u_n(0)|^2\right)
    \end{equation*}
    due to the orthogonality of $\theta_n$.
    Thus equation \eqref{eqNLSF1} is reduced to 
    \begin{equation}
        \label{eqNLSNodes1}
         i \partial_t u_n + \H_n  u_n = F_n, \quad u_n(0,x,y) = f_n(x,y)
    \end{equation}
Then, by \eqref{eqStrich1D1}, we obtain 
that 
\begin{equation}
\begin{aligned}
    \label{eqStrich1D11}
         \| e^{it \H_n} f_n \|_{L^{q}_t L^r_x}\leq  C\| f_n \|_{L^2_{x}},  \quad 
        \left\| \int_0^t e^{i(t-s) \H_n} F_n ds \right\|_{L^{q}_t L^r_x}\leq  C \| F_n \|_{L^{q_1'}_t L^{r_1'}_x}
\end{aligned}   
\end{equation}
where the constant $C>0$ does not depend on $n$. Summing over all $n$, we get 
\begin{equation*}
    \begin{aligned}
        \| e^{it \H_n} f_n \|_{\ell^2_n L^{q}_t L^r_x} + \left\| \int_0^t e^{i(t-s) \H_n} F_n ds \right\|_{\ell^2_n L^{q}_t L^r_x} 
       \leq  C\left(\| f_n \|_{\ell^2_n L^2_{x}} + \| F_n \|_{\ell^2_n L^{q_1'}_t L^{r_1'}_x}\right)
    \end{aligned}
\end{equation*}
Since $p,q \geq 2$ and $\max(p_1',q_1') \leq 2 \leq \min(q,r)$, we can apply the Minkowski inequality to get
\begin{equation*}
    \begin{aligned}
        \| e^{it \H_n} f_n \|_{ L^{q}_t L^r_x\ell^2_n} + \left\| \int_0^t e^{i(t-s) \H_n} F_n ds \right\|_{ L^{q}_t L^r_x\ell^2_n} 
       \leq  C\left(\| f_n \|_{ L^2_{x}\ell^2_n} + \| F_n \|_{ L^{q_1'}_t L^{r_1'}_x \ell^2_n}\right).
    \end{aligned}
\end{equation*}
The conclusion in \eqref{eqStrichEst1} follows from the Plancherel inequality.
\end{proof}

\begin{prop}\label{prpStrich2}
    We have 
    \begin{equation}
        \label{eqStrichEst2}
\left\|  e^{it\H} f \right\|_{L^p_t W^{1,q}_x L^2_y}
+ \left\|  \int_0^t e^{i(t-s)\H} F(s) \, ds \right\|_{L^p_t W^{1,q}_x L^2_y}
\leq C \left( \|  f \|_{H^1_x L^2_y} + \| F \|_{L^{p'}_t W^{1,q'}_x L^2_y} \right).
    \end{equation}
\end{prop}

\begin{proof}
To overcome the fact that $\partial_x$ and $e^{-it\H}$ do not commute, we introduce the square roots of the operators $\H_n$ defined in the proof of Proposition \ref{prpStrich1}. We show it only for the operator $\H_{0} = -\partial_{xx} + \gamma \delta_0$, as for the other it is equivalent. Let $a > \frac{\gamma^2}{4}$ be any fixed constant. Then the operator $\H_0 + a$ is positive and  $(\H_0 + \alpha)^{1/2}$ with domain given by the domain of the quadratic form associated to $\H_0$,  is well defined, positive and self-adjoint, see \cite{Be68}. In particular, we have
\begin{align*}
    \left( (a + \H_{0})^\frac{1}{2} \right)^2 = a + \H_0, \quad
    D((a + \H_{0})^\frac{1}{2}) = D(Q_{\H_0}) = H^1(\R), 
\end{align*}
where $Q_{\H_0}$ is the quadratic form associated with $\H_0$. This operator is positive and self-adjoint. We define the associated norm as 
\begin{equation*}
    \| \phi \|_{H^1_a} := \| (a + \H_{0})^\frac{1}{2} \phi\|_{L^2}
\end{equation*}
which implies that
\begin{equation*}
   \| \phi\|_{H^1}^2 \lesssim \| \phi \|_{H^1_a}^2 = a \| \phi \|_{L^2}^2 + Q_{\H_0}(\phi) \lesssim \| \phi\|_{H^1}^2.
\end{equation*}
In particular, the norm $\|\cdot\|_{H^1_a}$ and 
$\| \cdot \|_{H^1}$ are equivalent.

Notice that the operators $e^{it\H_0}$ and $(a + \H_0)$ commute. Furthermore, $e^{it\H_0}$ is bounded, hence it is closed. Thus,  the operators $e^{it\H_0}$ and $(a + \H_0)^{\frac{1}{2}}$ also commute. Since $e^{it\H_0}$ is unitary, it follows that
\begin{equation*}
    \| e^{it\H_0} \phi \|_{H^1_a} = \| \phi\|_{H^1_a},
\end{equation*}
which yields 
\begin{equation*}
    \| e^{it\H_0} \phi \|_{H^1} \lesssim \| \phi \|_{H^1}.
\end{equation*}
In the same way, by \eqref{eqStrich1D1}, we get 
\begin{equation*}
    \left\| \int_0^t e^{i(t-s) (\H_0 + a)} F ds \right\|_{L^{q}_t W^{1,q}_x}\lesssim \| (a + \H_0) F \|_{L^{q_1'}_t L^{r_1'}_x}  \lesssim  \| F \|_{L^{q_1'}_t W^{1,q_1'}_x} 
\end{equation*}
To conclude the proof, we first observe that the same reasoning applies for any $\H_n$ and the shift $a$ does not depend on $n$ as $\lambda_n$ are increasing. Thus, we have the same properties for every $n$, and we can sum on all $n$ in the same way as we have done in the last part of the proof of Proposition \ref{prpStrich1}.
\end{proof}

We are now in a position to show the local existence for \eqref{eqDeltaStrip}. 

\begin{proof}[Proof of Theorem \ref{thmLWP}]
    We define the map 
    \begin{equation*}
        \Theta_{u_0} (u) = e^{it\H} u_0 + i\int_0^t e^{i(t-s)\H} |u(s)|^{p-1}u(s) ds.
    \end{equation*}
    Let $(\alpha,\beta)$ be the Strichartz admissible pair defined in \eqref{eqAlphaBeta}.
    We show that for all $u_0 \in H^1(\Strip_L)$, there exists $T = T(u_0)$, $R = R(u_0) >0$ such that $\Theta_{u_0}$ maps a ball $B_X(0,R)$ of the space
    \begin{equation*}
        X= L^\alpha((-T,T);L^\beta_x H^1_y) \cap L^\alpha ((-T,T);W^{1,\beta}_x L^2_y)
    \end{equation*}
    into itself. We first estimate the nonlinear term.  By H\"older inequality and the embedding $H^1_y \subset L^\infty_y$, we get 
    \begin{equation}
        \| |u|^{p-1}u\|_{L^{\alpha'}_t L^{\beta'}_x H^1_y} \leq \left\|  \| u \|^{p-1}_{L^\infty_y} \| u\|_{H^1_y}\right\|_{L^{\alpha'}_t L^{\beta'}_x} \lesssim   \| u\|_{L^{\alpha}_t L^{\beta}_x H^1_y}   \| u\|_{L^{(p-1)\tilde{\alpha}}_t L^{(p-1)\tilde{\beta}}_x H^1_y}^{p-1}, 
    \end{equation}
    where 
    \begin{equation*}
        \frac{1}{\tilde{\alpha}} + \frac{1}{\alpha} = 1 - \frac{1}{\alpha}, \quad  \frac{1}{\tilde{\beta}} + \frac{1}{\beta} = 1 - \frac{1}{\beta}.
    \end{equation*}
    By a direct computation, one can check that $(p-1)\tilde{\beta} = \beta$ and $(p-1) \tilde{\alpha} < \alpha$. Thus, by \eqref{eqStrichEst1} H\"older inequality in time, we see that there exists $a = a(p) >0$ such that
    \begin{equation}
        \label{eqLoc1}
        \|    \Theta_{u_0} (u) \|_{L^\alpha((-T,T); L^\beta_x H^1_y)} \lesssim \| u_0 \|_{L^2_xH^1_y} + T^{a} \| u \|_{L^\alpha((-T,T); L^\beta_x H^1_y)}^{p}.
    \end{equation}
    Furthermore, arguing in the same way, we have 
    \begin{equation*}
         \| \partial_x (|u|^{p-1} u)\|_{L^{\alpha'}_t L^{\beta'}_x L^2_y} \lesssim   \| \partial_x u\|_{L^{\alpha}_t L^{\beta}_x L^2_y}   \| u\|_{L^{(p-1)\tilde{\alpha}}_t L^{(p-1)\tilde{\beta}}_x H^1_y}^{p-1}.
    \end{equation*}
    Then by \eqref{eqStrichEst2} we have 
    \begin{equation}
         \label{eqLoc2}
        \|    \Theta_{u_0} (u) \|_{L^\alpha((-T,T); W^{1,\beta}_x H^1_y)} \lesssim \| u_0 \|_{H^1} + T^{a}  \|    u \|_{L^\alpha(-T,T), W^{1,\beta}_x H^1_y}  \| u\|_{L^{(p-1)\tilde{\alpha}}((-T,T); L^{(p-1)\tilde{\beta}}_x H^1_y)}^{p-1}.
    \end{equation}
    Thus, we obtain the claim by choosing $T>0$ small enough. 
    
    Next, we show that there exists $0 <T' \leq T$ such that $ \Theta_{u_0}$ is a contraction with respect to the norm $\| \cdot \|_{L^\alpha(-T,T), L^\beta_x L^2_y}$. Let $v_1,v_2 \in B_X(0,R)$ with $R$ given by the previous step. Then we have
    \begin{align*}
\left\| |v_1|^{p-1} v_1 - |v_2|^{p-1} v_2 \right\|_{L^{\alpha'}((-T,T);L^{\beta'}_x L^2_y)} &\lesssim \left\| \left\| v_1 - v_2 \right\|_{L^2_y} \left( \| v_1 \|_{L^\infty_y} + \| v_2 \|_{L^\infty_y} \right)^{p-1}\right\|_{L^{\alpha'}((-T,T);L^{\beta'}_x)} \\
&\lesssim T^{a} R \| v_1 - v_2 
\|_{L^{\alpha'}((-
T,T);L^{\beta'}_xL^2_y)}, 
    \end{align*}
where we have used the H\"older inequality and the Sobolev embedding in the same way as the previous step.

Finally, as $X \subset L^\alpha( (-T, T); L^\beta_x L^2_y)$, it is standard to see that it is complete with respect to the norm $\| \cdot \|_{L^\alpha(-T, T), L^\beta_x L^2_y}$. Thus, by the use of the Banach fixed point argument, there exists a fixed point of $\Theta_{u_0}$, and it is classical to see that it is indeed a $C((-T, T), H^1)$ solution to \eqref{eqDeltaStrip}.
\end{proof}

\section{Pitchfork Bifurcation at the Critical Length }\label{secBifurcation}
The scope of this section is to show that the line soliton $\phi_{\omega,\gamma}$ defined in \eqref{eqGS1D} undergoes a pitchfork bifurcation at a transverse length in which it becomes unstable. Let 
\[
    \phi_{\omega,\gamma}^{(L)}(x,y)
    := \phi_{\omega,\gamma}(x),
    \qquad (x,y)\in\mathbb{S}_{L}.
\] 
be the line soliton trivially extended in the $y$ direction to a function on $\mathbb{S}_{L}$. 
In this section, we show that it undergoes a pitchfork bifurcation at the length
\begin{equation}\label{eqLCritic}
    L_{*} := \frac{\pi}{\sqrt{\lambda_{1}}} >0,
\end{equation}
where $- \lambda_{1} < 0$ is the smallest eigenvalue associated with $\mathcal{L}_{+,0}$ defined in Proposition \ref{prpSpectr}.

After a scaling in the transverse variable $y \mapsto y/L$, \eqref{eqStationary2D} becomes
\begin{equation}
\label{eqS2D-2}
\begin{cases}
H_{L} u + \omega u - |u|^{p-1}u = 0,
& (x,y)\in \mathbb{S}, \\[2mm]
\partial_{\nu}u = 0, 
& \text{on }\partial\mathbb{S}=\mathbb{R}\times\{0,1\},
\end{cases}
\end{equation}
where $\Strip = \R \times [0,1]$ and  \( H_{L,\gamma}
:= -\partial_{xx} - L^{-2}\partial_{yy} + \gamma\delta_{0}(x). \)
Set \(
    \xi_\gamma(x,y) := \sqrt{2} \cos (\pi y) \chi_\gamma(x), 
\)
where $\chi_\gamma(x)$ is defined in \eqref{eqFirstEig}. Since $\gamma$ is fixed throughout this section, we suppress the subscript $\gamma$ in what follows. In these scaled variables,  the reference transverse length is unitary and independent of $L$ and thus, in the following, we use $\phi_{\omega,\gamma}$ and $\phi_{\omega,\gamma}^{(1)}$ interchangeably. Theorem \ref{thm1-bi} follows directly from the following proposition.

\begin{prop}\label{ex-bi}
Let $p \ge 2$. There exist $a_0>0$, a function  \( L \in C^2((-a_0,a_0);\mathbb R^+) \), with $L(0) = L_*$ and a family of positive  
solutions 
\(\varphi(a)\in \mathcal{D}(H_{L})\) of \eqref{eqS2D-2} with $L=L(a)$ for each $a\in(-a_0,a_0)$,  such that the map \( a \mapsto \varphi(a) \)
is of class $C^2$  and there exists $h \in C^2((-a_0,a_0);\mathcal{D}(H_{L})) $ such that for every $a\in(-a_0,a_0)$,
\begin{equation} \label{eq1-varphi}
\varphi(a) = \phi_{\omega,\gamma} 
+ a \xi + h(a),
\end{equation}
where 
\begin{equation}
     \label{eqH2Normh}
     \|h(a)\|_{H^2(\mathbb{S})} = O(a^2)\ \text{as}\ a\to 0.
\end{equation}
\end{prop}

To prove Proposition~\ref{ex-bi}, we employ 
the Lyapunov-Schmidt reduction.
We define 
\begin{equation} \label{eq51-bi}
\mathcal{F} : \mathbb{R}^{+} \times \mathcal{D}(H_{L}) \to L^{2}(\mathbb{S}), 
\qquad 
\mathcal{F}(L,u) := H_{L}u + \omega u - |u|^{p-1} u.
\end{equation}
It is clear that for each $L>0$ and $p>2$, 
$\mathcal{F}(L, \cdot)$ is $C^{2}$.

For each $n\in \mathbb{N}$, set $\theta_0^{(1)}(y) :=1$ and $\theta_n^{(1)}(y) := \sqrt{2}\,\cos(n\pi y), \ \forall n \in \N^*.$
Then any $f \in H^{1}(\mathbb{S})$ admits the expansion
\[
f(x,y) = \sum_{n=0}^{\infty} f_n(x) \theta_n^{(1)}(y),
\qquad 
f_n(x) := \int_0^1 f(x,y) \theta_n^{(1)}(y)\, dy.
\]
It follows that 
\begin{equation}\label{eq6-bi}
D_u \mathcal{F}\bigl(L, \phi_{\omega,\gamma}\bigr)f 
= \sum_{n=0}^{\infty} \biggl(\mathcal{L}_{+,0} + \Bigl(\frac{n\pi}{L}\Bigr)^2 \biggr) f_n(x)\, \theta_n^{(1)}(y).
\end{equation}
From \eqref{eq36-bi}, \eqref{eq6-bi} and 
$\xi(x, y) = 
\sqrt{2} \cos (\pi y) \chi(x) = \theta_{1}(y) \chi(x)$, 
we obtain that for any $L >0$,
\begin{equation}
   \label{eq1-bi}
  D_u \mathcal{F}\bigl(L, \phi_{\omega,\gamma}\bigr) \xi = \left( - \lambda_1 + \frac{\pi^2}{L^2} \right) \xi \implies \mathrm{Ker}\, D_u \mathcal{F}\bigl(L_{*}, \phi_{\omega,\gamma}\bigr) 
= \mathrm{Span}\{\xi\}.
\end{equation}
Notice that $\mathrm{Ker}\, D_u \mathcal{F}\bigl(L_{*}, \phi_{\omega,\gamma}\bigr) $ is simple. 
We define the function spaces
\begin{align*}
L^2_{\text{ort}} := 
\Bigl\{ u \in L^{2}(\mathbb{S}) \colon 
\langle u, \xi \rangle = 0 \Bigr\}, \quad H^2_{\text{ort}} :=  \mathcal{D}(H_{L}) \cap L^2_{\text{ort}}
\end{align*}
Notice that $\phi_{\omega,\gamma} \in H^2_{\text{ort}}$.  We seek a solution to $\mathcal{F}(L,u)=0$ of the form
\begin{equation}\label{eqSolForm}
    (L,u) = \bigl(L_{*} 
    \pm \delta,\, \phi_{\omega,\gamma} + a \xi + h\bigr),
\end{equation}
with $\delta>0$, $a\in \mathbb{R}$, and $h \in H^2_{\text{ort}}$. 

Let $P_{\perp}:L^2(\mathbb{S}) \to L^2_{\text{ort}}$ denote the orthogonal projection
\[
P_{\perp} u := u - \langle u, \xi \rangle \,\xi,
\]
and let
\[
\mathcal{F}_{\perp} : \R \times \mathbb{R}^{+}  \times H^2_{\text{ort}} \to L^{2}(\mathbb{S}), \qquad \mathcal{F}_{\perp}(a,L,h) := P_{\perp} \mathcal{F}\bigl(L, \phi_{\omega,\gamma} + a \xi + h\bigr).
\]

We now consider the \textit{auxiliary problem}, that is, we show that there exists
\(
h:(-a_0,a_0) \times \mathbb{R}^+ \to H^2_{\text{ort}}, \)
such that for any $(a,L)$ in a sufficiently small neighborhood of $(0,L^*)$ we have
\begin{equation}
    \label{eqAuxProb}
    \mathcal{F}_{\perp}(a,L,h(a,L)) = 0.
\end{equation}

Observe that for any $L >0$, 
\begin{equation} \label{eq4-bi}
\mathcal{F}_{\perp} (0, L, 0) 
= P_{\perp} \mathcal{F}(L, \phi_{\omega, \gamma}) = 0.
\end{equation}
By \eqref{eq51-bi}, the Fr\'echet derivative with respect to $h$ reads
\[
D_{h} \mathcal{F}_{\perp}(a, L, h) = H_L + \omega - p \bigl| \phi_{\omega, \gamma} + a \xi + h \bigr|^{p-1},
\]
and therefore, as an operator restricted to $H^2_{\mathrm{ort}}$, 
\begin{equation} \label{eq2-bi}
\mathrm{Ker}\; D_{h} \mathcal{F}_{\perp}(0, L_{*}, 0)|_{H^2_{\mathrm{ort}}}
= \mathrm{Ker}\; P_{\perp} D_{u} \mathcal{F}(L_{*}, \phi_{\omega, \gamma})|_{H^2_{\mathrm{ort}}} = \{0\}.
\end{equation}
It follows from the Fredholm alternative (see Appendix \ref{AppFredholm}) that
\begin{equation} \label{eq3-bi}
\mathrm{Ran}\; D_{h} \mathcal{F}_{\perp}(0, L_{*}, 0)
= \mathrm{Ran}\; D_{u} \mathcal{F}(L_{*}, \phi_{\omega, \gamma})|_{H^2_{\mathrm{ort}}} = L^2_{\mathrm{ort}}.
\end{equation}
From \eqref{eq2-bi} and \eqref{eq3-bi}, we conclude that
\(
D_{h} \mathcal{F}_{\perp}(0, L_{*}, 0) \colon H^2_{\mathrm{ort}} \to L^2_{\mathrm{ort}}
\)
is bijective. 

Applying the implicit function theorem, we then obtain the following result.

\begin{lem} \label{lem1-bi}
Let $p \geq 2$. Then there exist some constants $a_{0}, \delta_{0}, r_{0} > 0$ and a $C^{1}$-map
\[
h: (-a_{0}, a_{0}) \times (L_* - \delta_0, L_* + \delta_0) \to H^2_{\mathrm{ort}} \cap \left\{ u \in H^{2}(\mathbb{S}) \colon \|u\|_{H^{2}(\mathbb{S})} < r_{0} \right\}
\]
such that 
\[
\mathcal{F}_{\perp}(a, L, h) = 0 \quad \iff \quad h = h(a, L).
\]
Moreover, if $p >2$ then $h$ is $C^2$.
\end{lem}
It follows from \eqref{eq4-bi} and Lemma \ref{lem1-bi} that, for any $L \in (L_{*} - \delta_{0}, L_{*} + \delta_{0})$, 
\begin{equation} \label{eq7-bi}
h(0, L) = 0.
\end{equation}
Thus, we define
\begin{equation} \label{eq31-bi}
\varphi(a, L) := \phi_{\omega, \gamma} + a \, \xi + h(a, L).
\end{equation}
Now we compute the derivatives of $ h$. This is done exploiting the invertibility of  $D_h \mathcal{F}_\perp(L,a,h)  =  D_{u} \mathcal{F} (L, \varphi(a, L))$ for $(a,L)$ in a neighborhood of $(0,L_*)$, which follows by continuity and Lemma \ref{lem1-bi} (up to restricting $a_0$ and $\delta_0$ if needed). 

\begin{lem}[First derivatives of $h$] \label{lem3-bi}
For any $p \geq 2$, we have
\begin{equation} \label{eq14-bi}
\frac{\partial h}{\partial L}(0, L_*) 
= \frac{\partial h}{\partial a}(0, L_*) 
= 0.
\end{equation}
\end{lem}

\begin{proof}
    Observe that $\mathcal{F}_\perp(L,a,h(a, L)) = 0$ implies 
    \begin{equation} \label{eq24-bi}
    \begin{split}
        0 = 
        \frac{\partial 
        \mathcal{F}_\perp}{\partial L}(L,a,h(a, L)) & = D_h \mathcal{F}_\perp(L,a,h) \frac{\partial h}{\partial L}
        (a, L) + \frac{\partial 
        \mathcal{F}_\perp}{\partial L}(L,a,h) \\
         & =   
        P_{\perp} D_u
        \mathcal{F}(L, \varphi(a, L)) 
         \frac{\partial h}{\partial L}(a, L) 
        + \frac{2}{L^{3}} 
        P_{\perp} \partial_{yy} 
        \varphi(a, L), 
    \end{split}
    \end{equation}
    \begin{equation} \label{eq23-bi}
    \begin{split}
        0 &= \frac{\partial}{\partial a} 
        \mathcal{F}_\perp (L,a,h) = \frac{\partial}{\partial a} 
        P_\perp \mathcal{F}(L,\varphi(a,L))  = D_u 
        \mathcal{F}_\perp(L,a,h)
        \frac{\partial \varphi}
        {\partial a}(a, L)  \\
        &  =   
        P_{\perp} D_{u}
        \mathcal{F}(L, \varphi(a, L)) 
        \frac{\partial h}{\partial a}
        (a, L) + P_{\perp} D_{u}
        \mathcal{F}
        (L, \varphi(a, L)) \xi.
    \end{split}
    \end{equation}
 From \eqref{eq7-bi}, $\partial_{yy} 
 \phi_{\omega, \gamma} = 0$ and 
 $(\mathcal{L}_{+, 0}- \frac{1}{L_{*}^2} 
 \partial_{yy})\xi = 0$ (see \eqref{eq1-bi}), 
 it follows that
    \begin{align*}
        &\frac{\partial h}{\partial L}(0, L_*)  = - \frac{2}{L_{*}^{3}}
    (P_{\perp} D_{u} \mathcal{F} (L_*, \phi_{\omega, \gamma})|_{H^2_{\text{ort}}})^{-1} 
    P_{\perp} \partial_{yy} 
    \phi_{\omega, \gamma}
    = 0,  \\
    &\frac{\partial h}{\partial a}(0, L_*)  = - (P_{\perp} 
    D_{u} \mathcal{F}(L_*, 
    \phi_{\omega, \gamma})
    |_{H^2_{\text{ort}}})^{-1} P_{\perp} 
    (\mathcal{L}_{+, 0} - 
    \frac{1}{L_{*}^2} 
    \partial_{yy})
    \xi = 0. 
    \end{align*}
\end{proof}
\begin{lem}[Second derivatives of 
$h$]\label{lem4-bi}
For any $p >2$,  we obtain 
    \begin{equation}\label{eq41-bi}
    \frac{\partial^{2} h}
    {\partial L^{2}}(0, L_*) 
    = \frac{\partial^{2} h}
    {\partial L \partial a}
    (0, L_*) = 0 
    \end{equation}
    and 
    
     \begin{equation} \label{eq26-bi}
     \begin{split}
    \frac{\partial^{2} h}
    {\partial a^{2}}(0, L_*) 
    & = p (p-1) 
     (P_{\perp} D_{u} 
    F(L, \phi_{\omega, 
    \gamma})_{|H^2_{\text{ort}}})^{-1} 
    \phi_{\omega, \gamma}^{p-2} \xi^2 \\
    & = p (p-1) (P_{\perp} (\mathcal{L}_{+, 0}
    - \frac{1}{L_{*}^2} \partial_{yy}) |_{H^2_{\text{ort}}})^{-1} 
    \phi_{\omega, \gamma}^{p-2} 
    \xi^2. 
    \end{split}
    \end{equation}   
    
\end{lem}
\begin{proof}
Differentiating \eqref{eq24-bi} with 
respect to $L$, we obtain 
    \[
    \begin{split}
0 & = P_{\perp} D_{u}^2
        F(L, \varphi(a, L))
        \left(\frac{\partial h}{\partial 
        L}(a, L)\right)^2 
        +  P_{\perp} D_{u}
        F(L, \varphi(a, L)) \frac{\partial^2 h}
        {\partial L^2} (a, L) 
        + 
        \frac{2}{L^{3}} P_{\perp} 
        \partial_{yy} 
        \frac{\partial h}{\partial L}
        (L, a) \\
        & \quad - \frac{6}{L^{4}} 
        P_{\perp} \partial_{yy} 
        \varphi(a, L)
        \\
        & = 
        - p(p-1) P_{\perp}
        |\varphi(a, L)|^{p-2}
        \left(\frac{\partial h}{\partial 
        L}(a, L)\right)^2 
        +  P_{\perp} D_{u}
        F(L, \varphi(a, L)) \frac{\partial^2 h}
        {\partial L^2}(a, L)
       + 
       \frac{2}{L^{3}} P_{\perp} 
        \partial_{yy} 
        \frac{\partial h}{\partial L}
        (L, a)
        \\
        & \quad - \frac{6}{L^{4}} 
        P_{\perp}  
        \partial_{yy} \varphi(a, L).  
    \end{split}
    \] 
Similarly, 
differentiating \eqref{eq24-bi} with 
respect to $a$, we obtain 
    \[
    \begin{split}
0 & = 
P_{\perp} D^2_{u}
        F(L, \varphi(a, L)) 
       \xi
        \frac{\partial h}{\partial L}(a, L) 
        + 
P_{\perp} D^2_{u}
        F(L, \varphi(a, L)) 
        \frac{\partial h}{\partial a}(a, L) 
        \frac{\partial h}{\partial L}(a, L)
         \\
        & \quad 
        +  P_{\perp} D_{u}
        F(L, \varphi(a, L)) \frac{\partial^2 h}
        {\partial L \partial a}
        (a, L)
        + \frac{2}{L^3} 
        P_{\perp} 
        \frac{\partial \varphi_{yy}}
        {\partial a} (a, L) 
        \\
        & = 
       P_{\perp} D^2_{u}
        F(L, \varphi(a, L)) 
       \xi
        \frac{\partial h}{\partial L}(a, L)
        -p (p-1) P_{\perp} 
        |\varphi(a, L)|^{p-2} 
        \frac{\partial h}{\partial a}(a, L) 
        \frac{\partial h}{\partial L}(a, L)
          \\
        & \quad 
        + P_{\perp} D_{u}
        F(L, \varphi(a, L)) \frac{\partial^2 h}
        {\partial L \partial a} 
        (a, L)
        + \frac{2}{L^{3}} 
        P_{\perp}  
        (- \xi + \frac{\partial h_{yy}}{\partial 
        a}(a, L)) \\
        & = 
        P_{\perp} D^2_{u}
        F(L, \varphi(a, L)) 
       \xi
        \frac{\partial h}{\partial L}(a, L)
        -p (p-1) P_{\perp} 
        |\varphi(a, L)|^{p-2} 
        \frac{\partial h}{\partial a}(a, L) 
        \frac{\partial h}{\partial L}(a, L) \\
         & \quad + P_{\perp} D_{u}
        F(L, \varphi(a, L)) \frac{\partial^2 h}
        {\partial L \partial a} 
        (a, L)
    \end{split}
    \]
Here we have used that $P_{\perp} \xi 
= P_{\perp} \frac{\partial h_{yy}}{\partial 
        a} = 0$.  
Thus, we see from \eqref{eq14-bi} 
that \eqref{eq41-bi} holds. 
   Finally,  differentiating \eqref{eq23-bi} in $a$, we get
    
    \begin{equation}\label{eq64-bi}
    \begin{split}
    0 
    & = -p(p-1) P_{\perp} 
    |\varphi(a, L)|^{p-2} 
    (\xi + 
    \frac{\partial h}{\partial a} 
    (a, L)) \frac{\partial h}
    {\partial a} 
    (a, L)
    + P_{\perp} D_{u} 
    F(L, \varphi(a, L))|_{H^2_{\text{ort}}} 
    \frac{\partial^2 h}
    {\partial a^2}(a, L) \\
    & \quad 
    - p (p-1) P_{\perp} 
    |\varphi(a, L)|^{p-2} 
    (\xi 
    + \frac{\partial h}{\partial a} 
    (a, L))\xi \\
    & = P_{\perp} D_{u} 
    F(L, \varphi(a, L))_{|H^2_{\text{ort}}} 
    \frac{\partial^2 h}
    {\partial a^2} (a, L)
    - p (p-1) P_{\perp} 
    |\varphi(a, L)|^{p-2} \bigg(
    \frac{\partial h}{\partial a} (a, L) + \xi \bigg)^2. 
    \end{split}  
    \end{equation}
In addition, one has 
      \begin{equation} \label{eq48-bi}
    \int_{0}^{1} \cos^3 (\pi y) dy
    = 0. 
    \end{equation}
This yields that $\langle
\phi_{\omega, \gamma}^{p-2} \xi^{2}, \xi \rangle 
= 0$. Thus,     
    we see from \eqref{eq14-bi}
    and \eqref{eq64-bi} that 
    \eqref{eq26-bi} holds. 
\end{proof} 

\subsection{The bifurcation problem} \label{secBifProb}

We now introduce the \textit{bifurcation function}: we define 
\[
\mathcal{F}_{\parallel}: 
(-a_{0}, a_{0}) \times 
(L_{*} - \delta_{0}, 
L_{*} + \delta_{0}) \to \R, 
\qquad 
\mathcal{F}_{\parallel}(a, L) 
:= \langle \mathcal{F}(L, \varphi(a, L)), \xi \rangle,
\]
so that solving $\mathcal{F}(L, \phi_{\omega, \gamma} + a \xi + h) = 0$ is equivalent to the coupled system
\begin{equation}\label{eq5-bi}
\begin{cases}
\mathcal{F}_{\perp}(L, a, h) = 0, & \text{(auxiliary equation)},\\
\mathcal{F}_{\parallel}(a, L) = 0, & \text{(bifurcation equation)}.
\end{cases}    
\end{equation}
We will apply the Crandall-Rabinowitz 
Transversality Theorem~\cite{MR288640}. 
Consider
\begin{equation} \label{eq44-bi}
g(a, L) := \dfrac{\mathcal{F}_{\parallel}(a, L)}{a}.  
\end{equation}
Observe from \eqref{eq31-bi} that for any $L >0$, 
\[
\mathcal{F}_{\parallel}(0, L) 
= \langle \mathcal{F}(L, \varphi(0, L)), \xi \rangle
= \langle \mathcal{F}(L, \phi_{\omega, \gamma}), \xi \rangle
= 0.
\]
This together with the l'Hopital rule 
implies that $g(0, L) = 
\dfrac{\partial \mathcal{F}_{\parallel}}{\partial a}(0, L)$. 
For $a \neq 0$ we have
\[
\mathcal{F}_{\parallel}(a, L) = 0 \quad \Longleftrightarrow \quad g(a, L) = 0.
\]

We show the following.
    \begin{prop}\label{prop1-bi}
Assume that $p \geq 2$. Then
\begin{enumerate}
    \item[\textrm{(i)}] $g(0, L_*) = 0$, 
    \item[\textrm{(ii)}] $g$ is $C^{1}$ 
    in $(-a_{0}, a_{0}) \times (L_* - \delta_0, L_* + \delta_0)$ 
    and 
    \begin{equation}
        \label{eqIFTCond}
        \frac{\partial g}{\partial L}(0, L_*) 
     < 0, \qquad 
    \frac{\partial g}{\partial a}(0, L_*) = 0.
    \end{equation}
\end{enumerate}   
\end{prop}

We now give the proof of 
Proposition \ref{prop1-bi}.  
To this end, we define 
    \begin{equation} \label{eq13-bi}
    \mathcal{L}_{+}(a, L) 
    := D_{u} \mathcal{F}(L, 
    \varphi(a, L)) 
    \end{equation}   
For any $L >0$, let $\lambda_{2}(a, L)$ 
be the second eigenvalue of 
$\mathcal{L}_{+}(a, L)$ and
$\phi_{2}(a, L)$ be the 
corresponding eigenfunction 
with $\|\phi_{2}(a, L)\|_{L^{2}} = 1$. 
Thus, we have 
    \begin{equation} \label{eq45-bi}
    \mathcal{L}_+(a, L) \phi_{2}(a, L) = 
    \lambda_{2}(a, L) \phi_{2}(a, L). 
    \end{equation}
  
  We check the transversality condition of Crandall-Rabinowitz~\cite{MR288640}, that is, whether the kernel of the linearized operator crosses zero transversely as $L$ is varied in a neighborhood of $L_*$.
\begin{lem}\label{lem2-bi}
Assume that $p \geq 2$.
It holds
    \begin{equation}
    \label{eq33-bi}
    \partial_L \lambda_{2}(a, L) = 
    - \frac{2}{L^{3}} \|\partial_{y} 
    \phi_{2}(a, L)\|_{L^{2}}^{2} 
    - p(p-1) \int_{\mathbb{S}}
    \varphi^{p-2}(a, L) 
    \frac{\partial h}{\partial L}(a, L) 
    |\phi_{2}(a, L)|^2 \, dxdy. 
    \end{equation}
\end{lem}

\begin{proof}
By differentiating \eqref{eq45-bi} with respect to $L$, we obtain
\begin{equation}
    \begin{aligned}\label{eq33bis-bi}
\frac{2}{L^{3}} \partial_{yy} \phi_{2}(a, L)
- p(p-1) \varphi^{p-2}(a, L) \frac{\partial \varphi}{\partial L}(a, L) \, \phi_{2}(a, L)
+ \mathcal{L}_+(a, L) \frac{\partial \phi_2}{\partial L}(a, L) \\
= \frac{\partial \lambda_{2}}{\partial L}(a, L) \, \phi_{2}(a, L) 
+ \lambda_{2}(a, L) \frac{\partial \phi_2}{\partial L}(a, L).
\end{aligned}
\end{equation}
Since $\|\phi_{2}(a, L)\|_{L^{2}} = 1$ for any $L>0$, we have the orthogonality condition $$\langle \phi_{2}(a, L), \frac{\partial \phi_{2}}{\partial L}(a, L) \rangle = 0.$$ Thus, multiplying both sides of \eqref{eq33bis-bi} by $\phi_{2}(a, L)$ and integrating over $\mathbb{S}$, and exploiting 
\(
\frac{\partial \varphi}{\partial L}(a, L) = \frac{\partial h}{\partial L}(a, L),
\)
\eqref{eq33-bi} follows.
\end{proof}

Let us define
\begin{equation} \label{eq:lambda*}
\lambda'_{*} := \lim_{L \to L_{*}, \, a \to 0} \frac{\partial \lambda_{2}}{\partial L}(a, L).
\end{equation}
Notice that $\phi_{2}(0, L_*) = \xi$. Using 
this, \eqref{eq14-bi}, \eqref{eq33-bi} and 
$\|\chi\|_{L^{2}(\R)} = 1$, we obtain
\begin{equation}
\label{eq34-bi}
\lambda'_{*} 
= - 2 L_{*}^{-3} \|\partial_{y} \phi_{2}(0, L_*)\|_{L^{2}}^{2} 
= - 4 L_{*}^{-3} \pi^2 \int_{\R} \chi^2(x)\, dx \int_{0}^{1} \sin^2(\pi y) \, dy 
= - 2 L_{*}^{-3} \pi^2 < 0.
\end{equation}  
We are now in a position to prove Proposition \ref{prop1-bi}. 
\begin{proof}[Proof of Proposition \ref{prop1-bi}] 
\medskip

\noindent\textrm{(i)} 
By \eqref{eq1-bi}, \eqref{eq7-bi} and 
\eqref{eq31-bi}, we have
\[
D_{u} \mathcal{F}(L_{*}, \varphi(0, L_*)) \, \xi
= D_{u} \mathcal{F}(L_{*}, 
\phi_{\omega, \gamma}) \, \xi
= \bigl(-\lambda_{1} 
+ \frac{\pi^2}{L_{*}^2}  \bigr) \xi
= 0.
\]
Then, using \eqref{eq31-bi} and \eqref{eq14-bi}, 
we obtain
\[
\begin{split}
g(0, L_*) 
&= \frac{\partial \mathcal{F}_{\parallel}}{\partial a}(0, L_*) 
= \Bigl\langle D_{u} \mathcal{F}(L_{*}, \varphi(0, L_*)) \frac{\partial \varphi}{\partial a}(0, L_*), \, \xi \Bigr\rangle 
= \Bigl\langle D_{u} \mathcal{F}(L_{*}, \varphi(0, L_*)) \xi, \, \xi \Bigr\rangle 
= 0.
\end{split}
\]

\medskip

\noindent\textrm{(ii)} 
It is clear that $g$ is $C^{1}$ for $a \neq 0$ for any $p \geq 2$. 
Thus, it is enough to consider the case 
$a = 0$. 
For $a = 0$, 
we will show 
   \begin{equation} \label{eq11-bi}
   \lim_{a \neq 0, 
   a\to 0, L \to L_{*}} 
   \frac{\partial g}{\partial a}(a, L) = 0 
   \end{equation}
and 
    \begin{equation}\label{eq12-bi}
    \lim_{a \neq 0, 
    a\to 0, L \to L_{*}} 
    \frac{\partial g}{\partial L}(a, L) = \lambda'_{*} < 0. 
    \end{equation}

\noindent \textbf{Proof of \eqref{eq11-bi}: 
} 
For $a \neq 0$, 
since $\mathcal{F}_{\parallel}(0, L) = 0$, 
we have
\[
\frac{\partial g}{\partial a}(a, L) 
= - \frac{1}{a^2} \bigl( \mathcal{F}_{\parallel}(a, L) - \mathcal{F}_{\parallel}(0, L) \bigr)
+ \frac{1}{a} \frac{\partial \mathcal{F}_{\parallel}}{\partial a}(a, L).
\]
Adding and subtracting 
\[
\frac{1}{a} \frac{\partial \mathcal{F}_{\parallel}}{\partial a}(0, L)
= \frac{1}{a} \Bigl\langle D_u \mathcal{F}(L, \varphi(0, L)) \bigl[ \xi + \frac{\partial h}{\partial a}(0, L) \bigr], \, \xi \Bigr\rangle
\]
in the above expression, we obtain
\[
\begin{split}
\lim_{\substack{a \to 0 \\ a \neq 0,\, L \to L_*}} \frac{\partial g}{\partial a}(a, L) 
&= - \lim_{\substack{a \to 0 \\ a \neq 0,\, L \to L_*}} \frac{1}{a^2} \left(\mathcal{F}_{\parallel}(a, L) - \mathcal{F}_{\parallel}(0, L) - a \frac{\partial \mathcal{F}_{\parallel}}{\partial a}(0, L)\right) \\
&\quad + \lim_{\substack{a \to 0 \\ a \neq 0,\, L \to L_*}} \frac{1}{a} \left(\frac{\partial \mathcal{F}_{\parallel}}{\partial a}(a, L) - \frac{\partial \mathcal{F}_{\parallel}}{\partial a}(0, L)\right) \\
&=: I_1 + I_2.
\end{split}
\]
By \eqref{eq13-bi}, one has 
\begin{equation} \label{eq43-bi}
\begin{split}
\frac{\partial \mathcal{F}_{\parallel}}{\partial a}(a, L) 
&= \bigl\langle D_u \mathcal{F}(L, \varphi(a, L)) \frac{\partial \varphi}{\partial a}(a, L), \, \xi \bigr\rangle = \langle \mathcal{L}_+(a, L) \frac{\partial \varphi}{\partial a}(a, L), \, \xi \rangle \\
&= \bigl\langle \mathcal{L}_+(a, L) \bigl[ \xi + \frac{\partial h}{\partial a}(a, L) \bigr], \, \xi \bigr\rangle.
\end{split}
\end{equation}
It follows that
\[
I_2 = \lim_{\substack{a \to 0 \\ a \neq 0,\, L \to L_*}} 
\frac{1}{a} \bigl\langle 
\mathcal{L}_+(a, L) \bigl[ \xi + \frac{\partial h}{\partial a}(a, L) \bigr] 
- \mathcal{L}_+(0, L) \bigl[ \xi + \frac{\partial h}{\partial a}(0, L) \bigr], \, \xi \bigr\rangle.
\]

We will show that $I_2$ exists, 
separating the two cases $p>2$ and $p=2$.

\medskip

\textbf{(Case  $p > 2$).}
When $p > 2$, we see that 
$\mathcal{L}_{+}(a, L)$ is $C^{1}$ and 
$h(a, L)$ is $C^{2}$ in $a$. 
This yields 
    \begin{equation} \label{eq49-bi}
    \begin{split}
    I_{2} 
    & = \lim_{a \neq 0, 
    a\to 0, L \to L_{*}} 
    \frac{1}{a}
   \left\langle 
    \mathcal{L}_{+}(a, L)
    [\xi 
        + \frac{\partial h}{\partial a}
        (a, L)] 
        - \mathcal{L}_{+}(0, L) 
        [\xi 
        + \frac{\partial h}{\partial a}
        (a, L)], \xi \right\rangle \\
    & \quad  + 
    \lim_{a \neq 0, 
    a\to 0, L \to L_{*}} 
    \frac{1}{a}
   \left\langle 
  \mathcal{L}_{+}(0, L) 
        [\xi 
        + \frac{\partial h}{\partial a}
        (a, L)] 
        - \mathcal{L}_{+}(0, L)
[\xi 
+ \frac{\partial h}{\partial a} (0, L)], 
\xi
\right\rangle
    \\
     & = \lim_{a \neq 0, 
    a\to 0, L \to L_{*}} 
    \frac{1}{a}
   \left\langle 
    (\mathcal{L}_{+}(a, L) - \mathcal{L}_{+}(0, L)) 
        [\xi 
        + \frac{\partial h}{\partial a}
        (a, L)],  
        \xi \right\rangle
     \\
    & \quad  + 
    \lim_{a \neq 0, 
    a\to 0, L \to L_{*}} 
    \frac{1}{a}
   \left\langle 
  \mathcal{L}_{+}(0, L) 
        [\frac{\partial h}{\partial a}
        (a, L)] 
        - \frac{\partial h}{\partial a} 
        (0, L)], 
\xi
\right\rangle
     \\
    & = \left\langle \partial_{a} 
    \mathcal{L}_{+}(0, L_*)[\xi 
    + \frac{\partial h}{\partial a}
    (0, L_*)],
    \xi \rangle 
      + \langle \mathcal{L}_{+}(0, L_*) 
    \frac{\partial^{2} h}{\partial a^2}
    (0, L_*),
    \xi \right\rangle. 
    \end{split}
    \end{equation}
Then we have by \eqref{eq1-bi} that 
\begin{equation} \label{eq25-bi}
\begin{aligned}
    \left\langle 
\mathcal{L}_{+}(0, L_*) 
\frac{\partial^{2} h}{\partial a^{2}}(0, L_*),
\, \xi 
\right\rangle
& =
\left\langle
\frac{\partial^{2} h}{\partial a^{2}}(0, L_*),
\, \mathcal{L}_{+}(0, L_*) \xi
\right\rangle \\
&
=
\left\langle
\frac{\partial^{2} h}{\partial a^{2}}(0, L_*),
\, (\mathcal{L}_{+,0} - \frac{1}{L_{*}^2} \partial_{yy}) \, \xi
\right\rangle
= 0.
\end{aligned}
\end{equation}
In addition, observe from \eqref{eq14-bi} that
\[
\partial_{a}\mathcal{L}_{+}(0, L_*) 
= - p (p-1)\, \phi_{\omega,\gamma}^{\,p-2} \xi,
\qquad 
\frac{\partial h}{\partial a}(0, L_*) = 0. 
\]
These together with 
$\xi = \sqrt{2} \cos (\pi y) \chi(x)$ and \eqref{eq48-bi}
imply
\begin{equation} 
\label{eq50-bi}
\begin{split}
\left\langle 
\partial_{a} \mathcal{L}_{+}(0, L_*) 
\Bigl[\, \xi + \frac{\partial h}{\partial a}(0, L_*) \Bigr],
\, \xi 
\right\rangle
&= 
\left\langle 
\partial_{a} \mathcal{L}_{+}(0, L_*) \xi,
\, \xi \right\rangle \\
&= 
- p(p-1) 2^\frac{3}{2}
\int_{\mathbb{R}} 
\phi_{\omega,\gamma}^{\,p-2}(x) \, \chi^{3}(x)\, dx 
\int_{0}^{1} \cos^{3}(\pi y)\, dy 
= 0.
\end{split}
\end{equation}
Thus, by \eqref{eq49-bi}, \eqref{eq25-bi} and 
\eqref{eq50-bi}, we have $I_{2} = 0$. 

\medskip
By the l'Hôpital rule, we have 
\begin{equation}\label{eq1112-bi}
\begin{split}
I_1 
= - \lim_{\substack{a \to 0 \\ a \neq 0,\, L \to L_*}} 
\frac{1}{2a} \left( 
\frac{\partial \mathcal{F}_{\parallel}}{\partial a}(a, L) - 
\frac{\partial \mathcal{F}_{\parallel}}{\partial a}(0, L) 
\right) = - \frac{1}{2} I_2 = 0. 
\end{split}
\end{equation}
Thus, \eqref{eq11-bi} follows from   
    \[
    \frac{\partial g}{\partial a}(0, L_*) 
    = 
\lim_{a \neq 0, 
    a\to 0, L \to L_{*}} 
    \frac{\partial g}{\partial a}(a, L) 
    = 0.
    \]

\textbf{(Case 2) $p = 2$.} 

As in  \eqref{eq49-bi}, we obtain 
\begin{equation} \label{eq49-2-bi}
    \begin{split}
    I_{2} 
     & = \lim_{a \neq 0, 
    a\to 0, L \to L_{*}} 
    \frac{1}{a}
   \left\langle 
    (\mathcal{L}_{+}(a, L) - \mathcal{L}_{+}(0, L)) 
        [\xi 
        + \frac{\partial h}{\partial a}
        (a, L)],  
        \xi \right\rangle
     \\
    & \quad  + 
    \lim_{a \neq 0, 
    a\to 0, L \to L_{*}} 
    \frac{1}{a}
   \left\langle 
  \mathcal{L}_{+}(0, L) 
        [\frac{\partial h}{\partial a}
        (a, L) 
        - \frac{\partial h}{\partial a} 
        (0, L)], \xi \right\rangle \\
    & =: I_{2, 1} + I_{2, 2}. 
    \end{split}
    \end{equation}
Observe from 
\eqref{eq51-bi}, \eqref{eq31-bi} 
and \eqref{eq13-bi} that 
    \[
    \mathcal{L}_{+}(a, L) - \mathcal{L}_{+}(0, L) 
    = 2(|\phi_{\omega, \gamma} 
    + a \xi + h(a, L)| - 
    |\phi_{\omega, \gamma} 
    + h(0, L)|)
    \leq 2(|a| |\xi| 
    + |h(a, L)- h(0, L)|). 
    \]
By mean value theorem, we obtain 
    \[
    |h(a, L) - h(0, L)| 
    \leq \biggl| \frac{\partial h}{\partial a} 
    (a', L)\biggl||a|
    \]
for some $a' \in \R$ with 
$|a'| \leq a$. 
These yield that 
    \[
  \begin{split}
  & \quad \biggl|\frac{1}{a}
   \left\langle 
    (\mathcal{L}_{+}(a, L) - \mathcal{L}_{+}(0, L)) 
        [\xi 
        + \frac{\partial h}{\partial a}
        (a, L)],  
        \xi \right\rangle 
        \biggl| \\
  & = \biggl|
    \int_{\mathbb{S}}
    \frac{
    2(|\phi_{\omega, \gamma} 
    + a \xi + h(a, L)| - 
    |\phi_{\omega, \gamma} 
    + h(0, L)|)
    [\xi + \frac{\partial h}{
    \partial a}(a, L)]\xi
    }{a} dxdy 
  \biggl| \\
  & \leq 2\int_{\mathbb{S}}
  |\xi|
  \left(|\xi| + \biggl| 
  \frac{\partial h}{\partial a} 
    (a', L)\biggl| \right)  
    \biggl|\xi + 
    \frac{\partial h}{\partial a} 
    (a, L)
    \biggl| dxdy
    \end{split}
    \]
Note that 
the integrand can be bounded by 
some integrable function because
$h(a, L) \in H^{2}(\mathbb{S})$ 
is $C^{1}$ in a neighborhood of 
$(0, L_*)$. 

In addition, since $(a, L) \mapsto 
h(a, L) \in H^{2}(\mathbb{S}) 
\hookrightarrow L^{\infty}\cap 
C(\mathbb{S})$ is continuous and 
$\phi_{\omega, \gamma}(x, y) > 0$ 
for each $(x, y) \in 
\mathbb{S}$, 
there exists $\delta(x, y) > 0$, 
we obtain $\phi_{\omega, \gamma}
(x) 
+ a \xi(x, y) +h(a, L)(x, y) > 0$
for $|a|, |L - L_*| < \delta(x, y)$. 
Moreover, we have by \eqref{eq14-bi} that     
    \[
    \begin{split}
    & \quad
    \lim_{a \neq 0, 
    a\to 0, L \to L_{*}}
    \frac{
    2(|\phi_{\omega, \gamma} 
    + a \xi + h(a, L)| - 
    |\phi_{\omega, \gamma} 
    + h(0, L)|)
    [\xi + \frac{\partial h}{
    \partial a}(a, L) ]\xi
    }{a}(x, y) \\
    & = \lim_{a \neq 0, 
    a\to 0, L \to L_{*}}
    \frac{
    2(a \xi + h(a, L) -  
    h(0, L))
    [\xi + \frac{\partial h}{
    \partial a}(a, L)]\xi
    }{a}(x, y) \\
    & = 2 \left[\xi + 
    \frac{\partial h}{\partial a} (0, L_{*})\right]^2 \xi(x, y) 
    = 2\xi^3(x, y)
    \end{split}
    \]
for each $(x, y) \in \mathbb{S}$. 
By the Lebesgue dominated convergence
theorem and \eqref{eq48-bi}, 
we obtain 
    \[
    I_{2, 1} = 
    2 \int_{\mathbb{S}}
    \xi^{3}(x, y) dxdy 
    = 2^{\frac{5}{2}}
    \int_{\mathbb{R}} 
    \chi^3(x) dx \int_{0}^{1} 
    \cos^3(\pi y) dy = 0. 
    \]

Next, we consider $I_{2, 2}$.  
It follows from \eqref{eq23-bi} that 
    \[
        \frac{\partial h}{\partial a}
        (a, L) 
        = - ( P_{\perp} D_{u}
        \mathcal{F}
        (L, \varphi(a, L))|_{H^2_{\text{ort}}})^{-1}
        P_{\perp} D_{u}
        \mathcal{F}
        (L, \varphi(a, L)) \xi. 
    \]
It follows that 
    \begin{equation} \label{eq55-bi}
    \begin{split}
    & \quad 
    \dfrac{\frac{\partial h}{\partial a}(a, L) 
    - \frac{\partial h}{\partial a}(0, L)}{a}
    (x, y) \\
    & = - 
    \dfrac{ 
    (P_{\perp} 
    \mathcal{L}_{+}(a, L)
    |_{H^2_{\text{ort}}})^{-1}
    - ( P_{\perp} D_{u}
        \mathcal{L}_{+}(0, L)|_{H^2_{\text{ort}}})^{-1}
    }{a} P_{\perp} \mathcal{L}_{+}(a, L) \xi \\
    & \quad  
    - ( P_{\perp}
    \mathcal{L}_{+}(0, L)
    |_{H^2_{\text{ort}}})^{-1}
        P_{\perp} 
        \left(\dfrac{
        \mathcal{L}_{+}(a, L)
        - \mathcal{L}_{+}(0, L)
        }{a} 
        \right)\xi. 
    \end{split}
    \end{equation}
Moreover, we see that 
    \begin{equation}\label{eq54-bi}
    \begin{split}
    & \quad 
     (P_{\perp} 
    \mathcal{L}_{+}(a, L)
    |_{H^2_{\text{ort}}})^{-1}
    - ( P_{\perp} D_{u}
        \mathcal{L}_{+}(0, L)|_{H^2_{\text{ort}}})^{-1} \\
    & = (P_{\perp} \mathcal{L}_{+}(a, L)|_{H^2_{\text{ort}}})^{-1}
        P_{\perp} \Big(
        \mathcal{L}_{+}(0, L)
        - \mathcal{L}_{+}(a, L)
        \Big)|_{H^2_{\text{ort}}}
        (P_{\perp} 
        \mathcal{L}_{+}(0, L)
        |_{H^2_{\text{ort}}})^{-1} \\
    & = - (P_{\perp} D_{u}
        \mathcal{L}_{+}(a, L)|_{H^2_{\text{ort}}})^{-1}
        P_{\perp}\Big(
        \mathcal{L}_{+}(a, L) 
        - \mathcal{L}_{+}(0, L)
        \Big)\big|_{H^2_{\text{ort}}}
        (P_{\perp} 
        \mathcal{L}_{+}(0, L)|_{H^2_{\text{ort}}})^{-1}
    \end{split}
    \end{equation}
It follows from \eqref{eq55-bi},  
\eqref{eq54-bi} and $p = 2$ that 
\begin{equation} \label{eq53-bi}
    \begin{split}
    & \quad 
    \dfrac{\frac{\partial h}{\partial a}(a, L) 
    - \frac{\partial h}{\partial a}(0, L)}{a}
- p (P_{\perp} 
\mathcal{L}_{+}(0, L_*)
|_{H^2_{\text{ort}}})^{-1}
    P_{\perp} \phi_{\omega, \gamma}^{p-1}
    \xi^2 \\
    & = 
    (P_{\perp} 
    \mathcal{L}_{+}(a, L)
    |_{H^2_{\text{ort}}})^{-1}
        P_{\perp}\Big(
        \dfrac{ 
        (\mathcal{L}_{+}(0, L)
        - \mathcal{L}_{+}(a, L))
        }{a}
        \Big)\big|_{H^2_{\text{ort}}}
        (P_{\perp}
        \mathcal{L}_{+}(0, L)
        |_{H^2_{\text{ort}}})^{-1} 
        P_{\perp} \mathcal{L}_{+}(a, L) 
        \xi 
        \\
        & \quad 
        - ( P_{\perp} 
        \mathcal{L}_{+}(0, L)
        |_{H^2_{\text{ort}}})^{-1}
        \dfrac{P_{\perp} 
        \left(\mathcal{L}_{+}(a, L) 
        - \mathcal{L}_{+}(0, L)
        \right)}{a} \xi 
        - 2 (P_{\perp} 
\mathcal{L}_{+}(0, L)
|_{H^2_{\text{ort}}})^{-1}
    P_{\perp} \phi_{\omega, \gamma}^{p-1}
    \xi^2 \\
    &  \quad + 2 (P_{\perp} 
\mathcal{L}_{+}(0, L)
|_{H^2_{\text{ort}}})^{-1}
    P_{\perp} \phi_{\omega, \gamma}
    \xi^2 -  2 (P_{\perp} 
\mathcal{L}_{+}(0, L_*)
|_{H^2_{\text{ort}}})^{-1}
    P_{\perp} \phi_{\omega, \gamma}^{p-1}
    \xi^2. 
 \end{split}
    \end{equation}

Then, as in a similar computation to 
$I_{2, 1}$, 
we have 
    \begin{equation} \label{eq56-bi}
    \begin{split}
    \left\|\dfrac{\frac{\partial h}{\partial a}(a, L) 
    - \frac{\partial h}{\partial a}(0, L)}{a}
- 2 (P_{\perp} 
\mathcal{L}_{+}(0, L_*)
|_{H^2_{\text{ort}}})^{-1}
    P_{\perp} \phi_{\omega, \gamma}
    \xi^2\right\|_{H^{2}} \to 0 
    \qquad \mbox{as $a \to 0$}. 
    \end{split}
    \end{equation}
This implies that 
    \[
    \lim_{a \neq 0, a \to 0, L \to L_{*}} 
    \dfrac{\frac{\partial h}{\partial a}(a, L) 
    - \frac{\partial h}{\partial a}(0, L)}{a}
    = p (P_{\perp} \mathcal{L}_{+}(0, L_*)|_{H^2_{\text{ort}}})^{-1}
    P_{\perp} \phi_{\omega, \gamma}
    \xi^2 \quad \mbox{in $H^{2}(\mathbb{S})$}. 
    \]
Then, we obtain 
    \[
    \begin{split}
    I_{2,2} 
    & = - 2
    \langle 
    \mathcal{L}_{+}(0, L_*)
    (P_{\perp} \mathcal{L}_{+}(0, L_*)|_{H^2_{\text{ort}}})^{-1}
    P_{\perp} \phi_{\omega, \gamma}
    \xi^2, \xi \rangle \\
    & = - \langle
    (P_{\perp} \mathcal{L}_{L_{*}, +}|_{H^2_{\text{ort}}})^{-1}
    P_{\perp} (\partial_{a} \mathcal{L}_{+})
    \xi^2,  \mathcal{L}_{+, 0} 
    \xi \rangle = 0. 
    \end{split}
    \]
Thus, we obtain $I_{2} = 0$ in case 
$p = 2$, too.

\noindent \textbf{Proof of \eqref{eq12-bi}: 
}    
  By \eqref{eq31-bi} and \eqref{eq44-bi}, 
for $a \neq 0$, we have
    \begin{equation} \label{eq15-bi}
    \begin{split}
    \frac{\partial g}{\partial L} 
    (a, L) 
    & =
    \frac{1}{a} \left(
    \left\langle 
    D_{L}\mathcal{F}(L, 
    \varphi(a, L)),
    \xi 
    \right\rangle +
    \left\langle 
    D_{u}\mathcal{F}(L, 
    \varphi(a, L)) 
    \frac{\partial 
    \varphi}{\partial L}(a, L),
    \xi 
    \right\rangle \right) \\
    & = \frac{1}{a} \left(
    2 L^{-3}
    \left\langle \varphi_{yy}(a, L), 
    \xi \right\rangle
    + 
    \left\langle 
    \mathcal{L}_{+}(a, L) \frac{\partial 
    h}{\partial L}(a, L),
    \xi 
    \right\rangle \right)\\
    & = 
    \left(- 2 L^{-3}  
    \pi^2
    \|\xi
    \|_{L^{2}}^{2}
    + 2 L^{-3} \frac{\langle 
    h_{yy}(a, L), 
    \xi \rangle}{a} \right)
    + \frac{1}{a} 
    \langle 
    \mathcal{L}_{+}(a, L) \frac{\partial 
    h}{\partial L}(a, L),
    \xi 
    \rangle 
    \end{split}
    \end{equation}
Exploiting \eqref{eqH2Normh} and \eqref{eq34-bi}
yields 
    \begin{equation} 
    \label{eq30-bi}
    \lim_{a \neq 0, 
    a\to 0, L \to L_{*}} 
    \left(- 2 L^{-3}  
    \pi^2
    \|\xi
    \|_{L^{2}}^{2}
    + 2 L^{-3} \frac{\langle 
    h_{yy}(a, L), 
    \xi \rangle}{a} \right) = -2 L^{-3} \pi^2 
    = \lambda_{*}'. 
    \end{equation} 
Moreover, by \eqref{eq14-bi}, we obtain 
    \begin{equation} \label{eq16-bi}
    \begin{split}
    & \quad \lim_{a \neq 0, 
    a\to 0, L \to L_{*}} \frac{1}{a}
    \langle 
    \mathcal{L}_{+}(a, L) \frac{\partial 
    h}{\partial L}(a, L),
    \xi 
    \rangle \\
    & = \lim_{a \neq 0, 
    a\to 0, L \to L_{*}} 
    \frac{1}{a}\left( 
    \langle 
    \mathcal{L}_{+}(a, L) \frac{\partial 
    h}{\partial L}(a, L),
    \xi 
    \rangle - 
    \langle 
    \mathcal{L}_{+}(0, L_*) \frac{\partial 
    h}{\partial L}(0, L_*),
    \xi 
    \rangle
    \right) \\
    & = \lim_{a \neq 0, 
    a\to 0, L \to L_{*}} 
    \frac{1}{a}\left(
    \langle 
    \mathcal{L}_{+}(a, L) \frac{\partial 
    h}{\partial L}(a, L),
    \xi 
    \rangle - 
    \langle 
    \mathcal{L}_{+}(a, L) \frac{\partial 
    h}{\partial L}(0, L_*),
    \xi 
    \rangle
    \right) \\
    & \quad + 
    \lim_{a \neq 0, 
    a\to 0, L \to L_{*}} 
    \frac{1}{a}\left( 
    \langle 
    \mathcal{L}_{+}(a, L) \frac{\partial 
    h}{\partial L}(0, L_*),
    \xi 
    \rangle - 
    \langle 
    \mathcal{L}_{+}(0, L_*) \frac{\partial 
    h}{\partial L}(0, L_*),
    \xi 
    \rangle
    \right) \\
    & =: II_{1} + II_{2} 
    \end{split}
    \end{equation}
It follows from \eqref{eq14-bi} that
   \begin{equation} \label{eq18-bi}
   \begin{split}
   II_{2} = \langle 
   \partial_{a} \mathcal{L}_{+}(0, L_*) 
   \frac{\partial h}{\partial L} 
   (0, L_*), \xi \rangle
   & = - p (p-1) \langle 
   \phi_{\omega, \gamma}^{p-2}
   \xi
    \frac{\partial h}{\partial L} 
   (0, L_*), \xi \rangle \\ 
   & = -p(p-1) 
   \int_{\mathbb{S}} 
   \phi_{\omega, \gamma}^{p-2}
   |\xi|^2 
    \frac{\partial h}{\partial L} 
   (0, L_*) \, dxdy \\
   & = 0. 
   \end{split}
   \end{equation}
When $p > 2$, 
since $\mathcal{L}_{+}(0, L_*) \xi = 0$, we have 
    \begin{equation} \label{eq17-bi}
    II_{1} = 
    \langle \mathcal{L}_{+}(0, L_*) 
    \frac{\partial^2 h}{\partial L 
    \partial a}(0, L_*), 
    \xi 
    \rangle = 0.  
    \end{equation}
When $p = 2$, it follows from \eqref{eq24-bi} 
that 
    \[
    \frac{\partial h}{\partial L}(a, L) 
    = - \frac{2}{L^{3}}
    (P_{\perp} L_{+}(a, L)
    |_{H^{2}_{\text{ort}}})^{-1} 
    P_{\perp} \partial_{yy} \varphi(a, L). 
    \]
This implies that 
\begin{equation} \label{eq56-bi}
    \begin{split}
    & \quad 
    \dfrac{\frac{\partial h}{\partial L}(a, L) 
    - \frac{\partial h}{\partial L}(0, L)}{a}\\
    & = - \frac{2}{L^{3}}
    \dfrac{ 
    (P_{\perp} 
    \mathcal{L}_{+}(a, L)
    |_{H^2_{\text{ort}}})^{-1}
    - ( P_{\perp} 
        \mathcal{L}_{+}(0, L)|_{H^2_{\text{ort}}})^{-1}
    }{a} P_{\perp} \partial_{yy} 
    \varphi(a, L) \\
    & \quad  
    - ( P_{\perp}
    \mathcal{L}_{+}(0, L)
    |_{H^2_{\text{ort}}})^{-1}
        P_{\perp} 
        \left(\dfrac{
        \partial_{yy} \varphi(a, L)
        - \partial_{yy} \varphi(0, L)
        }{a} 
        \right)
    \end{split}
    \end{equation} 
Then, we see from \eqref{eq54-bi} 
and $\varphi(0, L_*) = \phi_{\omega, \gamma}$
that 
    \begin{equation} \label{eq57-bi}
    \begin{split}
    & \quad \lim_{a \neq 0, a \to 0, L \to L_*}
    \left(
    - \frac{2}{L^{3}}
    \dfrac{ 
    (P_{\perp} 
    \mathcal{L}_{+}(a, L)
    |_{H^2_{\text{ort}}})^{-1}
    - ( P_{\perp} 
        \mathcal{L}_{+}(0, L)|_{H^2_{\text{ort}}})^{-1}
    }{a} P_{\perp} \partial_{yy} 
    \varphi(a, L)
    \right) \\
    & = \frac{2}{L_{+}^3} 
    (P_{\perp} \mathcal{L}(0, L_{*})
    |_{H^2_{\text{ort}}})^{-1} 
    P_{\perp} (\phi_{\omega, \gamma}\xi)|_{H^2_{\text{ort}}} 
    (P_{\perp} \mathcal{L}(0, L_{*})
    |_{H^2_{\text{ort}}})^{-1} 
    P_{\perp} \partial_{yy} 
    \varphi(0, L_*) = 0.
    \end{split}
    \end{equation}
In addition, since $ \frac{\partial \varphi}{\partial a} 
        (0, L_*) = \xi$, 
we have 
    \begin{equation} \label{eq58-bi}
    \begin{split}
    & \quad \lim_{a \neq 0, a \to 0, L \to L_*}
    ( P_{\perp}
    \mathcal{L}_{+}(0, L)
    |_{H^2_{\text{ort}}})^{-1}
        P_{\perp} 
        \left(\dfrac{
        \partial_{yy} \varphi(a, L)
        - \partial_{yy} \varphi(0, L)
        }{a} 
        \right) \\
    & = ( P_{\perp}
    \mathcal{L}_{+}(0, L_{*})
    |_{H^2_{\text{ort}}})^{-1}
        P_{\perp} \partial_{yy} 
        \frac{\partial \varphi}{\partial a} 
        (0, L_*) = 0. 
        \end{split}
    \end{equation}
Therefore, by \eqref{eq16-bi}, \eqref{eq56-bi}, 
\eqref{eq57-bi} and \eqref{eq58-bi}, one has 
$II_{1} = 0$. 

This together with 
\eqref{eq15-bi}--\eqref{eq16-bi}, we obtain 
\eqref{eq12-bi}. In particular, by \eqref{eq15-bi}, 
\eqref{eq30-bi} and \eqref{eq18-bi} it follows that 
    \[
    \frac{\partial g}{\partial L}(0, L_*) 
    = \lambda_{*}' < 0. 
    \]

\end{proof}
We now in the position to 
give the proof of Proposition \ref{ex-bi}.
\begin{proof}[Proof of Proposition \ref{ex-bi}]
It follows from Proposition \ref{prop1-bi} and the implicit function theorem that there exists a neighborhood $(L_{*} - \delta_{1}, L_{*} + \delta_{1}) \times (-a_{1}, a_{1})$ and a map $a \in (-a_{1}, a_{1}) \mapsto L(a) \in (L_{*} - \delta_{1}, L_{*} + \delta_{1})$ with $L(0) = L_{*}$ satisfying $g(a, L(a)) = 0$ for $a \in (-a_{0}, a_{0})$. Setting $h(a) = h(a, L(a))$ and 
\begin{equation} \label{eq62-bi}
\varphi(a) = \phi_{\omega, \gamma}+ a \xi + h(a),
\end{equation}
we see that $\varphi(a)$ satisfies \eqref{eqS2D-2} with $L = L(a)$. 

Clearly, we see that $\varphi (a) 
\in C^{2}((-a_0, a_0), H^{2}(\mathbb{S}))$ 
for sufficiently small $a_{0} > 0$ when 
$p > 2$. 
Moreover, it follows from \eqref{eq56-bi} that 
we deduce that 
$\varphi (a) 
\in C^{2}((-a_0, a_0), H^{2}(\mathbb{S}))$ 
even when $p = 2$. 
Then, by \eqref{eq7-bi} and \eqref{eq14-bi}, we have
\begin{equation} \label{eq59-bi}
h(0) = h(L(0), 0) = h(L_{*}, 0) = 0, \qquad
\frac{d h}{d a}(0) = \frac{\partial h}{\partial L}(L_{*}, 0) \frac{d L}{d a}(0) + \frac{\partial h}{\partial a}(L_{*}, 0) = 0.
\end{equation}
This implies 
$h(a) 
\in C^{2}((-a_0, a_0), H^{2}(\mathbb{S}))$ 
that 
\begin{equation*}
    \|h(a)\|_{H^{2}} = O(a^2).
\end{equation*}

Finally, we shall show 
the positivity of $\varphi(a)$. 
We put 
    \[
    \mathcal{L}_{-}(a):= H_{L(a)} + \omega - 
|\varphi(a)|^{p-1}.
    \]
$\mathcal{L}_{-}(0) 
\phi_{\omega, \gamma} = 0$ and 
$\phi_{\omega, \gamma}(x, y) > 0$ 
for $(x, y) \in \mathbb{S}_{L(a)}$. 
Thus, the first eigenvalue of 
$\mathcal{L}_{-}(0)$ is $0$. 
By a perturbation argument, we see that 
the first eigenvalue of 
$\mathcal{L}_{-}(a)$ 
is close to 
$0$ if $|a|>0$ is sufficiently small. 
In addition, $\mathcal{L}_{-}(a) 
\varphi(a) = 0$. 
Thus, $0$ is also the first eigenvalue of 
$\mathcal{L}_{-}(a)$ and 
$\varphi(a)$ is the corresponding eigenfunction. 
Then, since $\varphi(a)$ is the first 
eigenfunction of the operator 
$\mathcal{L}_{-}(a)$,  
we see that $\varphi(a)$ is positive. 
This completes the proof.
\end{proof}

\subsection{Direction of 
the bifurcation curve}

In this section, we determine the direction of the bifurcation branch.  
We begin with the following result.

\begin{prop} \label{prop2-bi}
The map \(
a \in (-a_{1}, a_{1}) \longmapsto 
L(a) \in (L_{*} - \delta_{1},\, L_{*} + \delta_{1})
\)
is of class $C^{2}$, and admits the expansion
\begin{equation} \label{eq60-bi}
L(a) = L_{*} 
+ \frac12 \frac{d^{2} L}{d a^{2}}(0)\, a^{2} 
+ o(a^{2}).
\end{equation}
Moreover,
\begin{equation}\label{eq19-bi}
\begin{split}
\frac{d^{2} L}{d a^{2}}(0)
& =
\frac{p(p-1)(p-2)}{3 \lambda_{*}'} 
\int_{\mathbb{S}} 
\phi_{\omega,\gamma}^{p-3} 
\xi^{4}\, dx\, dy \\
&\quad 
+ \frac{p^{2}(p-1)^{2}}{\lambda_{*}'} 
\left\langle 
\phi_{\omega,\gamma}^{p-2} \xi^{2},\,
\bigl(P_{\perp}(\mathcal{L}_{+,0} 
- L_{*}^{-2} \partial_{yy})\!\mid_{H^{2}_{\mathrm{ort}}}\bigr)^{-1}
\bigl(\phi_{\omega,\gamma}^{p-2} \xi^{2}\bigr)
\right\rangle .
\end{split}
\end{equation}
\end{prop}
To prove Proposition \ref{prop2-bi}, 
we need the following lemma, 
which comes from the maximum principle: 
\begin{lem} [Lemma 4.1 of \cite{BeNi90}]
\label{lem-exp-decay}
 Let $p > 1$ and $\varphi$ be a positive 
 solution of \eqref{eqS2D-2}. 
Then, for $\varepsilon > 0$, there exists 
$C_{1, \varepsilon}, C_{2, \varepsilon} > 0$ such that 
    \[
    C_{1, \varepsilon} e^{- (\sqrt{\omega} 
    - \varepsilon)|x|} \leq \varphi(x, y) 
    \leq C_{2, \varepsilon} 
    e^{- (\sqrt{\omega} 
    - \varepsilon)|x|} 
    \qquad \mbox{for $(x, y) \in \mathbb{S}$}. 
    \]
\end{lem}
\begin{remark}
Using the maximum principle, 
we can also find that 
for $\varepsilon > 0$, there exists 
$\widetilde{C}_{1, \varepsilon}, 
\widetilde{C}_{2, \varepsilon} > 0$ such that 
    \begin{equation} \label{eq1-decay}
    \widetilde{C}_{1, \varepsilon} 
    e^{- (\sqrt{\omega + \lambda_{1}} 
    - \varepsilon)|x|} \leq |\xi(x)| 
    \leq \widetilde{C}_{2, \varepsilon} 
    e^{- (\sqrt{\omega + \lambda_{1}} 
    - \varepsilon)|x|} 
    \qquad \mbox{for $x \in \mathbb{R}$}. 
    \end{equation}
\end{remark}

\begin{proof}[Proof of Proposition \ref{prop2-bi}]
We divide the proof into two cases: $p>2$ and $p=2$.

\medskip
\textbf{(Case $p>2$).}
Since $a \mapsto L(a)$ is defined by the implicit function theorem applied to the map $g$ in a neighborhood of $(0, L_{*})$, and since $g$ is $C^{2}$ in for $|a|\in (0,a_0)$, it follows that $L$ is of class $C^{2}$ in the same interval.
Differentiating \( g(a, L(a)) = 0 \) in $a$ yields
\begin{equation*}
\frac{\partial g}{\partial a}(a, L(a))
+ 
\frac{\partial g}{\partial L}(a, L(a))\, \frac{dL}{da}(a)
= 0.
\end{equation*}
Consequently, by \eqref{eqIFTCond} we have
\begin{equation} \label{eq38-bi}
\frac{dL}{da}(0)
= 
-\, \frac{
\frac{\partial g}{\partial a}(0, L_{*})
}{
\frac{\partial g}{\partial L}(0, L_{*})
}
= 0.
\end{equation}
We compute the second derivative $\frac{d^{2}L}{da^{2}}(0)$. By \eqref{eqIFTCond} and \eqref{eq12-bi}, we get
\begin{equation} \label{eq46-bi}
\begin{aligned}
\frac{d^{2}L}{da^{2}}(0)
&=
\lim_{a\to 0} 
\frac{1}{a}\Bigl( \frac{dL}{da}(a) - \frac{dL}{da}(0) \Bigr) =
- \lim_{\substack{a\neq 0 \\ a\to 0}}
\frac{1}{a}\,
\frac{ \frac{\partial g}{\partial a}(a, L(a)) }
     { \frac{\partial g}{\partial L}(a, L(a)) } =
- \frac{1}{\lambda_{*}'} 
\lim_{\substack{a\neq 0 \\ a\to 0}}
\frac{1}{a}\,
\frac{\partial g}{\partial a}(a, L(a)).
\end{aligned}
\end{equation}
We now show that the limit
\[
\lim_{\substack{a\neq 0,\; a\to 0 \\ L\to L_{*}}}
\frac{1}{a}\,
\frac{\partial g}{\partial a}(a, L)
\]
exists.  
Once this is established, the continuity of the map 
\[
(a,L)\mapsto \frac{1}{a}\,\frac{\partial g}{\partial a}(a,L)
\quad\text{for } a\neq 0,
\]
together with the fact that $L(a)\to L_{*}$ as $a\to 0$, implies that
\[
\lim_{\substack{a\neq 0 \\ a\to 0}}
\frac{1}{a}\,
\frac{\partial g}{\partial a}(a, L(a))
=
\lim_{\substack{a\neq 0,\; a\to 0 \\ L\to L_{*}}}
\frac{1}{a}\,
\frac{\partial g}{\partial a}(a, L),
\]
which justifies the identification of the two limits. For $a \neq 0$, we have by \eqref{eq44-bi} that 
    \[
    \begin{split}
    \frac{1}{a}
    \frac{\partial g}{\partial a} 
    (a, L) 
    = - \frac{1}{a^{3}} 
    \left\{ 
    \mathcal{F}_{\parallel}(a, L) 
    - \mathcal{F}_{\parallel}(0, 
    L)
    \right\}
    + \frac{1}{a^2} 
    \frac{\partial 
    \mathcal{F}_{\parallel}}
    {\partial a} (a, L).  
    \end{split}
    \]
Adding and subtracting 
    \[
    \frac{1}{a^{2}} 
    \frac{\partial \mathcal{F}_{\parallel}}{\partial a}(0, L_*)
    = 
    \frac{1}{a^2} \langle 
D_{u} \mathcal{F}(L_*, 
\phi_{\omega, \gamma})
\xi, 
\xi
\rangle
    \]    
to the equation above, 
we obtain 
    \begin{equation} \label{eq47-bi}
    \begin{split}
    \lim_{a \neq 0, a \to 0, 
    L \to L_*}
    \frac{1}{a}
    \frac{\partial g}{\partial a} 
    (a, L(a)) 
    & = -
    \lim_{a \neq 0, a \to 0, 
    L \to L_*}
    \frac{1}{a^{3}} 
    \left\{ 
    \mathcal{F}_{\parallel}(a, L) 
    - \mathcal{F}_{\parallel}(0, 
    L) 
    - a \frac{\partial \mathcal{F}_{\parallel}}{\partial a}(0, L_*)
    \right\} \\
    & \quad 
    + 
    \lim_{a \neq 0, a \to 0, 
    L \to L_*}
    \frac{1}{a^2}
    \left\{
    \frac{\partial 
    \mathcal{F}_{\parallel}}
    {\partial a} (a, L) 
    - \frac{\partial \mathcal{F}_{\parallel}}{\partial a}(0, L_*)
    \right\} \\
    & =: III_{1} + III_{2}
    \end{split}
    \end{equation}
Then, by the l'Hopital rule, 
we have 
    \begin{equation} \label{eq42-bi}
    \begin{split}
    III_{1} 
    & =  - 
    \lim_{a \neq 0, a \to 0, 
    L \to L_*} 
     \frac{ 
    \left\{ 
    \frac{\partial \mathcal{F}_{\parallel}}
    {\partial a}
    (a, L)  
    - \frac{\partial \mathcal{F}_{\parallel}}{\partial a}(0, L_*) 
    \right\}
    }{3 a^{2}} = - \frac{1}{3} III_{2}. 
    \end{split}
    \end{equation}
Thus, it suffices to 
consider $III_2$. 
As an intermediate step, we now show that 
\begin{equation}\label{eq32bis-bi}
\frac{\partial^2 \mathcal{F}_{\parallel}}{\partial a^2}(0, L_*) = 0.
\end{equation}
Indeed, it follows from \eqref{eq43-bi} 
that
  \begin{equation} \label{eq32-bi}
\begin{split}
\frac{\partial^2 \mathcal{F}_{\parallel}}{\partial a^2}(a, L) = 
\langle D_{u}^2 \mathcal{F} 
    (L, \varphi(a, L)) 
    \left(\frac{\partial \varphi}{\partial a}(a, L)\right)^2, 
    \xi \rangle 
    + \langle 
    D_{u} \mathcal{F} 
    (L, \varphi(a, L)) 
    \frac{\partial^2 \varphi}
    {\partial a^2}(a, L), 
    \xi \rangle.
\end{split}
\end{equation}
Observe from \eqref{eq25-bi} that 
\begin{equation} \label{eq27-bi}
    \begin{split}
    \langle 
    D_{u} \mathcal{F} 
    (L_{*}, \phi_{\omega, \gamma}^{(1)}) \frac{\partial^2 \varphi}
    {\partial a^2}(0, L_*), 
    \xi \rangle
    = \langle 
      \mathcal{L}_{+}(0, L_*)
    \frac{\partial^2 h}
    {\partial a^2}(0, L_*), \xi  
    \rangle
    = \langle 
    \frac{\partial^2 h}
    {\partial a^2}(0, L_*), \mathcal{L}_{+}(0, L_*) \xi  
    \rangle
    = 0. 
    \end{split}
    \end{equation}
In addition, we have by \eqref{eq31-bi} 
and \eqref{eq48-bi} that 
     \begin{equation} \label{eq28-bi}
    \begin{split}
    \langle D_{u}^2 \mathcal{F} 
    (L_{*}, \phi_{\omega, \gamma}^{(1)}) 
     \left(\frac{\partial \varphi}{\partial a}
     (0, L_*)\right)^2, 
    \xi \rangle  
    & = 
    \langle D_{u}^2 \mathcal{F} 
    (L_{*}, \phi_{\omega, \gamma}^{(1)}) 
    \xi^2, 
    \xi \rangle \\
    & = -p(p-1) 
    \int_{\R} \phi_{\omega, \gamma}^{p-2} 
    \chi^3 dx 
    \int_{0}^{1} \cos^3 \pi y dy
    = 0. 
    \end{split}
    \end{equation}
These together with \eqref{eq32-bi}, 
yields that  \eqref{eq32bis-bi}.

Thus, using \eqref{eq32bis-bi} and the l'Hopital rule, we have 
    \begin{equation}\label{eq29-bi}
    \begin{split}
    III_{2} 
    & = 
    \lim_{a \neq 0, a \to 0, L \to L_{*}}
    \frac{1}{a^2}
    \left\{
    \frac{\partial 
    \mathcal{F}_{\parallel}}
    {\partial a} (a, L) 
    - \frac{\partial \mathcal{F}_{\parallel}}{\partial a}(0, L_*)
    - a \frac{\partial^2 \mathcal{F}_{\parallel}}{\partial a^2}(0, L_*)
    \right\} \\
    & = 
    \frac{1}{2}
    \lim_{a \neq 0, a \to 0, L \to L_{*}}
    \frac{1}{a}
    \left\{
    \frac{\partial^2 \mathcal{F}_{\parallel}}{\partial a^2}(a, L_*)
    - \frac{\partial^2 \mathcal{F}_{\parallel}}{\partial a^2}(0, L_*)
    \right\}. 
    \end{split}
    \end{equation}
Since 
$D_{u} \mathcal{F}(L, 
\varphi(0, L_*)) \xi 
= \mathcal{L}_{+}(0, L_{*}) 
\xi 
= 0$, one has by \eqref{eq32-bi} 
that 
\begin{equation} \label{eq29-2-bi}
\begin{split}
& \quad 
\lim_{a \neq 0, a \to 0, L \to L_{*}}
\frac{1}{a}
    \left\{
    \frac{\partial^2 \mathcal{F}_{\parallel}}{\partial a^2}(a, L)
    - \frac{\partial^2 \mathcal{F}_{\parallel}}{\partial a^2}(0, L_*)
    \right\} \\
& = 
\lim_{a \neq 0, a \to 0, L \to L_{*}}
\frac{1}{a}\Bigg\{
\left\langle D_{u}^2 \mathcal{F} 
    (L, \varphi(a, L)) 
    \left(\frac{\partial \varphi}{\partial a}(a, L)\right)^2 - D_{u}^2 \mathcal{F} 
    (L, \varphi(0, L)) 
    \left(\frac{\partial \varphi}{\partial a}(0, L)\right)^2, 
    \xi \right\rangle
    \\
&\quad  + 
\lim_{a \neq 0, a \to 0, L \to L_{*}}
\left\langle 
    D_{u} \mathcal{F} 
    (L, \varphi(a, L)) 
    \frac{\partial^2 \varphi}
    {\partial a^2}(a, L) - D_{u} \mathcal{F} 
    (L, \varphi(0, L)) 
    \frac{\partial^2 \varphi}
    {\partial a^2}(0, L), 
    \xi \right\rangle \Bigg\}
    \\
& = 
\lim_{a \neq 0, a \to 0, L \to L_{*}}
\frac{1}{a}\Bigg\{
\left\langle \left(D_{u}^2 \mathcal{F} 
    (L, \varphi(a, L)) 
    - D_{u}^2 \mathcal{F} 
    (L, \varphi(0, L)) 
    \right)
    \left(\frac{\partial \varphi}{\partial a}(a, L)\right)^2, 
    \xi \right\rangle \\
& \quad 
    + \left\langle D_{u}^2 \mathcal{F} 
    (L, \varphi(0, L))
    \left[\left(\frac{\partial \varphi}{\partial a}(a, L)\right)^2
    - 
    \left(\frac{\partial \varphi} 
    {\partial a}(0, L)\right)^2\right], 
    \xi \right\rangle \\
& \quad 
+ \left\langle 
    \left(D_{u} \mathcal{F} 
    (L, \varphi(a, L)) 
    - D_{u} \mathcal{F} 
    (L, \varphi(0, L))\right)
    \frac{\partial^2 \varphi}
    {\partial a^2}(a, L), 
    \xi \right\rangle \\
    & \quad 
    + \left\langle 
    D_{u} \mathcal{F} 
    (L, \varphi(0, L)) 
    \left(\frac{\partial^2 \varphi}
    {\partial a^2}(a, L) 
    - 
    \frac{\partial^2 \varphi}
    {\partial a^2}(0, L) 
    \right), 
    \xi \right\rangle
    \Bigg\} \\
& = 
\lim_{a \neq 0, a \to 0, L \to L_{*}}
\frac{1}{a}\Bigg\{
\left\langle \left(D_{u}^2 \mathcal{F} 
    (L, \varphi(a, L)) 
    - D_{u}^2 \mathcal{F} 
    (L, \varphi(0, L)) 
    \right)
    \left(\frac{\partial \varphi}{\partial a}(a, L)\right)^2, 
    \xi \right\rangle \\
& \quad 
    + \left\langle D_{u}^2 \mathcal{F} 
    (L, \varphi(0, L))
    \left[\left(\frac{\partial \varphi}{\partial a}(a, L)\right)^2
    - 
    \left(\frac{\partial \varphi} 
    {\partial a}(0, L)\right)^2\right], 
    \xi \right\rangle \\
& \quad 
+ \left\langle 
    \left(D_{u} \mathcal{F} 
    (L, \varphi(a, L)) 
    - D_{u} \mathcal{F} 
    (L, \varphi(0, L))\right)
    \frac{\partial^2 \varphi}
    {\partial a^2}(a, L), 
    \xi \right\rangle
    \Bigg\} \\
& =: III_{2, 1} + III_{2,2} 
+ III_{2, 3}. 
\end{split}   
\end{equation}

It follows from the 
mean value theorem that 
there exists $\theta = \theta (a) \in 
(0, 1)$ such that 
    \[
    \begin{split}
    & \quad \frac{1}{a}
    \left\langle \left(D_{u}^2 \mathcal{F} 
    (L, \varphi(a, L)) 
    - D_{u}^2 \mathcal{F} 
    (L, \varphi(0, L)) 
    \right)
    \left(\frac{\partial \varphi}{\partial a}(a, L)\right)^2, 
    \xi \right\rangle \\
    & = -p (p-1) (p-2) 
    \int_{\mathbb{S}} 
    \varphi^{p-3}(\theta a, L)
    \frac{\partial \varphi}{\partial a}
    (\theta a, L) \left(\frac{\partial \varphi}{\partial a}(a, L)\right)^2 
    \xi (x, y)dxdy.
    \end{split}
    \]
For each $(x , y) \in \mathbb{S}$, 
we have by 
$\varphi(0, L) = \phi_{\omega, \gamma}
^{(1)}$ and 
$\frac{\partial \varphi}{\partial a}(0, L_*) 
= \xi$ that 
    \[  
    \begin{split}
    \lim_{a \to 0}
    \varphi^{p-3}(\theta a, L)
    \frac{\partial \varphi}{\partial a}
    (\theta a, L) \left(\frac{\partial \varphi}{\partial a}(a, L)\right)^2 
    \xi (x, y)
    & = \varphi^{p-3}(0, L)
    \left(\frac{\partial \varphi}{\partial a}(0, L)\right)^3 
    \xi (x, y) \\
    & = (\phi_{\omega, \gamma}^{(1)})^{p-3} 
    \xi^{4}(x, y). 
    \end{split}
    \]
In addition, by Lemma \ref{lem-exp-decay},  
Proposition \eqref{ex-bi} and 
\eqref{eq1-decay},
we see that 
    \[
    \biggl|\varphi^{p-3}(\theta a, L)
    \frac{\partial \varphi}{\partial a}
    (\theta a, L) \left(\frac{\partial \varphi}{\partial a}(a, L)\right)^2 
    \xi (x, y) \biggl| 
    \leq C
    \exp(-\left(\sqrt{\omega + 
    \lambda_{1}} - (p-3) \sqrt{\omega} + \varepsilon \right)|x|) 
    \in L^{1}(\mathbb{S}). 
    \]
Then by the Lebesgue dominated convergence 
theorem, one has 
    \[
    \begin{split}
    III_{2, 1} & = 
     \lim_{a \neq 0, a \to 0, L \to L_{*}}
    \frac{1}{a}
    \left\langle \left(D_{u}^2 \mathcal{F} 
    (L, \varphi(a, L)) 
    - D_{u}^2 \mathcal{F} 
    (L, \varphi(0, L)) 
    \right)
    \left(\frac{\partial \varphi}{\partial a}(a, L)\right)^2, 
    \xi \right\rangle \\ 
   & = - p(p-1)(p-2) 
    \int_{\R \times (0, 1)}  
    (\phi_{\omega, \gamma}^{(1)})^{p-3}
    (\xi)^4 dxdy 
    \end{split}
    \] 
Since 
$D_{u}^2 \mathcal{F} 
    (L, \varphi(0, L)) 
= -p (p-1)
(\phi_{\omega, \gamma}^{(1)})^{p-2}    
    $ 
and 
\[
\frac{\partial^2 \varphi}
    {\partial a^2}
    (0, L)
= \frac{\partial^2 h}
    {\partial a^2}
    (0, L)
= p (p-1) (P_{\perp} (\mathcal{L}_{+, 0}
    - \frac{1}{L_{*}^2} \partial_{yy}) |_{H^2_{\text{ort}}})^{-1} 
    \phi_{\omega, \gamma}^{p-2} 
    \xi^2 
\]   
(see \eqref{eq26-bi}), 
we obtain 

\[
\begin{split}
    & III_{2, 2} =  
     \lim_{a \neq 0, a \to 0, L \to L_{*}}
    \frac{1}{a}
    \left\langle D_{u}^2 \mathcal{F} 
    (L, \varphi(0, L))
    \left[\left(\frac{\partial \varphi}{\partial a}(a, L)\right)^2
    - 
    \left(\frac{\partial \varphi} 
    {\partial a}(0, L)\right)^2\right], 
    \xi \right\rangle \\
    & = 2
    \langle D_{u}^2 \mathcal{F} 
    (L, \varphi(0, L))
    \frac{\partial \varphi}{\partial a}(0, L) 
    \frac{\partial^2 \varphi}{\partial a^2}(0, L), 
    \xi \rangle \\
    & = - 2 p^2 (p-1)^2 
    \langle 
    \phi_{\omega, \gamma}^{p-2} \xi^2, 
    (P_{\perp} (\mathcal{L}_{+, 0}
    - \frac{1}{L_{*}^2} \partial_{yy}) |_{H^2_{\text{ort}}})^{-1} 
    \phi_{\omega, \gamma}^{p-2} 
    \xi^2
    \rangle.
\end{split}
    \]
Similarly, it follows that  
   
     \[
    \begin{split}
    III_{2, 3} 
    & = 
    \lim_{a \neq 0, a \to 0, L \to L_{*}}
    \frac{1}{a}
    \langle 
    \left(D_{u} \mathcal{F} 
    (L, \varphi(a, L)) 
    - D_{u} \mathcal{F} 
    (L, \varphi(0, L))\right)
    \frac{\partial^2 \varphi}
    {\partial a^2}(a, L), 
    \xi \rangle \\
    & = 
    \langle D_{u}^2 \mathcal{F} 
    (L, \varphi(0, L))
    \frac{\partial \varphi}{\partial a}(0, L) 
    \frac{\partial^2 \varphi}{\partial a^2}(0, L), 
    \xi \rangle \\
    & = - p^2 (p-1)^2 
    \langle 
    \phi_{\omega, \gamma}^{p-2} \xi^2, 
    (P_{\perp} (\mathcal{L}_{+, 0}
    - \frac{1}{L_{*}^2} \partial_{yy}) |_{H^2_{\text{ort}}})^{-1} 
    \phi_{\omega, \gamma}^{p-2} 
    \xi^2
    \rangle.
    \end{split}
    \]
Therefore, we have by 
\eqref{eq29-bi} and \eqref{eq29-2-bi} that 
    
     \[
    \begin{split}
    III_{2} 
    & = 
    - \frac{1}{2}p(p-1)(p-2) 
    \int_{\R \times (0, 1)}  
    (\phi_{\omega, \gamma}^{(1)})^{p-3}
    (\xi)^4 dxdy \\
    & \quad - \frac{3}{2} 
    p^2 (p-1)^2 
    \langle 
    \phi_{\omega, \gamma}^{p-2} \xi^2, 
    (P_{\perp} (\mathcal{L}_{+, 0}
    - \frac{1}{L_{*}^2} \partial_{yy}) |_{H^2_{\text{ort}}})^{-1} 
    \phi_{\omega, \gamma}^{p-2} 
    \xi^2
    \rangle. 
    \end{split}
    \]
This together with \eqref{eq46-bi}, 
\eqref{eq47-bi} and \eqref{eq42-bi}
implies that 
   
     \[
    \begin{split}
     \frac{d^2 L}{d a^{2}}(0) 
     & = 
     - \frac{1}{\lambda_{*}'} 
    \lim_{a \neq 0, a \to 0} 
    \frac{1}{a}
    \frac{\partial g}{\partial a} 
    (a, L(a))  = - \frac{2}{3 \lambda_{*}'}
    III_{2} \\
    & =  \frac{p(p-1)(p-2) }{3 
    \lambda_{*}'}
    \int_{\R \times (0, 1)}  
    (\phi_{\omega, \gamma}^{(1)})^{p-3}
    (\xi)^4 dxdy \\
    & \quad + \frac{1}{\lambda_{*}'} 
    p^2 (p-1)^2 
    \langle 
    \phi_{\omega, \gamma}^{p-2} \xi^2, 
    (P_{\perp} (\mathcal{L}_{+, 0}
    - \frac{1}{L_{*}^2} \partial_{yy}) |_{H^2_{\text{ort}}})^{-1} 
    \phi_{\omega, \gamma}^{p-2} 
    \xi^2
    \rangle.  
    \end{split}
    \] 
This completes the proof. 
\end{proof}

\section{Orbital stability of 
the line soliton}\label{sec:stab}
In this section, we show Theorems \ref{thm-sta} and \ref{thmStab2}. 

\subsection{Stability of the 
line soliton for $0 < L < L_{*}$} \label{secStabBor}

We fix $\gamma < 0$, $\omega > \frac{\gamma^2}{4}$ and we define the action functional $S_{L, \gamma}:H^1(\mathbb{S}_{L}) \to \mathbb{R}$ by
    \begin{equation}\label{eq-action}
    S_{L, \gamma}(u) 
    := \frac{1}{2} \|\nabla u\|_{L^{2}(\mathbb{S}_{L})}^{2}  + \omega \|u\|_{L^2_{x, y}(\mathbb{S}_{L})}^{2} + \frac{\gamma}{2} \int_0^L |u(0, y)|^2 dy - \frac{1}{p+1} \|u\|_{L^{p+1}(\mathbb{S}_{L})}^{p+1}.
    \end{equation}
We write a complex-valued function $u \in  H^{1}(\mathbb{S}_{L})$ as $u = u^R + i u^I$ with $u^R, u^I \in H^{1}(\mathbb{S}_{L})$ real valued. The quadratic form associated with the second variation of $S_{L, \gamma}$ at $\phi_{\omega, \gamma}^{(L)}$ can be written as 
\begin{equation*}
    \langle S_{L, \gamma}^{\prime \prime}(\phi_{\omega, \gamma}^{(L)})u, u \rangle = \langle \mathcal{L}_{L, +} u^R, u^R \rangle + \langle \mathcal{L}_{L, -} u^I, u^I \rangle,  
\end{equation*}
where $\mathcal{L}_{L, \pm}: \mathcal{D}(H) \to L^2(\mathbb S)$ are self-adjoint operators defined by
\begin{equation}\label{eqL+-Def}
    \mathcal{L}_{L, \pm} = H + \omega - \left( 1 + \frac{p-1}{2} \pm \frac{p-1}{2} \right) (\phi_{\omega, \gamma}^{(L)})^{p-1}.
\end{equation}
By Weyl's theorem, the essential spectrum of $\mathcal{L}_{L, \pm}$ is given by $[\omega, \infty)$, and the discrete spectrum consists of isolated eigenvalues with finite multiplicity. It is a standard consequence of the positivity of ground states to conclude that $ker(\mathcal{L}_{L, -}) = span\{\phi_{\omega, \gamma}^{(L)}\}$ and $\mathcal{L}_{L, -}$ is non-negative. As the kernel is isolated, it follows that there exists $\delta >0$ such that 
\begin{equation}\label{eqL-1D}
    \langle \mathcal{L}_{L, -} u^I, u^I \rangle \geq \delta \|u^I\|_{L^2_{x, y}(\mathbb{S}_{L})}^2
\end{equation}
for all $u^I \in H^1(\mathbb{S}_{L})$ satisfying $ \langle  u^I, \phi_{\omega, \gamma}^{(L)}\rangle = 0$.

Next, by a Fourier expansion, any $u^R \in H^1(\mathbb{S}_{L})$ real-valued can be written as 
\begin{equation*}
    u^R(x, y) = u_0(x) + \sum_{n = 1}^{\infty} u_n(x) \cos(n \pi L^{-1}y)
\end{equation*}
where $(u_n)_{n \geq 0}\subset H^1(\mathbb{R})$ are real-valued functions. Then formally one has, with $\mathcal{L}_{+, 0}$ defined in \eqref{eq-01-b9}, 
\begin{equation}\label{eqL+11}
\mathcal{L}_{L, +} u^R  = \mathcal{L}_{+, 0} u_0 + \sum_{n = 1}^{\infty} \left( \mathcal{L}_{+, 0} + \left( \frac{n \pi}{L} \right)^2 \right) u_n(x) \cos(n \pi L^{-1}y).
\end{equation}

Since $\mathcal{L}_{+, 0}$ has exactly one negative eigenvalue, it follows that there exists $c >0$ such that $\langle \mathcal{L}_{+, 0} u_0, u_0 \rangle \geq c \|u_0\|_{L^2(\mathbb{R})}^2$ for all $u_0 \in H^1(\mathbb{R})$ satisfying $\langle u_0, \phi_{\omega, \gamma}^{(L)}\rangle = 0$ \cite{FuOhOz08,GrShSt87}. Moreover, one has $\langle (\mathcal{L}_{+, 0} + (n\pi L^{-1})^2) u_n, u_n \rangle \geq ((n\pi L^{-1})^2 - \lambda_1) \| u_n \|_{L^2(\mathbb{R})}^2 $ for all $n \geq 1$, where $-\lambda_1$ is the smallest eigenvalue of $\mathcal{L}_{+, 0}$. Therefore, for any $u^R \in H^1(\mathbb{S}_{L})$ real-valued satisfying $\langle u^R, \phi_{\omega, \gamma}^{(L)}\rangle = 0$,  we have $\langle u_0, \phi_{\omega, \gamma}^{(L)}\rangle = 0$ and thus, from \eqref{eqL+11}, it follows that

\begin{equation}\label{eqL+1C}
    \begin{aligned}
         \langle \mathcal{L}_{L, +} u^R, u^R \rangle & \geq c \| u_0 \|_{L^2(\mathbb{R})}^2 + \sum_{n = 1}^{\infty} ((n\pi L^{-1})^2 - \lambda_1) \| u_n \|_{L^2(\mathbb{R})}^2 \\
         &\geq \min\{c, ((\pi L^{-1})^2 - \lambda_1) \} \| u^R \|_{L^2_{x, y}(\mathbb{S}_{L})}^2.
    \end{aligned}
\end{equation}

Therefore, if $0 < L < L_{*} = \frac{\pi}{\sqrt{\lambda_1}}$, then for any $u \in H^1(\mathbb{S}_{L})$ satisfying $\langle u, \phi_{\omega, \gamma}^{(L)}\rangle = \langle u, i \phi_{\omega, \gamma}^{(L)}\rangle =  0$, one gets from \eqref{eqL-1D} and \eqref{eqL+1C} that
\begin{equation*}
    \langle S_{L, \gamma}^{\prime \prime}(\phi_{\omega, \gamma}^{(L)})u, u \rangle \geq C \|u\|_{L^2_{x, y}(\mathbb{S}_{L})}^2
\end{equation*}
for some $C >0$ independent of $u$. This implies that the line soliton is orbitally stable by the classical criterion in \cite{GrShSt87} and concludes the proof of the stability part of Theorem \ref{thm-sta}.

\subsection{Instability of the 
line soliton for $L > L_{*}$}

This section is devoted to the proof of Theorem \ref{thm-sta} {\rm (ii)}, that is, the instability of the line soliton for $L > L_{*}$. We fix $\gamma < 0$ and $\omega > \frac{\gamma^2}{4}$. We derive the equation for the perturbation term $\eta$ in 
\begin{equation} \label{eq-error}
u(t) =e^{i\omega t}\left(
\phi_{\omega, \gamma}^{(L)} + \eta(t)\right).
\end{equation} 
Plugging \eqref{eq-error} into \eqref{eqDeltaStrip}, we see that $\eta$ satisfies
\begin{equation*}
\begin{split}
i \partial_{t} \eta & = - H \eta - \omega \eta + |\phi_{\omega, \gamma}^{(L)} + \eta|^{p-1}(\phi_{\omega, \gamma}^{(L)} + \eta) - |\phi_{\omega, \gamma}^{(L)}|^{p-1} \phi_{\omega, \gamma}^{(L)}
\end{split}
\end{equation*}
We identify $\eta$ with $\vec{\eta} = (\eta^R, \eta^I)  : = (\text{Re}\; \eta, \text{Im}\; \eta) $ and write the equation above as
\begin{equation}
\label{WS-Ham-v}
\partial_t \vec{\eta}= J 
(\mathcal{S}(\phi_{\omega, \gamma}^{(L)}) \vec{\eta} 
+ N(\vec{\eta},\phi_{\omega, \gamma}^{(L)})),
\end{equation}
with 
\begin{equation} \label{complex}
J:=\begin{pmatrix}
0 &-1\\
1 &0
\end{pmatrix}, \quad \mathcal{S}(\phi_{\omega, \gamma}^{(L)}):=\begin{pmatrix}
\mathcal{L}_{L,+} &0 \\
0  & \mathcal{L}_{L,-}
\end{pmatrix},
\end{equation}
where $\mathcal{L}_{L, \pm}$ are defined in \eqref{eqL+-Def} and
\begin{equation} \label{linearized-nonlinear}
N(\vec{\eta},\phi_{\omega, \gamma}^{(L)}) 
= \begin{pmatrix}
- |\eta + \phi_{\omega, \gamma}^{(L)}|^{p-1}(\eta^R + \phi_{\omega, \gamma}^{(L)}) + p|
\phi_{\omega, \gamma}^{(L)}|^{p-1}\eta^R +|\phi_{\omega, \gamma}^{(L)}|^{p-1}
\phi_{\omega, \gamma}^{(L)}\\
-|\eta + 
\phi_{\omega, \gamma}^{(L)}|^{p-1} \eta^I + 
|\phi_{\omega, \gamma}^{(L)}|^{p-1}\eta^I
\end{pmatrix},
\end{equation}

\subsubsection{Linear instability}

Following the same arguments of \S\ref{secStabBor}, one verifies that for $L > L_{*}$, the operator $\mathcal{L}_{L,+}$ has at least two negative eigenvalues, $-\lambda_1$ and $(\pi L^{-1})^2 - \lambda_1$, with associated eigenfunctions $\chi$ and $\chi \cos(\pi L^{-1} y)$, respectively. In particular,  $\mathcal{L}_{L,+}$ is not a positive operator on the space $\mathcal{D}(H) \cap \{ f \in L^2(\mathbb{S}) : (f,\phi_{\omega,\gamma}^{(L)}) = 0 \}$. In this case, it is classical to show that $J \mathcal {S}(\phi_{\omega, \gamma}^{(L)})$ has a real-valued positive eigenvalue, which implies the linear instability of the line soliton. This result can be found in a general framework, for instance, in \cite[Theorem $7.1.16$]{KaPr13Sp} or the specific case of transverse instability in \cite[Theorem $1.1$]{RoTz10Inst}. We give a proof for the sake of completeness in Appendix \ref{AppenLinStab}.

\subsubsection{Linear instability implies 
nonlinear instability}
For each $k \in \mathbb{N} \cup \{0\}$, 
define an orthogonal projection $P_{\leq k}$ as
\begin{equation} \label{projection}
P_{\leq k} u(x,y) = \sum_{n= 0}^k u_n(x) \theta_{n}(y), \quad (x,y) \in \mathbb{S}_{L}, 
\end{equation}
where
\begin{equation*}
u(x,y) = \sum_{n=-\infty}^{\infty} u_n(x)
\theta_{n}(y), \qquad 
u_{n}(x) = \sqrt{\frac{2}{L}} 
\int_{0}^{L} u(x, y) 
\theta_{n}(y)dy. 
\end{equation*}
From Lemma \ref{lem:posi-eigen}, 
we see that there exists a
positive eigenvalue of 
$J 
(\mathcal{S}(\phi_{\omega, \gamma}^{(L)})$ 
for $L >  L_{*}$. 
Thus, we can define 
\begin{equation} \label{spectral-radius}
\lambda_{0} := \max \left\{ \lambda >0 \colon\;
\lambda \in \sigma_{\text{disc}} 
(J 
(\mathcal{S}(\phi_{\omega, \gamma}^{(L)})) \right\}. 
\end{equation}
We can estimate the semigroup for the low-frequency part. 
\begin{lem}
\label{lem-3-2}
For a positive integer $k$ and $\varepsilon > 0$, there exists $C_{k, \varepsilon}>0$ 
such that
\begin{equation} \label{semi-est}
\left\|e^{-tJ 
(\mathcal{S}(\phi_{\omega, \gamma}^{(L)})}P_{\leq k}\eta
\right\|_{L_{x, y}^2(\mathbb{S}_{L})} \leq 
C_{k, \varepsilon} e^{(\lambda_{0} + \varepsilon)t}\|\eta\|_{L^2_{x, y}(\mathbb{S}_{L})}, 
\quad t>0, \ \eta \in 
L_{x, y}^2(\mathbb{S}_{L},\mathbb{C}). 
\end{equation}
\end{lem}
\begin{proof}
Let 
\begin{equation*}
- J S(a) = 
\begin{pmatrix}
0 & - \partial_{x}^{2} + \omega + 
\gamma \delta_0 + \left(\frac{a \pi}{L} \right)^2 
- \phi_{\omega, \gamma}^{p-1} \\
\partial_{x}^{2} - \omega - \gamma \delta_0
- \left(\frac{a \pi}{L} \right)^2 
+ p \phi_{\omega, \gamma}^{p-1} & 0
\end{pmatrix}
\end{equation*}
We follow the approach of \cite[Proposition $10$]{MR2916078}. First, we need to show the spectral mapping property  
\begin{equation} \label{spectral-mapping}
\sigma(e^{- JS(a)}) = e^{\sigma(- JS(a))}. 
\end{equation}
This is done in Appendix \ref{sec:spectral-m} (see also \cite{GeJoLaSt00} for the case $\gamma =0$). 
In addition, we see that 
\begin{equation} \label{eq-linearizedop}
J\mathcal{S}(\phi_{\omega, \gamma}^{(L)})
= \bigoplus_{n \in \mathbb{N} \cup 
\{0\}} J S_{L, \gamma}^{\prime \prime}(a),
\end{equation}
where, for $a \in \R$,
\begin{equation} \label{linerized-ope-1}
\begin{split}
 S_{L, \gamma}^{\prime \prime}(a)  & := \begin{pmatrix}
\mathcal{L}_{L, +, a}& 0 \\
0 & \mathcal{L}_{L, -, a} 
\end{pmatrix} \\
& =
\begin{pmatrix}
- \partial_{xx} + \omega + \gamma \delta_0 
+ \left(\frac{a \pi}{L} \right)^2 
- p (\phi_{\omega, \gamma}^{(L)})^{p-1} 
& 0 \\
0 & - \partial_{xx} + \omega + \gamma \delta_0 
+ \left(\frac{a \pi}{L} \right)^2 
- (\phi_{\omega, \gamma}^{(L)})^{p-1}
\end{pmatrix}.
\end{split}
\end{equation}
These together with 
\eqref{spectral-radius} imply that 
the spectral radius of $e^{JS(n/2\pi)}$ is 
less than or equal to $e^{\lambda_{0}}$, 
where $\lambda_0$ is given by 
\eqref{spectral-radius}. 
Therefore, by applying 
Lemmas 2 and 3 in \cite{Shatah-Strauss2}, 
for any $\varepsilon>0$ and each $n \in \mathbb{N}$, 
there exists $C_{n, \varepsilon}>0$ such that 
\begin{equation} \label{spectral-mapping-eq1}
\|e^{- t J S(n/2\pi)}\eta\|_{L_{x}^{2}(\R)} \leq C_{n, \varepsilon} 
e^{(\lambda_{0} + \varepsilon) t} \|\eta\|_{L_{x}^{2}(\R)} \qquad 
\mbox{for all $\eta \in L_{x}^{2}(\mathbb{R})$}. 
\end{equation}
\noindent
Hence, for $t >0$ and 
$\eta \in L_{x, y}^{2}(\mathbb{S}_{L})$, we have 
\begin{equation*}
\|e^{- t J S_{L, \gamma}^{\prime \prime}(\phi_{\omega, \gamma})} 
P_{\leq k} \eta\|_{L^{2}_{x, y}(\mathbb{S}_{L})} 
\leq \left\|\sum_{n = -k}^{k} 
e^{- t J S(n)} \eta_n \theta_{n}(y) \right\|_{L_{x, y}^{2}(\mathbb{S}_{L})} 
\leq C_{k, \varepsilon} e^{(\lambda_{0} + \varepsilon) t} \|\eta\|_{L_{x, y}^{2}
(\mathbb{S}_{L})}, 
\end{equation*}
where 
\begin{equation*}
\eta(x, y) = \sum_{n = 0}^{\infty} \eta_n(x) 
\theta_{n}(y). 
\end{equation*}
This completes the proof. 
\end{proof}
Next, we control the $L^2$ norm of the low frequencies 
for the nonlinear term by the energy norm.

\begin{lem}
\label{lem-3-1}
For $k>1$, there exists $C>0$ such that
\begin{equation}\label{nonlinear-est3}
\|N(\eta,\phi_{\omega, \gamma}^{(L)})\|_{L^2_{x,y}(\mathbb{S}_{L})} 
\leq 
\begin{cases}
C\|\eta\|_{H^1}^{p} \qquad &\mbox{if $1<p\leq 2$},\\
C (\|\eta\|_{H^1}^2+\|\eta\|_{H^1}^{p}) \qquad 
&\mbox{if $2<p\leq 5$},
\end{cases}
\end{equation}
$N(\eta,\phi_{\omega, \gamma})$ is given by \eqref{linearized-nonlinear} and $\eta\in H^1(\mathbb{S}_{L})$.
\end{lem}

\begin{proof}
We shall split the proof into two cases according to the value of $p$.

\textbf{(Case 1).} $1<p\leq 2$. 
 
From \eqref{linearized-nonlinear}, we obtain 
\begin{equation*}
\begin{split}
N(\eta, \phi_{\omega, \gamma})
& = \int_{0}^{1} \frac{d}{d \theta}
(f(\theta \eta + \phi_{\omega, \gamma})) d \theta 
- D f(\phi_{\omega, \gamma}) \eta \\
& = \int_0^1 
(D f(\theta \eta + \phi_{\omega, \gamma}) - D f(\phi_{\omega, \gamma}))\eta(x,y)d\theta.
\end{split}
\end{equation*}
From the result of \cite[Lemma 2.4]{ginibrle-velo}, we have
\begin{equation*}
||a|^{p-1}-|b|^{p-1}|\leq |a-b|^{p-1},
\end{equation*}
which together with \eqref{linearized-nonlinear} yields
\begin{equation*}
|N(\eta, \phi_{\omega, \gamma})| \leq |\eta(x,y)|^{p} 
\qquad \mbox{for $1 < p \leq 2$}.
\end{equation*}
Hence, we have
\begin{equation*}
\|N(\eta,\phi_{\omega, \gamma})\|_{L^2_{x, y}(\mathbb{S}_{L})} \leq 
\|\eta\|_{L^{2p}(\mathbb{S}_{L})}^p.
\end{equation*}
On the other hand, since $H^1(\mathbb{S}_{L})  \hookrightarrow L^q
(\mathbb{S}_{L})$ for $2\leq q < \infty$, 
we obtain 
\begin{equation*}
\|N(\eta,\phi_{\omega, \gamma})\|_{L^2_{x,y}(\mathbb{S}_{L})} \lesssim \|\eta\|_{H^1}^p.
\end{equation*}
\par
\textbf{(Case 2).} $2 < p \leq 5$.
We put 
\begin{equation*}
F(s) := f(s \eta + \phi_{\omega, \gamma}),  
\end{equation*}
where $f(z) = |z|^{p-1}z$ for $s > 0$. 
We note that $F$ is $C^{2}((0, \infty), \mathbb{R})$ 
is $C^{2}$ if $p > 2$. 
From \eqref{linearized-nonlinear}, we can write 
\begin{equation*}
N(\eta,\phi_{\omega, \gamma})= \int^1_0 (1-s) F^{\prime \prime}(s) ds.
\end{equation*} 
By the definition of $F(s)$, we have
\begin{equation*}
\begin{split}
F^{\prime \prime}(s)
= & \frac{p^{2} -1}{4} |s \eta + 
\phi_{\omega, \gamma}^{(L)}|^{p-3} 
(s\overline{\eta} 
+ \phi_{\omega, \gamma}^{(L)})\eta^{2} 
+ \frac{p^{2} -1}{2} |s \eta + 
\phi_{\omega, \gamma}^{(L)}|^{p-3}
(s \eta + \phi_{\omega, \gamma}^{(L)}) |\eta|^{2} \\
& + \frac{(p-1)(p-3)}{4} |s \eta + 
\phi_{\omega, \gamma}^{(L)}|^{p-5}
(s \eta + \phi_{\omega, \gamma}^{(L)})^{3}
(\overline{\eta})^{2}. 
\end{split}
\end{equation*}
Thus, since $p>2$, we obtain
\begin{equation*}
F^{\prime \prime}(s) \lesssim |\eta|^2 |s \eta + 
\phi_{\omega, \gamma}|^{p-2} \lesssim |\eta|^2 \left(|\eta|^{p-2} + |\phi_{\omega, \gamma}|^{p-2}\right).
\end{equation*}
As a consequence, from the fact that 
$\|\phi_{\omega, \gamma}\|_{L^\infty_x(\R)} \lesssim 1$, we get
\begin{equation*}
|N(\eta,\phi_{\omega, \gamma})| \lesssim |\eta|^2 + |\eta|^p.
\end{equation*} 
Then, we obtain 
\begin{equation}
\label{estim-NL-L2-L1}
\| N(\eta,\phi_{\omega, 
\gamma}^{(L)})\|_{L^2_{x, y}(\mathbb{S}_{L})} 
\lesssim \|\eta\|_{L^4(\mathbb{S}_{L})}^2 + 
\|\eta\|_{L^{2p}(\mathbb{S}_{L})}^p 
\lesssim \|\eta\|_{H^1}^2 + 
\|\eta\|_{H^1}^p.
\end{equation}
From \eqref{estim-NL-L2-L1},  
we have obtained \eqref{nonlinear-est3} 
in this case. This completes the proof.
\end{proof}

Lemma \ref{lem:posi-eigen} below 
guarantees that there exists 
a positive eigenvalue $\lambda_{0}$ 
of the operator $J \mathcal{S}^{\prime \prime}_{\omega, \gamma}(\phi_{\omega, \gamma})$ 
and the corresponding eigenfunction $\psi \in H^{2}(\mathbb{S}_{L}, \mathbb{C})$ with
 $\|\psi\|_{L_{x, y}^2(\mathbb{S}_{L})}=1$ is written as
\begin{equation}
\label{def:chi_0}
\psi(x,y)=\psi_0(x) \theta_{k_0}(y) \quad \mbox{with 
$k_{0} \in \mathbb{N}$, $\psi_0 \in H_{x}^2(\mathbb{R},\mathbb{C})$}.
\end{equation}
For  $\delta>0$, let $u_{\delta}(t)$ be the solution of 
(\ref{eqDeltaStrip}) 
with initial data $u_{\delta}|_{t=0} =\phi_{\omega, \gamma}+\delta\psi$ and $\eta_{\delta}(t)$ be the solution of (\ref{WS-Ham-v}) 
having an initial data $\eta_{\delta}|_{t=0}=\delta \psi$.
Then, we infer that $u_{\delta}(t)=e^{i\omega t}(\phi_{\omega, \gamma}+\eta_{\delta}(t))$. 
We have the following:
\begin{lem}\label{lem-3-3}
There exist a positive integer $K_{L}$ and 
positive constant $C$ such that for $\delta>0$ and $t>0$, 
\begin{equation*}
\|\eta_{\delta}(t)\|_{H^1} \leq 
C\|P_{\leq K_{L}}\eta_{\delta}(t)\|
_{L^2_{x, y}(\mathbb{S}_{L})} + o(\delta).
\end{equation*}
\end{lem}
\begin{proof}
Using the Taylor expansion and the fact that 
 $\mathcal{S}^{\prime}_{\omega}(\phi_{\omega, \gamma})=0$, for $\eta_{\delta} \in H^1$ and any fixed $t>0,$ we write
\begin{equation} \label{eq1-lem3-3}
\begin{split}
S_{L, \gamma}(u_{\delta}(t))
& = S_{L, \gamma}(\phi_{\omega, \gamma} + \eta_{\delta}(t)) \\
& =   
S_{L, \gamma}(\phi_{\omega, \gamma}) + 
\langle S_{L, \gamma}^{\prime}(\phi_{\omega, \gamma}), \eta_{\delta}(t)\rangle_{H^{-1}, H^1} 
+ \frac{1}{2} \langle S_{L, \gamma}^{\prime \prime}(\phi_{\omega, \gamma})\eta_{\delta}(t), 
\eta_{\delta}(t)\rangle_{H^{-1}, H^1} \\ 
& \quad+ o(\|\eta_{\delta}(t)\|_{H^1}^{2}) \\
& = 
S_{L, \gamma}(\phi_{\omega, \gamma}) 
+ \frac{1}{2} \langle S_{L, \gamma}^{\prime \prime}(\phi_{\omega, \gamma})\eta_{\delta}(t), 
\eta_{\delta}(t)\rangle_{H^{-1}, H^1}+o(\|\eta_{\delta}(t)\|_{H^1}^{2}), 
\end{split}
\end{equation}
and, at $t=0$,
\begin{equation*}
\begin{split} 
S_{L, \gamma}(\phi_{\omega, \gamma}+\delta \psi)
& = S_{L, \gamma}(\phi_{\omega, \gamma}) + 
\delta \langle S_{L, \gamma}^{\prime}(\phi_{\omega, \gamma}), \psi\rangle_{H^{-1}, H^1} 
+ \frac{\delta^{2}}{2}
\langle S_{L, \gamma}^{\prime \prime}(\phi_{\omega, \gamma})\psi,\psi\rangle_{H^{-1}, H^1} +o(\delta^{2}) \\
& = S_{L, \gamma}(\phi_{\omega, \gamma}) + \frac{\delta^{2}}{2}
\langle S_{L, \gamma}^{\prime \prime}(\phi_{\omega, \gamma})\psi,\psi\rangle_{H^{-1}, H^1}+o(\delta^{2}). 
\end{split}
\end{equation*}
From the conservation laws of the mass and the 
energy (see Theorem \ref{thmLWP}), we have 
\begin{equation} \label{eq2-lem3-3}
S_{L, \gamma}(u_{\delta}(t)) 
= S_{L, \gamma}(\phi_{\omega, \gamma}+\eta_{\delta}(t))
= S_{L, \gamma}(\phi_{\omega, \gamma}+\delta \psi)
\end{equation}
for any $t\geq 0$.
Moreover, 
since $J$ is a skew-symmetric operator, 
we see that
\begin{equation} \label{eq3-lem3-3}
\langle S_{L, \gamma}^{\prime \prime}(\phi_{\omega, \gamma})\psi,\psi\rangle_{H^{-1}, H^1} = 
(JS_{L, \gamma}^{\prime \prime}
(\phi_{\omega, \gamma})\psi,J^{-1}\psi\rangle_{H^{-1}, H^1}=(\lambda_0 \psi, J^{-1}\psi)_{L_{x, y}^2(\mathbb{S}_{L})}=0. 
\end{equation}
Thus, from \eqref{eq1-lem3-3}--\eqref{eq3-lem3-3}, 
we infer that
\begin{equation}\label{eq4-lem3-3}
\langle S_{L, \gamma}^{\prime \prime}
(\phi_{\omega, \gamma})\eta_{\delta}(t), 
\eta_{\delta}(t)\rangle_{H^{-1}, H^1} 
= o(\|\eta_{\delta}(t)\|_{H^1}^2)+o(\delta^2).
\end{equation}
Let 
\[
K_{L}= \max\left\{ 
k \in \mathbb{Z} \colon k \leq 
1 + \frac{L}{L_{*}}
\right\}.
\]
Since $S_{L, \gamma}^{\prime \prime}(n)$ is also positive 
for $|n| > 1 + \frac{L}{L_*}$, 
we see that 
$S_{L, \gamma}^{\prime \prime}(\phi_{\omega, \gamma})P_{> K_{L}}$ is positive, 
that is, there exists a constant $C_{0} > 0$ such that 
\begin{equation}\label{eq6-lem3-3}
\langle S_{L, \gamma}^{\prime \prime}(\phi_{\omega, \gamma})
P_{> K_{L}}\eta_{\delta}(t), P_{> K_{L}}\eta_{\delta}(t) \rangle _{H^{-1}, H^1} 
\geq C_{0} \|P_{> K_{L}}\eta_{\delta}(t)\|_{H^1}^{2}. 
\end{equation}
By the definition of $S_{L, \gamma}^{\prime \prime}(n)$, 
we can take constants $C_{1} \in (0, C_{0})$ and 
$C_{2} > 0$ such that
\begin{equation} \label{eq5-lem3-3}
\langle S_{L, \gamma}^{\prime \prime}(
\phi_{\omega, \gamma})P_{\leq K_{L}}\eta_{\delta}(t), 
P_{\leq K_{L}}\eta_{\delta}(t) 
\rangle_{H^{-1}, H^1}  \geq 
C_{1} \|P_{\leq K_{L}}\eta_{\delta}(t)\|_{H_{x}^1(\R)}^2 
- C_{2}\|P_{\leq K_{L}}\eta_{\delta}(t)\|_{L_{x}^2(\R)}^2.
\end{equation}
Thus, from \eqref{eq4-lem3-3}, \eqref{eq6-lem3-3} 
and \eqref{eq5-lem3-3}, 
we obtain 
\begin{equation*} 
\begin{split}
\|\eta_{\delta}(t)\|_{H^1}^2
& =
\|P_{\leq K_{L}}\eta_{\delta}(t)\|_{H^1}^2 + \|P_{> K_{L}}\eta_{\delta}(t)\|_{H^1}^2\\
&\leq \frac{1}{C_{1}} \langle S_{L, \gamma}^{\prime \prime}(\phi_{\omega, \gamma})
P_{\leq K_{L}}\eta_{\delta}(t),P_{\leq K_{L}}\eta_{\delta}(t)\rangle_{H^{-1}, H^1} 
+ \frac{C_{2}}{C_{1}} 
\|P_{\leq K_{L}}\eta_{\delta}(t)\|_{L^{2}_{x,y}
(\mathbb{S}_{L})}^2\\
& \quad + \frac{1}{C_{0}} \langle S_{L, \gamma}^{\prime \prime}
(\phi_{\omega, \gamma})P_{> K_{L}}\eta_{\delta}(t),P_{> K_{L}}\eta_{\delta}(t)\rangle_{H^{-1}, H^1} 
\\
& \leq 
 \frac{C_{2}}{C_{1}} 
 \|P_{\leq K_{L}}\eta_{\delta}(t)\|_{L^{2}_{x,y}
 (\mathbb{S}_{L})}^2 
+ \frac{1}{C_{1}} \langle S_{L, \gamma}^{\prime \prime}(\phi_{\omega, \gamma})
\eta_{\delta}(t), \eta_{\delta}(t)\rangle_{H^{-1}, H^1} \\
&\leq \frac{C_{2}}{C_{1}} \|P_{\leq K_{L}}\eta_{\delta}(t)\|_{L^2_{x,y}(\mathbb{S}_{L})}^2 
+ o(\delta^2) + o(\|\eta_{\delta}(t)\|_{H^1}^2), 
\end{split}
\end{equation*} 
where we have used the fact that $0 < C_{1} < C_{0}$  
in the third inequality and 
the orthogonality between $P_{\leq K_{L}}$ and $S_{L, \gamma}^{\prime \prime}(\phi_{\omega, \gamma})P_{> K_{L}}=P_{> K_{L}}S_{L, \gamma}^{\prime \prime}(\phi_{\omega, \gamma})$. This finishes the proof of Lemma \ref{lem-3-3}.
\end{proof}
Now, we give the proof of Theorem \ref{thm-sta} (ii).

\begin{proof}[Proof of Theorem \ref{thm-sta} (ii)] 
From the Duhamel formula associated to \eqref{WS-Ham-v}, 
$\eta_{\delta}$ satisfies the following: 
\begin{equation} \label{eq1-eta}
\begin{split}
\eta_{\delta}(t) 
& = e^{J S_{L, \gamma}^{\prime \prime}(\phi_{\omega, \gamma})t} \delta \psi 
- J \int_{0}^{t} e^{- (t-s) J S_{L, \gamma}^{\prime \prime}(\phi_{\omega, \gamma})}
N(\eta_{\delta}(s), \phi_{\omega, \gamma}) ds \\
& = \delta e^{\lambda_{0}t} \psi 
- J \int_{0}^{t} e^{- (t-s) J S_{L, \gamma}^{\prime \prime}(\phi_{\omega, \gamma})}
N(\eta_{\delta}(s), \phi_{\omega, \gamma}) ds. 
\end{split}
\end{equation}

Using Lemmas \ref{lem-3-3}, \ref{lem-3-2} and \ref{lem-3-1} 
in this order, we estimate the $H^1$ norm of $\eta_{\delta}$ 
in the following way
\begin{equation*}
\begin{split}
\|\eta_{\delta}(t)\|_{H^1} 
& \lesssim \delta e^{\lambda_{0} t} \|\psi\|_{H^1} 
+\int_{0}^{t} \|e^{- (t-s) J S_{L, \gamma}^{\prime \prime}(\phi_{\omega, \gamma})}
P_{\leq K_{L}} N
L(\eta_{\delta}(s), \phi_{\omega, \gamma})\|_{L_{x, y}^{2}(\mathbb{S}_{L})}ds 
+ o(\delta) \\
&\lesssim \delta e^{\lambda_0t}
\|\psi\|_{L_{x, y}^2(\mathbb{S}_{L})} 
+ \int_0^t e^{(\lambda_{0} + \varepsilon)(t-s)}
\|P_{\leq K_{L}}
N(\eta_{\delta}(s),\phi_{\omega, \gamma})\|_{L_{x, y}^2(\mathbb{S}_{L})} ds
 + o(\delta) \\
& \lesssim \delta e^{\lambda_0 t} + \int_0^t e^{(1+\varepsilon _0)\lambda_0(t-s)}(\|\eta_{\delta}(s)\|_{H^1}^2+\|\eta_{\delta}(s)\|_{H^1}^p)ds,
\end{split}
\end{equation*}
where $\varepsilon_0=\varepsilon/\lambda_0$ and $\varepsilon>0$ given in Lemma \ref{lem-3-2}. Thus, there exist constants $C_{1} > 1$ and $C_{2} >0$ such that 
\begin{equation*}
\|\eta_{\delta}(t)\|_{H^1} 
\leq C_{1} \delta e^{\lambda_0 t} + 
C_{2} \int_0^t e^{(1+\varepsilon _0)\lambda_0(t-s)}
(\|\eta_{\delta}(s)\|_{H^1}^2+\|\eta_{\delta}(s)\|_{H^1}^p)ds.
\end{equation*} 
We shall show that for sufficiently small 
$\delta>0$ and $\varepsilon_0 >0$, we have 
\begin{equation} \label{up-bound}
\|\eta_{\delta}(t)\|_{H^1} \leq 2 C_{1} \delta e^{\lambda_0 t} \quad 
\mbox{for $t \in [0,T_{\varepsilon_0 ,\delta}]$},
\end{equation}
where
\begin{equation*}
T_{\varepsilon _0,\delta}=\frac{\log (\varepsilon _0/\delta)}{\lambda_0}. 
\end{equation*}
We put 
\begin{equation*}
T_{*} = \sup\{t > 0 \colon \|\eta_{\delta}(t)\|_{H^1} \leq 2 C_{1} \delta e^{\lambda_{0}t} 
\}. 
\end{equation*}
Since $\|\eta_{\delta}(0)\|_{H^1} = \delta (< 2C_{1} \delta)$, 
we see that $T_{*} > 0$. 
In order to prove \eqref{up-bound}, 
it is enough to show that $T_{\varepsilon, \delta} \leq T_{*}$. 
Suppose to the contrary that $T_{*} < T_{\varepsilon, \delta}$. 
Then, we have 
\begin{equation} \label{nonlinear-est1}
\begin{split}
2 C_{1} \delta e^{\lambda_{0} T_{*}} 
= \|\eta_{\delta}(T_{*})\|_{H^1} 
\leq C_{1} \delta e^{\lambda_{0} T_{*}} 
+ \int_{0}^{T_{*}} e^{(1 + \varepsilon_{0}) \lambda_{0} (T_{*} - s)} 
(\|\eta_{\delta}(s)\|_{H^1}^{2} + \|\eta_{\delta}(s)\|_{H^1}^{p}) ds.  
\end{split}
\end{equation}
Note that for $0 < s < T_{*} (< T_{\varepsilon_{0}, \delta})$, 
we obtain 
\begin{equation} \label{nonlinear-est2}
2 C_{1} \delta e^{\lambda_{0} s} < 
2 C_{1} \delta e^{\lambda_{0} T_{\varepsilon_{0}, \delta}} = 2 C_{1} \varepsilon_{0} \ll 1. 
\end{equation}
It follows from \eqref{nonlinear-est1} 
and \eqref{nonlinear-est2} 
that
\begin{equation*}
\begin{split}
C_{1} \delta e^{\lambda_{0} T_{*}} 
& \leq 
\int_{0}^{T_{*}} e^{(1 + \varepsilon_{0}) \lambda_{0} (T_{*} - s)} 
\left[(2C_{1} \delta)^{2} e^{2 \lambda_{0} s} 
+ (2C_{1} \delta)^{p} e^{p \lambda_{0} s}\right]ds \\
& \leq 8 C_{1}^{2} \delta^{2} e^{(1 + \varepsilon_{0})\lambda_{0} T_{*}} 
\int_{0}^{T_{*}} e^{(1 - \varepsilon_{0}) \lambda_{0} s}ds \\
& = 8 C_{1}^{2} \delta^{2} e^{(1 + \varepsilon_{0})\lambda_{0} T_{*}} 
\frac{e^{(1-\varepsilon_{0}) \lambda_{0} T_{*}}}{(1-\varepsilon_{0}) \lambda_{0}} \\
& = \frac{8 C_{1}^{2} \delta^{2}}{(1-\varepsilon_{0}) \lambda_{0}} e^{2 \lambda_{0}T_{*}}. 
\end{split}
\end{equation*}
This together with $T_{*} < T_{\varepsilon, \delta}$ 
yields that 
\begin{equation*}
1 \leq 
\frac{8 C_{1}^{2} \delta}{(1 - \varepsilon_{0}) \lambda_{0}} 
e^{\lambda_{0} T_{\varepsilon, \delta}} 
= \frac{8 C_{1}^{2} \delta}{(1 - \varepsilon_{0}) \lambda_{0}} 
\times \frac{\varepsilon_{0}}{\delta} 
= \frac{8 C_{1}^{2}}{(1 - \varepsilon_{0}) \lambda_{0}} 
\varepsilon_{0} < 1
\end{equation*}
for sufficiently small $\varepsilon_{0} > 0$, 
which is a contradiction.  
Thus, \eqref{up-bound} holds. 

Since $(e^{J S_{L, \gamma}^{\prime \prime} (\phi_{\omega, \gamma})})^{*} 
= e^{(J S_{L, \gamma}^{\prime \prime}(\phi_{\omega, \gamma}))^{*}}$ and 
\begin{equation*}
(J S_{L, \gamma}^{\prime \prime}(\phi_{\omega, \gamma}))^{*} J \psi 
= S_{L, \gamma}^{\prime \prime}(\phi_{\omega, \gamma}) J^{*} J \psi
= - S_{L, \gamma}^{\prime \prime}(\phi_{\omega, \gamma}) \psi
= J(J S_{L, \gamma}^{\prime \prime}(\phi_{\omega, \gamma}) \psi) 
= \lambda_{0} J \psi, 
\end{equation*}
we have by \eqref{eq1-eta} that
\begin{equation} \label{eq1-proof-insta}
\begin{split}
& \quad |(\psi,\eta_{\delta}(T_{\varepsilon _0,\delta}))
_{L_{x, y}^2(\mathbb{S}_{L})}| \\
&=\left| \delta e^{\lambda_0 T_{\varepsilon _0, \delta}} + 
\int_0^{T_{\varepsilon _0,\delta}}(\psi, - J e^{-(T_{\varepsilon _0,\delta}-s)
J S_{L, \gamma}^{\prime \prime}(\phi_{\omega, \gamma})}
N(\eta_{\delta}(s),\phi_{\omega, \gamma}))_{L^2}ds \right|\\
&\geq \varepsilon _0 -C \int_0^{T_{\varepsilon _0,\delta}} e^{(T_{\varepsilon _0,\delta}-s)\lambda_0}(\|\eta_{\delta}(s)\|_{H^1}^2+\|\eta_{\delta}(s)\|_{H^1}^p)ds\\
& \geq  \varepsilon _0 -C \int_0^{T_{\varepsilon _0,\delta}} e^{(T_{\varepsilon _0,\delta}-s)\lambda_0} (\delta e^{ \lambda_{0} s})^{\min\{2, p\}} ds \\	 
& \geq \varepsilon_0 - C \varepsilon_0 ^{\min \{p,2\}} \\
& \geq \frac{\varepsilon_{0}}{2}.
\end{split}
\end{equation}
From $P_{\leq 0} \phi_{\omega, \gamma}=\phi_{\omega, \gamma}$ and \eqref{eq-error},
there exists $\varepsilon _0 >0$ such that 
for $\varepsilon _0 > \delta >0$ and $\theta \in \R$, 
we have 
\begin{equation*}
\begin{split}
\|u_{\delta}(T_{\varepsilon _0,\delta})-e^{i\theta}\phi_{\omega, \gamma}\|_{L^2_{x, y}(\mathbb{S}_{L})}
&\geq \|(I-P_{\leq 0}) (u_{\delta}(T_{\varepsilon _0,\delta})-e^{i\theta}\phi_{\omega, \gamma})\|_{L^2_{x, y}(\mathbb{S}_{L})}\\
&= \|(I-P_{\leq 0})e^{-i\omega T_{\varepsilon _0,\delta}}u_{\delta}(T_{\varepsilon _0,\delta})\|_{L^2_{x, y}(\mathbb{S}_{L})}\\
&= \|(I-P_{\leq 0})(e^{-i\omega T_{\varepsilon _0,\delta}}u_{\delta}(T_{\varepsilon _0,\delta})-\phi_{\omega, \gamma})
\|_{L^2_{x, y}(\mathbb{S}_{L})} \\
& \geq \|(P_{\leq k_{0}} - P_{\leq 0})(e^{-i\omega T_{\varepsilon _0, \delta}}u_{\delta}(T_{\varepsilon _0,\delta})-\phi_{\omega, \gamma})
\|_{L^2_{x, y}(\mathbb{S}_{L})} \\
& = \|(P_{\leq k_{0}} - P_{\leq k_{0}-1})\eta_{\delta}(T_{\varepsilon_{0}, \delta})\|_{L^2_{x, y}(\mathbb{S}_{L})}.
\end{split}
\end{equation*}
Let  $v_{k_{0}}$ be the $k_{0}$-th Fourier mode of 
$\eta_{\delta}(T_{\varepsilon_{0}, \delta})$ on the $y$ variable.
From \eqref{def:chi_0} and 
the Cauchy-Schwarz inequality, we have
\begin{equation}
(\psi,\eta)_{L^2_{x,y}(\mathbb{S}_{L})} 
= (e^{i k_{0} y}\psi_0,\eta)_{L^2_{x,y}(\mathbb{S}_{L})}= 
(\psi_0,v_{k_{0}})_{L^2_{x}} 
\leq \|\psi_0\|_{L^2_{x}(\R)} \|v_{k_{0}}\|_{L^2_{x}(\R)}
\end{equation} 
for $\eta \in L^2_{x, y}(\mathbb{S}_{L})$. Thus, we have 
\begin{equation}\label{eq2-proof-insta}
\|(P_{\leq k_{0}}-P_{\leq k_{0}-1})\eta\|_{L^2_{x,y}(\mathbb{S}_{L})} = 
\|v_{k_{0}}\|_{L^2_{x}(\R)} \geq \|\psi_0\|_{L^2_{x}(\R)}^{-1} 
|(\psi,\eta)_{L^2_{x,y}(\mathbb{S}_{L})}| 
\end{equation}
for $\eta \in L^2_{x, y}(\mathbb{S}_{L})$. It follows from \eqref{eq1-proof-insta}--\eqref{eq2-proof-insta} that 
\begin{equation*}
\begin{split}
\|u_{\delta}(T_{\varepsilon _0,\delta})-e^{i\theta}\phi_{\omega, \gamma}\|_{L^2_{x,y}(\mathbb{S}_{L})} 
&\geq \|(P_{\leq 1}-P_{\leq 0})v_\delta(T_{\varepsilon _0,\delta})\|_{L^2_{x,y}
(\mathbb{S}_{L})} \\
&\gtrsim |(\psi,\eta_{\delta}(T_{\varepsilon _0,\delta}))_{L^2_{x,y}(\mathbb{S}_{L})}| 
\geq \frac{\varepsilon _0}{2}.
\end{split}
\end{equation*}
This implies that the standing wave $e^{i\omega t}\phi_{\omega, \gamma}$ is unstable.
\end{proof}

\subsection{Stability
of the bifurcation soliton $\phi(a)$}
In this subsection, we study the stability of 
the bifurcation soliton $\varphi(a)$ found in $\S\ref{secBifurcation}$
for sufficiently small $|a| > 0$. 

Using Lemma \ref{lemStabil} (1), 
we first show the following: 
\begin{prop}\label{prop1:critical}
    Let $\gamma < 0$, $\omega > \gamma^2/2$ 
 and $2 \leq p < 5$. 
 Then, we have 
    \[
    \partial_{\omega} 
    \|\varphi(a)\|_{L^{2}(\mathbb{S})}^2 
    > 0. 
    \]
for sufficiently small $|a| > 0$, where $\phi(a)$ is defined in 
\end{prop}
To prove Proposition \ref{prop1:critical}, 
we need the following: 
\begin{lem} \label{lem1:critical}
Let $\lambda_{2}(a)$ be the second eigenvalue of 
$\mathcal{L}_{+}(a, L(a))$. 
Then, we have 
\begin{equation} \label{eq1-bifur}
\|\varphi(a)\|_{L^{2}}^{2} = 
\|\phi_{\omega, \gamma}\|_{L^{2}}^{2} 
+ a^{2}  \frac{d \lambda_{2}}{d \omega}(0) 
+ \|h(a)\|_{L^{2}}^{2}. 
\end{equation}
\end{lem}
\begin{proof}
Observe from \eqref{eq59-bi} that 
\begin{equation} \label{eq1-varphi}
\varphi(a) = \phi_{\omega, \gamma} + a \xi 
+ h(a), \qquad 
h(a) = h(0) + \frac{\partial h}{\partial a} (0)a 
+ \frac{1}{2} \frac{\partial^{2} h}{\partial a^{2}}(0) a^{2} + o(a^{2}). 
\end{equation}
It follows from $h(a) = h(a, L(a))$ that 
    \[
    \begin{split}
    & \frac{\partial h}{\partial a} (a) 
    = \frac{\partial h}{\partial a} (a, 
    L(a)) 
    + \frac{\partial h}{\partial L} 
    (a, L(a)) \frac{d L}{d a}(a), \\
    & \frac{\partial^2 h}{\partial a^2} (a) 
    = \frac{\partial^2 h}{\partial a^2} 
    (a, 
    L(a)) 
    + 2  \frac{\partial^2 h}{
    \partial a \partial L} 
    (a, L(a)) \frac{d L}{d a}(a)
    + \frac{\partial^2 h}{\partial L^2} 
    (a, L(a)) \frac{d^{2} L}{d a^{2}}(a).  
    \end{split}
    \]
By \eqref{eq14-bi}, \eqref{eq41-bi} and 
\eqref{eq60-bi}, we obtain 
    \[
    \frac{\partial h}{\partial a}(0) 
    = 0, \qquad 
    \frac{\partial^2 h}{\partial a^2} (a) 
    = \frac{\partial^2 h}{\partial a^2} 
    (0, L_{*}) 
    = p (p-1) (P_{\perp} (\mathcal{L}_{+, 0}
    - \frac{1}{L_{*}^2} \partial_{yy}) |_{H^2_{\text{ort}}})^{-1} P_\perp
    \phi_{\omega, \gamma}^{p-2} 
    \xi^2. 
    \]
These together with \eqref{eq7-bi}, 
$h(a) \in H^{2}_{\text{ort}}$ and 
$\|\xi\|_{L^{2}} = 1$ (see \eqref{eqFirstEig}
and little below \eqref{eqS2D-2}) 
yield that 
$h(a) =\frac{1}{2} \frac{\partial^{2} h}{\partial a^{2}}(0) a^{2} + o(a^{2})$
and 
\begin{equation}\label{asy2-L2}
\begin{split}
\|\varphi(a)\|_{L^{2}}^{2} 
& = \|\phi_{\omega, \gamma}\|_{L^{2}}^{2} 
+ a^{2} \|\xi\|_{L^{2}}^{2} 
+ 2(\phi_{\omega, \gamma}, h(a))_{L^{2}}
+ \|h(a)\|_{L^{2}}^{2} \\
& =  \|\phi_{\omega, \gamma}\|_{L^{2}}^{2} 
+ a^{2} 
\left(
1 + 
(\phi_{\omega, \gamma}, 
\frac{\partial^2 h}{\partial a^{2}}
(0))_{L^{2}}
\right) + \|h(a)\|_{L^{2}}^{2}.
\end{split}
\end{equation}
Note that $\phi_{\omega, \gamma}$ satisfies 
\eqref{eqS2D-2}. 
Then, differentiating with respect to $\omega$, 
we see that 
$(\mathcal{L}_{+, 0}
    - \frac{1}{L_{*}^2} \partial_{yy})|_{H^2_{\text{ort}}}
)^{-1}\phi_{\omega, \gamma} 
= - \frac{\partial 
\phi_{\omega, \gamma}}{\partial \omega}
$.  
we have 
\begin{equation}\label{eq65-bi}
\begin{split}
\langle \phi_{\omega, \gamma}, \frac{\partial^{2} h}{\partial a^{2}}(0) 
\rangle 
& = p (p-1) 
\langle \phi_{\omega, \gamma}, 
(P_{\perp} (\mathcal{L}_{+, 0}
    - \frac{1}{L_{*}^2} \partial_{yy}) |_{H^2_{\text{ort}}})^{-1} P_\perp
    \phi_{\omega, \gamma}^{p-2} 
    \xi^2
\rangle \\
& = p (p-1) 
\langle (\mathcal{L}_{+, 0}
    - \frac{1}{L_{*}^2} \partial_{yy}) |_{H^2_{\text{ort}}})^{-1}
\phi_{\omega, \gamma}, 
    \phi_{\omega, \gamma}^{p-2} 
    \xi^2
\rangle \\
& = - p (p-1) 
\langle \frac{\partial 
\phi_{\omega, \gamma}}{\partial \omega}, 
    \phi_{\omega, \gamma}^{p-2} 
    \xi^2
\rangle. 
\end{split}
\end{equation}
Differentiating \eqref{eq45-bi} with respect to $\omega$, one has  
    \begin{equation} \label{eq11:critical}
    D_{u} \mathcal{F}(L(a), \varphi(a)) 
    \frac{\partial \phi_{2}}{\partial \omega}(a) 
    + \partial_{\omega} D_{u} \mathcal{F}(L(a), \varphi(a)) \phi_{2}(a, L)
    = \frac{d \lambda_{2}}{d \omega}(a) \phi_{2}(a) 
    + \lambda_{2}(a) \frac{\partial \phi_{2}}
    {\partial \omega}(a). 
    \end{equation}
By taking the scalar product with 
$\phi_{2}(a)$, we obtain 
    \begin{equation} \label{eq63-bi}
    \frac{d \lambda_{2}}{d \omega}(a) 
    = 1- p (p-1) 
    \int_{\mathbb{S}} \varphi^{p-2}(a) 
    |\phi_{2}(a)|^{2} \frac{\partial \varphi}
    {\partial \omega}(a) dxdy. 
    \end{equation}

This together with \eqref{eq62-bi} yields that 
$\frac{\partial \varphi}
    {\partial \omega}(0) = \frac{\partial 
    \phi_{\omega, \gamma}}{\partial \omega}$. 
Putting $a = 0$ in \eqref{eq63-bi}, one has 
    \[
    \frac{d \lambda_{2}}{d \omega}(0) 
    = 1- p (p-1) 
    \int_{\mathbb{S}} \phi_{\omega, 
    \gamma}^{p-2}(a) 
    \xi^{2}
    \frac{\partial \phi_{\omega, \gamma}}
    {\partial \omega}dxdy 
    = 1 + \langle \phi_{\omega, \gamma}, \frac{\partial^{2} h}{\partial a^{2}}(0) 
\rangle.
    \]
It follows from \eqref{asy2-L2} and \eqref{eq65-bi} 
that 
\begin{equation} \label{asy4-L2}
\|\varphi(a)\|_{L^{2}}^{2}
= \|\phi_{\omega, \gamma}\|_{L^{2}}^{2} 
+ a^{2}  \frac{d \lambda_{2}}{d \omega}(0) 
+ \|h(a)\|_{L^{2}}^{2}. 
\end{equation}
\end{proof}

\begin{proof}[Proof of Proposition \ref{prop1:critical}]
We claim that 
$\frac{\partial h}{\partial \omega}(0) = 0$. 
Observe that $\mathcal{F}_{\perp}(L, a, 
h(a, L(a))) = 0$ yields that 
    \begin{equation} \label{eq61-bi}
    \begin{split}
        0 = 
        \frac{\partial 
        \mathcal{F}_\perp}{\partial \omega}
        (L,a,h(a)) & = D_h 
        \mathcal{F}_\perp(L,a,h) \frac{\partial h}
        {\partial \omega}
        (a) + 
        P_{\perp}h(a).
    \end{split}
    \end{equation}
It follows from $h(0) = 0$ that 
$\frac{\partial h}{\partial \omega}(0) = 0$. 
Then, by \eqref{eq1-bifur}, one has 
    \[
    \partial_{\omega} \|\varphi(a)\|_{L^{2}}^{2} 
    = \partial_{\omega} \|\phi_{\omega, \gamma}\|_{L^{2}}^{2} 
    + a^{2} \frac{d^{2} \lambda_{2}}{d \omega^{2}}(0) 
    + 2 \int_{\mathbb{S}} h(a) \frac{\partial h}{\partial \omega} 
    (a) dx.   
    \]
This together with Lemma
\ref{lemStabil} (1) and $h(0) 
= \frac{\partial h}{\partial \omega}(0) 
= 0$
implies that 
$\partial_{\omega} \|\varphi(a)\|_{L^{2}}^{2} > 0$ 
for sufficiently small $|a| > 0$. 
\end{proof}

Then, from Proposition \ref{prop1:critical}
and the result of Grillakis, Shatah, and 
Strauss~\cite{MR1081647}, we can obtain the 
following: 
\begin{theorem}\label{thm2:critical}
Let $n(a)$
be the number of negative eigenvalues of 
$S_{L(a), \gamma}''(\varphi(a))$. 
Then the bifurcation soliton $e^{i \omega t} 
\varphi(a)$ is stable if 
$n(L) = 1$ and unstable if $n(L) = 2$.
\end{theorem}
From Theorem 
\ref{thm2:critical}, it suffices to 
count the number of the linearized operator 
$S_{L(a), \gamma}''(\varphi(a))$.

\begin{proposition}\label{prop:critical}
Let $\lambda_{2}(a)$ be the second eigenvalue of 
$\mathcal{L}_{+}(a, L(a))$. 
Then, we obtain 
   \begin{equation} \label{main-lambda}
    \lambda_{2} (a) 
    = -\frac{d^{2} L}{d a^{2}}(0)\lambda_*' 
    a^{2} + o(a^{2}).  
   \end{equation}
where $\lambda_*'$ is defined by 
\eqref{eq34-bi} and $\frac{d^{2} L}{d a^{2}}(0)$ in \eqref{eq19-bi}.
\end{proposition}

\begin{proof}
Differentiating \eqref{eq45-bi} with respect to $a$, we obtain 
    \begin{equation} \label{eq1:critical}
    \mathcal{L} (a, L(a)) 
    \frac{\partial \phi_{2}}{\partial a}(a) 
    + \partial_{a} \mathcal{L} (a, L(a)) \phi_{2}(a)
    = \frac{d \lambda_{2}}{d a}(a) \phi_{2}(a) 
    + \lambda_{2}(a) \frac{\partial \phi_{2}}
    {\partial a}(a). 
    \end{equation}
Observe that 
    \begin{equation} \label{eq3:critlcal}
    \begin{split}
     \partial_{a} \mathcal{L} (a, L(a))
    & = 2 L^{-3}(a) \frac{d L}{d a}(a) 
    \partial _{yy} 
    - p (p-1) \varphi^{p-2}(a) 
    \frac{\partial \varphi}{\partial a}(a) \\
    & = 2 L^{-3}(a) \frac{d L}{d a}(a) 
    \partial _{yy} 
    - p (p-1) \varphi^{p-2}(a) 
    \left(\xi + \frac{\partial h}{\partial a}
    (a) \right). 
    \end{split}
    \end{equation}
Then, we have by \eqref{eq59-bi} 
and \eqref{eq38-bi} that 
    \[
    \partial_{a} \mathcal{L} (a, L(a))|_{a = 0} 
    = - p(p-1) 
    (\phi_{\omega, \gamma}^{(1)})^{p-2} \xi
    \]
Taking a scalar product with $\phi_{2}(a)$ 
on \eqref{eq3:critlcal}, 
we obtain 
    \begin{equation}\label{eq2:critical}
    \begin{split}
    & \quad \langle 
    \mathcal{L} (a, L(a)) 
    \frac{\partial \phi_{2}}{\partial a}(a), 
    \phi_{2}(a) \rangle 
    + \langle \partial_{a} \mathcal{L} (a, L(a)) \phi_{2}(a), 
    \phi_{2}(a) \rangle \\
    & = 
    \frac{d \lambda_{2}}{d a}(a)
    \|\phi_{2}(a)\|_{L^{2}}^{2} 
    + \lambda_{2}(a) 
    \langle \frac{\partial \phi_{2}}
    {\partial a}(a), \phi_{2}(a) \rangle \\
    & = \frac{d \lambda_{2}}{d a}(a)
    + \lambda_{2}(a) 
    \langle \frac{\partial \phi_{2}}
    {\partial a}(a), \phi_{2}(a) \rangle.  
    \end{split}
    \end{equation}
 Substituting the above with $a = 0$, 
 we have 
    \[
    \begin{split}
    & \quad \langle 
    D_{u} \mathcal{F}(L_*, \varphi(0)) 
    \frac{\partial \phi_{2}}{\partial a}(0), 
    \phi_{2}(0) \rangle 
    + \langle \partial_{a} D_{u} 
    \mathcal{F}(L_*, \varphi(0)) \phi_{2}(0), 
    \phi_{2}(0) \rangle \\
    & = \frac{d \lambda_{2}}{d a}(0)
    + \lambda_{2}(0) 
    \langle \frac{\partial \phi_{2}}
    {\partial a}(0), \phi_{2}(0) \rangle. 
    \end{split}
    \]
Using $\lambda_{2}(0) = 0$, $\phi_{2}(0) 
= \xi$, $D_{u} \mathcal{F}(L_{*}, \varphi(0))
\phi_{2}(0) = 0$ and 
    \[
    \begin{split}
    \langle \partial_{a} D_{u} \mathcal{F}(L(0), \varphi(0)) \phi_{2}(0), 
    \phi_{2}(0) \rangle
    & = - p(p-1) 
    \langle (\phi_{\omega, \gamma}^{(1)})^{p-2} 
    \xi^2, \xi \rangle \\
    & = - p(p-1) 
    \int_{\R} 
    (\phi_{\omega, \gamma}^{(1)}(x))^{p-2} 
    \chi^{3}(x) dx 
    \int_{0}^{1} (\cos (\pi y))^3 dy \\
    & = 0,  
    \end{split}
    \]
we have $\frac{d \lambda_{2}}{d a}(0) = 0$. 

Putting $a = 0$ into \eqref{eq1:critical}, 
we have by $\lambda_{2}(0) = \frac{d \lambda_{2}}{d a}(0) = 0$ that  
    \[
    D_{u}\mathcal{F}(L_*, 
    \phi_{\omega, \gamma}^{(1)}) 
    \frac{\partial \phi_{2}}{\partial a}(0) 
    + \partial_{a} D_{u}\mathcal{F}(L_*, 
    \phi_{\omega, \gamma}^{(1)}) \phi_{2}(0) 
    = 0.
    \]
This together with $\phi_{2}(0) = \xi$ 
yields that 
    \begin{equation} \label{der2-h-a-0}
    \begin{split}
    \frac{\partial \phi_{2}}{\partial a}(0)  
    & = - (P_{\perp} D_{u}\mathcal{F}(L_*, 
    \phi_{\omega, \gamma}^{(1)}))^{-1} 
    \left(P_{\perp} 
    \partial_{a} D_{u}\mathcal{F}(L_*, 
    \phi_{\omega, \gamma}^{(1)}) \xi 
    \right) \\
    & = p (p-1) (P_{\perp} D_{u}\mathcal{F}(L_*, 
    \phi_{\omega, \gamma}^{(1)}))^{-1} 
     (\phi_{\omega, \gamma}^{(1)})^{p-2} 
     \xi^2. 
    \end{split}
    \end{equation}
Differentiating \eqref{eq1:critical} 
with respect to the parameter $a$ and 
taking the scalar product with 
$\phi_{2}(a)$, 
we obtain
\begin{equation} \label{2der-a-lambda}
\begin{split}
\frac{d^{2} \lambda_{2}}{d a^{2}}(a)
& = 2 \langle 
\left(
\partial_{a} 
\mathcal{L} (a, L(a)) \right)
\frac{\partial \phi_{2}}{\partial a}(a), \phi_{2}(a) 
\rangle + \langle (\partial_{a}^{2} \mathcal{L} (a, L(a)))
\phi_{2}(a), \phi_{2}(a) \rangle \\
& \quad + \langle 
\mathcal{L} (a, L(a))
\frac{\partial^2 \phi_{2}}{\partial a^2}(a), \phi_{2}(a) 
\rangle 
- \lambda_{2}(a) 
\langle \frac{\partial^2 \phi_{2}}{\partial a^2}(a), 
\phi_{2}(a) \rangle 
- 2 \frac{\partial \lambda_{2}}{\partial a}(a) 
\langle \frac{\partial \phi_{2}}{\partial a}
(a), \phi_{2}(a) \rangle
\\
& = 2 \langle \left(\partial_{a} 
\mathcal{L} (a, L(a))\right)
\frac{\partial \phi_{2}}{\partial a}(a), \phi_{2}(a) 
\rangle + \langle (\partial_{a}^{2} \mathcal{L} (a, L(a)))
\phi_{2}(a), \phi_{2}(a) \rangle. 
\end{split}
\end{equation}
Here, we have used the fact that
$\langle \frac{\partial 
\phi_{2}}{\partial a}(a), \phi_{2}(a) 
\rangle = 0$ because 
$\|\phi_{2}(a)\|_{L^{2}}^{2} = 1$ 
for each $a > 0$.

Differentiating \eqref{eq3:critlcal} with respect to the parameter $a$ 
yields that 
\begin{equation} \label{2der-L-a}
\begin{split}
\partial_{a}^{2} 
D_{u}\mathcal{F}(L(a), \varphi(a)) 
& = 
-6 L^{-4}(a) \frac{d L}{d a}(a) 
\partial_{yy} 
+ 2 L^{-3}(a) \frac{d^{2} L}{d a^2}(a) 
\partial_{yy} \\
& \quad 
- p(p-1)(p-2) 
\varphi^{p-3}(a)
\left(\xi + \frac{\partial h}{\partial a}
(a) \right)^{2} 
- p(p-1) \varphi^{p-2}(a) 
\frac{d^{2} h}{d a^{2}}(a)
\end{split}
\end{equation}
Moreover, 
differentiating \eqref{eq1-varphi} with respect to the parameter $a$, 
we have 
\begin{equation} \label{eq2-varphi}
\frac{d \varphi}{d a}(a) = 
\xi 
+ \frac{\partial h}{\partial a}(L(a), a) 
+ \frac{d L}{d a}(a) 
\frac{\partial h}{\partial L}
(L(a), a). 
\end{equation}
Substituting the above with $a = 0$, 
we have by \eqref{eq14-bi} that 
\begin{equation*}
\frac{d \varphi}{d a}(0) 
= \xi. 
\end{equation*}
Differentiating \eqref{eq2-varphi} with respect to the parameter $a$ gives 
\begin{equation*} 
\begin{split}
\frac{d^{2} \varphi}{d a^{2}}(a) 
& 
= \frac{\partial^{2} h}{\partial a^{2}}
 (L(a), a) 
+ 2 \frac{d L }{d a}(a) 
\frac{\partial^{2} h}{\partial a \partial L}
(L(a), a) \\[6pt]
& 
+ \frac{d^{2} L}{d a^{2}}
(a)
\frac{\partial h}{\partial L}(L(a), a) 
+ \left(\frac{d L }{d a}(a) \right)^{2} \frac{\partial^{2} h}
{d L ^{2}}(L(a), a). 
\end{split}
\end{equation*}
Substituting the above with $a = 0$, we have by \eqref{eq38-bi} that 
\begin{equation} \label{2der-a-varphi}
\frac{d^{2} \varphi}{d a^{2}} (0)
= \frac{d^{2} h}{d a^{2}} (L_*, 0) 
+ \frac{d^{2} L}{d a^{2}}(0)
\frac{\partial h}{\partial L}(L_*, 0).  
\end{equation}
By \eqref{2der-L-a} and \eqref{2der-a-varphi}, we obtain  
\begin{equation*} 
\begin{split}
& \quad \langle 
(\partial_{a}^{2} D_{u} \mathcal{F}(L_*, \phi_{\omega, \gamma}^{(1)}))\xi, \xi \rangle 
\\[6pt]
& = - 2 L_{*}^{-3} 
\frac{d^{2} L}{d a^{2}}(0)  
-p(p-1)(p-2) \int_{\mathbb{R} \times \mathbb{T}}
(\phi_{\omega, \gamma}^{(1)})^{p-3} (\xi)^{4} \, dxdy  \\[6pt]
& \quad - p(p-1) \int_{\mathbb{R} \times \mathbb{T}} 
(\phi_{\omega, \gamma}^{(1)})^{p-2} \xi^2 \frac{d^{2} \varphi}{d a^{2}}(0) 
\, dxdy  \\[6pt]
& = - 2 L_{*}^{-3} 
\frac{d^{2} L}{d a^{2}}(0) 
-p(p-1)(p-2) \int_{\mathbb{R} \times \mathbb{T}}
(\phi_{\omega, \gamma}^{(1)})^{p-3} (\xi)^{4} \, dxdy  \\[6pt]
& \quad - p(p-1) \int_{\mathbb{R} \times \mathbb{T}} 
(\phi_{\omega, \gamma}^{(1)})^{p-2} \xi^2 
\left( \frac{d^{2} h}{d a^{2}}(L_*, 0) + 
\frac{d^{2} L}{d a^{2}}(0)
 \frac{\partial h}{\partial L}(L_*, 0) \right)
\, dxdy  \\[6pt]
& = 
- 2 \L_{*}^{-3} 
\frac{d^{2} L}{d a^{2}}(0) 
-p(p-1)(p-2) \int_{\mathbb{R} \times \mathbb{T}}
(\phi_{\omega, \gamma}^{(1)})^{p-3} (\xi)^{4} \, dxdy  \\[6pt]
& \quad - p(p-1) \int_{\mathbb{R} \times \mathbb{T}} 
(\phi_{\omega, \gamma}^{(1)})^{p-2} \xi^2 
\frac{d^{2} h}{d a^{2}}(L_*, 0) \, dxdy  \\[6pt]
& \quad - p(p-1) \frac{d^{2} L}{d a^{2}}(0)
\int_{\mathbb{R} \times \mathbb{T}} 
(\phi_{\omega, \gamma}^{(1)})^{p-2} \xi^2 
 \frac{\partial h}{\partial L}(L_*, 0) 
\, dxdy . 
\end{split}
\end{equation*}
From \eqref{eq33-bi} and 
$\|\partial_{y} \phi_{2}(0, L_{*})\|_{L^{2}} 
= \|\partial_{y} \xi\|_{L^{2}} = 1$, we obtain
\begin{equation*}
\begin{split}
& - 2 L_{*}^{-3} \frac{d^{2} L}{d a^{2}}(0) 
- p(p-1) \frac{d^{2} L}{d a^{2}}(0)
\int_{\mathbb{R} \times \mathbb{T}} 
(\phi_{\omega, \gamma}^{(1)})^{p-2} \xi^2 
 \frac{\partial h}{\partial L}(L_*, 0) 
\, dxdy  \\[6pt]
& = \frac{d^{2} L}{d a^{2}}(0) 
\left(- 2 L_{*}^{-3} - p(p-1) \int_{\mathbb{R} \times \mathbb{T}} 
(\phi_{\omega, \gamma}^{(1)})^{p-2} \xi^2  
\frac{\partial h}{\partial L}(L_{*}, 0) 
\, dxdy \right) \\[6pt]
& = \frac{d^{2} L}{d a^{2}}(0) 
\lambda_{2}'.   
\end{split}
\end{equation*}
Thus, one has
\begin{equation} \label{2der-L-a-int}
\begin{split}
\langle 
(\partial_{a}^{2} D_{u} \mathcal{F}(L_*, \phi_{\omega, \gamma}^{(1)}))\xi, \xi \rangle 
& = \frac{d^{2} L}{d a^{2}}(0) 
\lambda_{2}'
-p(p-1)(p-2) \int_{\mathbb{R} \times \mathbb{T}}
(\phi_{\omega, \gamma}^{(1)})^{p-3} (\xi)^{4} \, dxdy \\
& \quad - p(p-1) \int_{\mathbb{R} \times \mathbb{T}} 
(\phi_{\omega, \gamma}^{(1)})^{p-2} \xi^2 
\frac{d^{2} h}{d a^{2}}(L_*, 0) \, dxdy 
\end{split}
\end{equation}
It follows from \eqref{2der-a-lambda} that 
\begin{equation}\label{2der-a-lambda-0}
\begin{split}
\frac{d^{2} \lambda_{2}}{d a^{2}}(0) 
= 
-2 p (p-1) 
\int_{\mathbb{R} \times \mathbb{T}}
(\phi_{\omega, \gamma}^{(1)})^{p-2} 
\xi^{2} 
\frac{\partial \phi_{2}}{\partial a}(0) \, dxdy  
+ \langle  
\partial_{a}^{2} D_{u} \mathcal{F}(L_{*}, 
\phi_{\omega, \gamma}^{(1)})
\xi, \xi \rangle  
\end{split}
\end{equation}
 It follows from 
\eqref{2der-a-lambda-0}, 
\eqref{2der-L-a-int}, \eqref{der2-h-a-0}, \eqref{eq26-bi} 
and \eqref{eq19-bi}, in this order,  
that 
\begin{equation*}
\begin{split}
\frac{d^{2} \lambda_{2}}{d a^{2}}(0) 
& = 
-2 p (p-1) 
\int_{\mathbb{R} \times \mathbb{T}}
(\phi_{\omega, \gamma}^{(1)})^{p-2} \xi^2 
\frac{\partial \phi_{2}}{\partial a}(0) \, dxdy  \\[6pt]
& \quad + \frac{d^{2} L}{d a^{2}}(0) 
\lambda_{2}'
-p(p-1)(p-2) \int_{\mathbb{R} \times \mathbb{T}}
(\phi_{\omega, \gamma}^{(1)})^{p-3} (\xi)^{4} \, dxdy \\
& \quad - p(p-1) \int_{\mathbb{R} \times \mathbb{T}} 
(\phi_{\omega, \gamma}^{(1)})^{p-2} \xi^2 
\frac{d^{2} h}{d a^{2}}(L_*, 0) \, dxdy 
\\[6pt]
& = - 3 p^{2} (p-1)^{2} 
\int_{\mathbb{R} \times \mathbb{T}}
(\phi_{\omega, \gamma}^{(1)})^{p-2} \xi^2 
(P_{\perp} D_{u} \mathcal{F}(L_{*}, \phi_{\omega, \gamma}^{(1)}))^{-1} 
\left((\phi_{\omega, \gamma}^{(1)})^{p-2} \xi^2 \right)\, dxdy  \\[6pt]
& \quad 
+ \frac{d^{2} L}{d a^{2}}(0) 
\lambda_{2}'
- p(p-1)(p-2) \int_{\mathbb{R} \times \mathbb{T}}
(\phi_{\omega, \gamma}^{(1)})^{p-3} (\xi)^{4} \, dxdy  
\\[6pt]
& = - 3 \frac{d^{2} L}{d a^{2}}(0) 
\lambda_*'
+ \frac{d^{2} L}{d a^{2}}(0) 
\lambda_*'
= -2 \frac{d^{2} L}{d a^{2}}(0) 
\lambda_*'. 
\end{split}
\end{equation*}
From this, we find 
that \eqref{main-lambda} holds. 
\end{proof}

\begin{proof}[Proof of Theorem \ref{thmStab2}.]
By Proposition \ref{prop1:critical} and the
Grillakis--Shatah--Strauss criterion~\cite{MR1081647}, the
stability of the bifurcating standing wave is determined by the
number of negative eigenvalues of
$S_{L(a),\gamma}''(\varphi(a))$. Since the bifurcation occurs from
the line soliton at which the second eigenvalue of
$\mathcal{L}_{+}$ vanishes, the only eigenvalue whose sign may
change for sufficiently small $|a|$ is $\lambda_{2}(a)$. By
Proposition \ref{prop:critical},
\begin{equation*}
    \lambda_{2}(a)
    =
    -\frac{d^{2}L}{da^{2}}(0)\lambda_{*}'a^{2}
    +o(a^{2}).
\end{equation*}
By \eqref{eq34-bi}, we have $\lambda_{*}'<0$. Hence, the sign of
$\lambda_{2}(a)$ is the same as the sign of
$\frac{d^{2}L}{da^{2}}(0)$. Therefore, for sufficiently small
$|a|>0$,
\[
\frac{d^{2}L}{da^{2}}(0)>0
\quad\Longrightarrow\quad n(a)=1,
\]
whereas
\[
\frac{d^{2}L}{da^{2}}(0)<0
\quad\Longrightarrow\quad n(a)=2.
\]
The stability and instability statements therefore follow from
Theorem \ref{thm2:critical}.

It remains to determine the sign of
$\frac{d^{2}L}{da^{2}}(0)$. By \eqref{eq19-bi},
\begin{equation*}
\begin{aligned}
\frac{d^{2}L}{da^{2}}(0)
={}&
\frac{p(p-1)(p-2)}{3\lambda_{*}'}
\int_{\mathbb{S}}
\phi_{\omega,\gamma}^{p-3}\xi_{\gamma}^{4}\,dx\,dy
\\
&+
\frac{p^{2}(p-1)^{2}}{\lambda_{*}'}
\left\langle
\phi_{\omega,\gamma}^{p-2}\xi_{\gamma}^{2},
\left(
P_{\perp}
\left(
\mathcal{L}_{+,0}
-L_{*}^{-2}\partial_{yy}
\right)
\big|_{H^{2}_{\mathrm{ort}}}
\right)^{-1}
\left(
\phi_{\omega,\gamma}^{p-2}\xi_{\gamma}^{2}
\right)
\right\rangle .
\end{aligned}
\end{equation*}
Here
\[
\xi_\gamma(x,y)= \sqrt{2}\chi_\gamma(x)\cos(\pi y),
\]
where $\chi_\gamma$ is the normalized positive eigenfunction associated
with the lowest eigenvalue of the one-dimensional linearized
operator $\mathcal{L}_{+,0}(\gamma)$ around
$\phi_{\omega,\gamma}$.

Let $f_\gamma(x) := \phi_{\omega,\gamma}^{p-2}(x)\chi_{\gamma}^{2}(x)$. Using the trigonometric identity, we expand the source term as
\begin{equation*}
    \phi_{\omega,\gamma}^{p-2}\xi_{\gamma}^{2}(x,y)
    = 2\phi_{\omega,\gamma}^{p-2}(x)\chi_{\gamma}^{2}(x)\cos^{2}(\pi y)
    = f_\gamma(x) + f_\gamma(x)\cos(2\pi y).
\end{equation*}

Observe that the operator $P_{\perp} \left( \mathcal{L}_{+,0} - L_{*}^{-2}\partial_{yy} \right)$ acts diagonally on the longitudinal Fourier modes $\cos(n\pi y)$. Introducing the decoupled one-dimensional operators
\begin{equation*}
    \mathbf{A}_{n}(\gamma) := \mathcal{L}_{+,0} + n^{2}\pi^{2}L_{*}^{-2} \quad \text{for } n \in \mathbb{N}_{0},
\end{equation*}
the action of the inverse operator on the active modes ($n=0$ and $n=2$) is given by
\begin{equation*}
    \left( P_{\perp} \left( \mathcal{L}_{+,0} - L_{*}^{-2}\partial_{yy} \right) \big|_{H^{2}_{\mathrm{ort}}} \right)^{-1} \left( \phi_{\omega,\gamma}^{p-2}\xi_{\gamma}^{2} \right)
    = \mathbf{A}_{0}^{-1}(\gamma)f_\gamma(x) + \mathbf{A}_{2}(\gamma)^{-1}f_\gamma(x)\cos(2\pi y).
\end{equation*}

Taking the inner product in $L^{2}(\mathbb{S})$, we integrate over $x \in \mathbb{R}$ and $y \in [0,1]$. By the $L^{2}$-orthogonality of the Fourier modes on $[0,1]$, the cross-terms vanish. In this way, we obtain the modal decomposition:
\begin{equation*}
\begin{aligned}
    &\left\langle
    \phi_{\omega,\gamma}^{p-2}\xi_{\gamma}^{2},
    \left(
    P_{\perp}
    \left(
    \mathcal{L}_{+,0}
    -L_{*}^{-2}\partial_{yy}
    \right)
    \big|_{H^{2}_{\mathrm{ort}}}
    \right)^{-1}
    \left(
    \phi_{\omega,\gamma}^{p-2}\xi_{\gamma}^{2}
    \right)
    \right\rangle_{L^{2}(\mathbb{S})} \\
    &\quad = \int_{\mathbb{R}} \int_{0}^{1} \Big( f_\gamma(x) + f_\gamma(x)\cos(2\pi y) \Big) \Big( \mathbf{A}_{0}^{-1}(\gamma)f_\gamma(x) + \mathbf{A}_{2}^{-1}(\gamma)f_\gamma(x)\cos(2\pi y) \Big) \, dy \, dx \\
    &\quad = \left\langle f_\gamma, \mathbf{A}_{0}^{-1}(\gamma)f_\gamma \right\rangle_{L^{2}(\mathbb{R})} + \frac{1}{2} \left\langle f_\gamma, \mathbf{A}_{2}^{-1}(\gamma)f_\gamma \right\rangle_{L^{2}(\mathbb{R})}.
\end{aligned}
\end{equation*}
For $\gamma=0$, both $\phi_{\omega,0}$ and $\chi_0$ are explicit, and
the sign of $\frac{d^{2}L}{da^{2}}(0)$ can be determined explicitly. For $\gamma<0$, although the corresponding
profiles can also be represented explicitly, the resulting
expression is considerably more involved. We therefore determine
the sign by perturbing from the case $\gamma=0$.

\medskip

\textbf{Case $\gamma = 0$:} For $\gamma = 0$, the linearized operator simplifies, and the principal eigenfunction is explicitly given by $\chi_0(x) = \phi_{\omega,0}^{(p+1)/2}(x)$. Consequently, the source term $f_0(x) = \phi_{\omega,0}^{p-2}(x)\chi_0^2(x)$ reduces to
\begin{equation*}
    f_0(x) = \phi_{\omega,0}^{p-2}(x) \phi_{\omega,0}^{p+1}(x) = \phi_{\omega,0}^{2p-1}(x).
\end{equation*}

Substituting these identities into the expression for $\frac{d^2L}{da^2}(0)$ at $\gamma = 0$, we obtain
\begin{equation*}
\begin{aligned}
\frac{d^{2}L}{da^{2}}(0) ={}& \frac{p^{2}(p-1)^{2}}{\lambda_{*}'} \Bigg[ \left\langle \phi_{\omega,0}^{2p-1}, \mathbf{A}_{0}^{-1}(0) \phi_{\omega,0}^{2p-1} \right\rangle_{L^{2}(\mathbb{R})} + \frac{1}{2} \left\langle \phi_{\omega,0}^{2p-1}, \mathbf{A}_{2}^{-1}(0) \phi_{\omega,0}^{2p-1} \right\rangle_{L^{2}(\mathbb{R})} \\
&\qquad\qquad\quad + \frac{p-2}{2p(p-1)} \int_{\mathbb{R}} \phi_{\omega,0}^{3p-1}(x) \, dx \Bigg].
\end{aligned}
\end{equation*}
We evaluate the terms inside the brackets. For the first term, 
we remark that applying $\mathbf{A}_0^{-1}(0)$ to $\phi_{\omega,0}^{2p-1}$ yields from \eqref{eqStationary2D}
    $$\mathbf{A}_0^{-1}(0) \phi_{\omega,0}^{2p-1} = \frac{p+1}{2p(p-1)} \left( \phi_{\omega,0}^p - (p+1)\omega \phi_{\omega,0} \right)$$
    which gives
    \begin{equation*}
        \left\langle \phi_{\omega,0}^{2p-1}, \mathbf{A}_0^{-1}(0) \phi_{\omega,0}^{2p-1} \right\rangle_{L^2(\mathbb{R})} = -\frac{(3p-1)(p+1)}{4p^2(p-1)} \int_{\mathbb{R}} \phi_{\omega,0}^{3p-1}(x) \, dx.
    \end{equation*}

For the second term, using the spectral operator bound $$\|\mathbf{A}_2^{-1}(0)\|_{L^2 \to L^2} \le \frac{1}{4\pi^2 L_*^{-2}},$$ 
with $\omega = \frac{4\pi^2 L_*^{-2}}{(p+3)(p-1)}$, 
we estimate
    \begin{equation*}
        \left\langle \phi_{\omega,0}^{2p-1}, \mathbf{A}_2^{-1}(0) \phi_{\omega,0}^{2p-1} \right\rangle_{L^2(\mathbb{R})} \le \frac{4(p+1)(3p-1)}{3(7p-3)(p+3)(p-1)} \int_{\mathbb{R}} \phi_{\omega,0}^{3p-1}(x) \, dx.
    \end{equation*}

Combining these results yields the upper bound
\begin{equation}\label{eqUpBoundP}
    \lambda_*' \frac{d^{2}L}{da^{2}}(0) \le \frac{(p-1)\left(3p^4 - 164p^3 - 284p^2 + 216p - 27\right)}{48  (7p-3)(p+3)} \int_{\mathbb{R}} \phi_{\omega,0}^{3p-1}(x) \, dx.
\end{equation}
One can easily see that the upper bound in \eqref{eqUpBoundP} is negative for $p\in [2,5)$, which implies from \eqref{main-lambda} that  $\lambda_2(a)$ is strictly positive for $\gamma = 0$ and $a$ sufficiently small. 

\medskip

\textbf{Case $\gamma <0$:} The coefficient
$\frac{d^{2}L}{da^{2}}(0)$ depends continuously on $\gamma$. Hence,
provided that its value at $\gamma=0$ is nonzero, its sign remains
unchanged for all sufficiently small $|\gamma|$. 

To rigorously justify this claim, it suffices to establish the continuity of the map $\gamma \mapsto \frac{d^{2}L}{da^{2}}(0)$ at $\gamma = 0$. We proceed via operator perturbation theory.

First, consider the one-dimensional Hamiltonian $H_{\gamma} = -\partial_{xx} + \gamma \delta_{0}$. As $\gamma \to 0$, $H_{\gamma} \to H_{0} = -\partial_{xx}$ in the norm resolvent sense on $L^{2}(\mathbb{R})$. The ground state profile $\phi_{\omega,\gamma}$ satisfies the stationary equation
\begin{equation*}
    H_{\gamma}\phi_{\omega,\gamma} + \omega \phi_{\omega,\gamma} - \phi_{\omega,\gamma}^{p} = 0.
\end{equation*}
By the implicit function theorem applied to the associated energy functional, the map $\gamma \mapsto \phi_{\omega,\gamma}$ is continuous into $H^{1}(\mathbb{R})$, with $\phi_{\omega,\gamma} \to \phi_{\omega,0}$ strongly in $H^{1}(\mathbb{R})$ as $\gamma \to 0$. Consequently, the linearized operator $\mathcal{L}_{+,0}(\gamma) = H_{\gamma} + \omega - p\phi_{\omega,\gamma}^{p-1}$ converges to $\mathcal{L}_{+,0}(0)$ in the norm resolvent sense.

Since the principal eigenvalue $\lambda_{\gamma}$ of 
$\mathcal{L}_{+,0}(\gamma)$ is simple and isolated, 
Kato's perturbation theory \cite[Page 67--68]{Kato1980}
ensures that the eigenvalue and its associated normalized positive eigenfunction $\chi_{\gamma}$ depend continuously on $\gamma$. Specifically, $\chi_{\gamma} \to \chi_{0}$ strongly in $H^{1}(\mathbb{R})$ as $\gamma \to 0$.

This strong convergence implies that the source term $f_{\gamma} = \phi_{\omega,\gamma}^{p-2}\chi_{\gamma}^{2}$ satisfies
\begin{equation*}
    \lim_{\gamma \to 0} \|f_{\gamma} - f_{0}\|_{L^{2}(\mathbb{R})} = 0.
\end{equation*}
Furthermore, for $n \in \{0, 2\}$, the operators $\mathbf{A}_{n}(\gamma) = \mathcal{L}_{+,0}(\gamma) + n^{2}\pi^{2} L_{*}^{-2}$ (with $\mathbf{A}_{0}$ restricted to $H^{2}_{\mathrm{ort}}$) are uniformly coercive for $|\gamma|$ small enough. The norm resolvent convergence of $\mathcal{L}_{+,0}(\gamma)$ thus yields the operator norm convergence of their inverses:
\begin{equation*}
    \lim_{\gamma \to 0} \left\| \mathbf{A}_{n}^{-1}(\gamma) - \mathbf{A}_{n}^{-1}(0) \right\|_{\mathcal{B}(L^{2}(\mathbb{R}))} = 0.
\end{equation*}

We can now pass to the limit in each component of $\frac{d^{2}L}{da^{2}}(0)$. For the inner product terms, we have
\begin{equation*}
\begin{aligned}
    &\left| \langle f_{\gamma}, \mathbf{A}_{n}^{-1}(\gamma)f_{\gamma} \rangle_{L^{2}} - \langle f_{0}, \mathbf{A}_{n}^{-1}(0)f_{0} \rangle_{L^{2}} \right| \\
    &\quad \le \|\mathbf{A}_{n}^{-1}(\gamma)\| \|f_{\gamma} - f_{0}\|_{L^{2}} (\|f_{\gamma}\|_{L^{2}} + \|f_{0}\|_{L^{2}}) + \|\mathbf{A}_{n}^{-1}(\gamma) - \mathbf{A}_{n}^{-1}(0)\| \|f_{0}\|_{L^{2}}^{2},
\end{aligned}
\end{equation*}
which vanishes as $\gamma \to 0$. Similarly, the local integral term converges continuously:
\begin{equation*}
    \lim_{\gamma \to 0} \int_{\mathbb{R}} \phi_{\omega,\gamma}^{p-3}(x) \chi_{\gamma}^{4}(x) \, dx = \int_{\mathbb{R}} \phi_{\omega,0}^{p-3}(x) \chi_{0}^{4}(x) \, dx.
\end{equation*}

Combining these limits, we conclude that
\begin{equation*}
    \lim_{\gamma \to 0^{-}} \frac{d^{2}L}{da^{2}}(0) = \left.\frac{d^{2}L}{da^{2}}(0)\right|_{\gamma=0} < 0.
\end{equation*}
By strict inequality and continuity, there exists a threshold $\gamma_{0} > 0$ such that $\frac{d^{2}L}{da^{2}}(0) < 0$ for all $\gamma \in (-\gamma_{0}, 0]$.
\end{proof}

\appendix

\section{Proof of \eqref{eq3-bi}}\label{AppFredholm}

In this appendix, we show \eqref{eq3-bi}. We recall that 
\begin{equation*}
    D_h F(0, L, 0) = H_L + \omega - p (\phi_{\omega, \gamma}^{(1)})^{p-1}
\end{equation*} 
where formally $H_L = -\partial_{xx} - \frac{1}{L^2} \partial_{yy} + \gamma \tau^* \tau$ with the domain $D(H_L) \subset H^2(\mathbb{S})$, while $\phi_{\omega, \gamma}^{(1)}$ is the trivial extension of the one-dimensional soliton $\phi_{\omega, \gamma}$ to $\mathbb{S}$. Here $\gamma \geq 0$ and $\omega > \gamma^2/4$ are fixed. In this setting, it has been shown in \cite[Lemma 3.3]{CoSh24} that for any $L>0$, the form associated with the self-adjoint operator $H_L + \omega$ is coercive, that is there exists $C > 0$ such that for any $u \in D(H_L)$,
\begin{equation*}
    ((H_L + \omega) u, u ) \geq C \|u\|_{L^2(\mathbb S)}^2.
\end{equation*}
Thus, by the Lax-Milgram theorem, we have that $H_L + \omega$ is bijective from $D(H_L)$ to $L^2(\mathbb S)$. 

Next, it is easy to see that the essential spectrum of $H_L + \omega$ is contained in $[\omega, \infty)$ by the classical Weyl's theorem. On the other hand, by the same theorem, as $\phi_{\omega, \gamma}$ decays exponentially for $x \to \pm \infty$, the essential spectrum of $D_h F(0, L, 0)$ is also contained in $[\omega, \infty)$, while the discrete spectrum is made of isolated eigenvalues with finite multiplicity. All this is well known in the one-dimensional case, and the extension to the strip is straightforward.

In particular, the kernel is finite dimensional and isolated. Thus if $T =D_h F(0, L^*, 0)_{|H^2_{\text{ort}}}: H^2_{\text{ort}} \to L^2_{\text{ort}}$ is the restriction of $D_h F(0, L^*, 0)$ to the orthogonal complement of the kernel, then $T$ has closed range and is injective. It follows directly that $ran(T) = \overline{ran(T)} = (\ker(T))^\perp= L^2_{\text{ort}}$ which concludes the proof of \eqref{eq3-bi}.

\section{Proof the spectral mapping property  \eqref{spectral-mapping}}
\label{sec:spectral-m}

In this section, we show \eqref{spectral-mapping}.
Let $\gamma < 0$ and $\omega > \gamma^2/4$ be fixed.
We define the operator $A: \mathcal{D}(\H) \times \mathcal{D}(\H) \to L^2(\mathbb S_L) \times L^2(\mathbb S_L)$ by
\begin{equation*}
    A = J S(\phi_{\omega, \gamma}^{(L)}) = \begin{pmatrix}0 & - \mathcal{L}_{+,L} \\ \mathcal{L}_{-,L} & 0
\end{pmatrix} 
\end{equation*}
Being a bounded perturbation of the skew-adjoint operator
\begin{equation*}
    \begin{pmatrix}0 & \Delta \\ -\Delta & 0
\end{pmatrix} 
\end{equation*}
the operator $A$ generates a $C_0$ semigroup on $L^2(\R) \times L^2(\R)$, which we denote by $\{e^{tA}\}_{t \geq 0}$. Moreover, it holds: 
\begin{lem}\label{lemSpABor}
   The essential spectrum of $A$ is such that $\sigma_{\text{ess}}(A) \subset \{iz : |z| \geq \omega, z \in \mathbb{R}\}$. The discrete spectrum of $A$ consists of a finite number of isolated eigenvalues with finite multiplicity.
\end{lem}
The proof of the above lemma follows from the Weyl's theorem and the fact that the essential spectrum of $A$ is the same as the one of the operator
\begin{equation*}
    \begin{pmatrix}0 & -\Delta + \omega \\ \Delta - \omega & 0
\end{pmatrix}
\end{equation*}
for which the same properties have been shown in \cite[Theorem $3.1$]{Gr88Inst}. 

Next, we as in Section \ref{secStabBor}, for any $(u,v) \in  \mathcal{D}(\H) \times \mathcal{D}(\H)$, we can pass in Fourier series in the $y$ variable and write
\begin{equation*}
    A \begin{pmatrix}u \\ v\end{pmatrix} = \sum_{n =0}^{\infty} A_n  \begin{pmatrix}u_n \\ v_n\end{pmatrix} \cos(n\pi L^{-1} y)
\end{equation*}
where $u_n$ and $v_n$ are the Fourier coefficients of $u$ and $v$ respectively, $D_n = - \partial_{xx} + (n\pi L^{-1})^2 + \gamma \delta_0 + \omega$ and
\begin{equation*}
    A_n = \begin{pmatrix}0 & - D_n + p (\phi_{\omega, \gamma})^{p-1} \\ D_n  - (\phi_{\omega, \gamma})^{p-1} & 0 \end{pmatrix}
\end{equation*}
Notice that Lemma \ref{lemSpABor} is true for any $A_n$ with $n \in \mathbb{N}$, thus the essential spectrum of $A_n$ is included in $\{iz : |z| \geq \omega + (n\pi L^{-1})^2, z \in \mathbb{R}\}$ and the discrete spectrum consists of a finite number of isolated eigenvalues with finite multiplicity. We show the spectral mapping property for each $A_n$:

\begin{lem}\label{lemSMTBor}
   For any $n \in \R, n \geq 0$, it holds that \(\sigma(e^{t A_n}) \setminus \{0\} = e^{t \sigma(A_n)}\).
\end{lem}

The proof of the above lemma is based on the following general result, see for instance \cite[p.$95$]{Ar86SMT}.

\begin{theorem} \label{thmArendt}
Let $B$ be the generator of a $C_0$ semigroup $\{e^{tB}\}_{t \geq 0}$ on a Hilbert space $X$. If $\sigma(B) \cap \{z = \alpha + i\tau : \tau = 0, \alpha \neq 0\} = \emptyset$ then $\sigma(e^{tB}) \cap \{ z \in \C : z = e^{t\alpha}\} = \emptyset$ if and only if 
\begin{equation}\label{eqPruss}
    \sup_{\tau \in \R} \|(\alpha + i \tau  - B)^{-1}\|_{\mathcal{L}(X)} < \infty.
\end{equation}
\end{theorem}

\begin{proof}[Proof of Lemma \ref{lemSMTBor}]
   Let $z = \alpha + i \tau$ with $\alpha, \tau \in \R$ and $\alpha, \tau \neq 0$. We show that \eqref{eqPruss} holds for $B = A_n$ for any $n \in \R, n \geq 0$. In particular, we have to show the boundedness for $|\tau| \to \infty$. To this end, we write $z - A_n$ as 
\begin{equation*}
    z - A_n = \begin{pmatrix}z & - D_n  \\ D_n & z \end{pmatrix} \left( I + T_n(z) \right) 
\end{equation*}
where
\begin{equation*}
    T_n(z) =  \begin{pmatrix}z & - D_n  \\ D_n & z \end{pmatrix}^{-1}  \begin{pmatrix}0 & p (\phi_{\omega, \gamma})^{p-1} \\ - (\phi_{\omega, \gamma})^{p-1} & 0 \end{pmatrix}.
\end{equation*}
Notice that $D_n^2$ is a positive self-adjoint operator, thus $z^2 + D_n^2$ is invertible when $\tau \neq 0$, and with the inverse $(z^2 + D_n^2)^{-1}: L^2(\R) \to \mathcal{D}(\H_n^2)$ being a bounded operator.
By direct computation, one can check that
\begin{equation*}
    \begin{pmatrix}z & - D_n  \\ D_n & z \end{pmatrix}^{-1} = \begin{pmatrix}z (z^2 + D_n^2)^{-1} & D_n (z^2 + D_n^2)^{-1} \\ - D_n (z^2 + D_n^2)^{-1} & z (z^2 + D_n^2)^{-1} \end{pmatrix}
\end{equation*}
which implies that
\begin{equation*}
     T_n(z)  = \begin{pmatrix} - D_n (z^2 + D_n^2)^{-1} (\phi_{\omega, \gamma})^{p-1} & p z (z^2 + D_n^2)^{-1}  (\phi_{\omega, \gamma})^{p-1} \\ - z (z^2 + D_n^2)^{-1} (\phi_{\omega, \gamma})^{p-1} & - p D_n  (z^2 + D_n^2)^{-1}  (\phi_{\omega, \gamma})^{p-1} \end{pmatrix}
\end{equation*}
We proceed by showing that $(I +T_n(z)): L^2(\R) \times L^2(\R) \to L^2(\R) \times L^2(\R)$ is a bounded operator, and remains bounded for $|\tau| \to \infty$ uniformly in $n \geq 0$. Indeed, on the one hand, by the spectral theorem for self-adjoint operators, we have that 
\begin{equation}\label{eqEst1Bor}
   \|D_n (z^2 + D_n^2)^{-1}\|_{\mathcal{L}(L^2(\R))} \leq \sup_{\lambda \in \sigma(D_n)} \lambda | z^2 + \lambda^2|^{-1} = \sup_{\lambda \in \sigma(D_n)} \frac{\lambda}{\sqrt{(\alpha^2 - \tau^2 + \lambda^2)^2 + 4\alpha \tau}} \leq C_1
\end{equation}
 where $C_1$ is uniform in $n\geq 0$ and $\tau$. On the other hand, by the same argument, we have that 
 \begin{equation}\label{eqEst2Bor}
     \|z (z^2 + D_n^2)^{-1}\|_{\mathcal{L}(L^2(\R))} \leq \sup_{\lambda \in \sigma(D_n)} \frac{|z|}{|z^2 + \lambda|} \leq \sup_{\lambda \in \sigma(D_n)} \frac{\sqrt{\alpha^2 + \tau^2}}{|Im(z^2)|} = \frac{\sqrt{\alpha^2 + \tau^2}}{2 \alpha \tau} \leq C_2
 \end{equation}
 where $C_2$ is uniform in $n\geq 0$ and $|\tau|$ sufficiently big. Thus, as $(\phi_{\omega, \gamma})^{p-1}$ is a multiplication operator by a function in $L^\infty(\R)$, we conclude that $T_n(z)$ is a bounded operator on $L^2(\R) \times L^2(\R)$ with the bound uniform in $n \geq 0$ and $\tau$. In particular, for $|\tau|$ large enough, we have that $\|T_n(z)\|_{\mathcal{L}(L^2(\R) \times L^2(\R))} < 1$, thus $(I + T_n(z))^{-1}$ exists and is bounded. 

Since for sufficiently large $|\tau|$, we have that $\|T_n(z)\| < 1$, we can write
\begin{equation*}
    \left\| (z - A_n)^{-1} \right\| = \left\| (I + T_n(z))^{-1} \right\| \left\| \begin{pmatrix}z & - D_n  \\ D_n & z \end{pmatrix}^{-1} \right\| \leq \frac{1}{1 - \|T_n(z)\|} \left\|\begin{pmatrix}z & - D_n  \\ D_n & z \end{pmatrix}^{-1} \right\| .    
\end{equation*}
From \eqref{eqEst1Bor} and \eqref{eqEst2Bor}, it follows that 
$$\begin{pmatrix}z & - D_n  \\ D_n & z \end{pmatrix}^{-1}$$ is a bounded operator on $L^2(\R) \times L^2(\R)$ with the bound is uniform in $n \geq 0$ and $\tau$. This concludes the proof.
\end{proof}

\section{Linear instability}\label{AppenLinStab}
In this section, we keep the notations of Section \ref{sec:stab}. We show the linear instability of the line soliton.  We shall show that for $L > L_*$, 
$J\mathcal{S}(\phi_{\omega, \gamma}^{(L)})$ has at least one positive eigenvalue. 

We see from \eqref{eq-linearizedop} 
that $\mathcal{S}(\phi_{\omega, \gamma}^{(L)})$ 
has an eigenvalue if and only if there exists $n \in \mathbb{Z}$ such that 
$J S_{L, \gamma}^{\prime \prime}(n)$ has the same eigenvalue. 
Therefore, 
in order to obtain the linear instability of the line soliton $\phi_{\omega, \gamma}^{(L)}$, 
it is sufficient to show that 
$J S_{L, \gamma}^{\prime \prime}(1)$ has a positive eigenvalue. 
We obtain the following:
\begin{lem}
\label{lem:posi-eigen}
For $L > L_*$, 
$J S_{L, \gamma}^{\prime \prime}(1)$ has at least one positive eigenvalue. 
\end{lem}
\begin{proof}

We recall that $\chi$ is the normalized eigenfunction of 
$\mathcal{L}_{+, 0}$ associated with the eigenvalue 
$-\lambda_1$ with $\|\chi\|_{L^{2}(\R)} = \sqrt{2}$. 
We put 
\begin{equation*}
\vec{\chi} = 
\begin{pmatrix}
\chi \\
0
\end{pmatrix}
\end{equation*} 
and
\begin{equation*}
U(v, a, \lambda):= 
S_{L, \gamma}^{\prime \prime}(a) 
(\vec{\chi}+v) + \lambda J^{-1} (\vec{\chi}+v),
\end{equation*}
with $v \in {\rm Span}(\vec{\chi})^{\perp}$ 
and $\omega, \ \lambda >0$, 
where
\begin{equation*}
{\rm Span}(\vec{\chi})^{\perp} = 
\left\{v 
= {}^{t}(v_{1}, v_{2})
\in L^2(\mathbb{R}) \times L^2(\mathbb{R}) 
\colon \langle v, \vec{\chi} \rangle = 0 \right\}. 
\end{equation*}
We know that the kernel of 
$S_{L, \gamma}^{\prime \prime}(\nu_{L})$ 
is spanned by $\vec{\chi}$, where 
$\nu_{L} = L/L_{*}$. 
Note that $\nu_{L} > 1$ for $L >  L_{*}$. 
We see that $U$ is a $C^\infty$ function from 
${\rm Span} (\vec{\chi})^{\perp} \times \mathbb{R} \times \mathbb{R}$ 
to $L^2(\mathbb{R})\times L^2(\mathbb{R})$ 
and 
\begin{equation*}
U(0, \nu_{L}, 0) 
= S_{L, \gamma}^{\prime \prime}(\nu_{L})\vec{\chi} = 0. 
\end{equation*}
Differentiate $U(v, a, \lambda)$ with respect to 
$a, v$ and substituting $v = 0, a = \nu_{L}$ and $\lambda = 0$, 
we have 
\begin{equation*}
(D_{a}U(0, \nu_{L}, 0), D_{v}U(0, \nu_{L}, 0)) 
(\mu, w)
= ( 2(\frac{\pi}{L})^2 \nu_{L} \mu \vec{\chi}, 
S_{L, \gamma}^{\prime \prime}
(\nu_{L}) w).
\end{equation*}
for any $\mu \in \mathbb{R}$ and 
$w \in {\rm Span}(\vec{\chi})^{\perp}$. 
Thus, $D_{a} U(0, \nu_{L}, 0) \in \mathcal{L}(\mathbb{R}, 
L^2(\mathbb{R})\times L^2(\mathbb{R}))$ 
is invertible 
since it is a linear one-to-one mapping from $\mathbb{R}$ to ${\rm Span}(\vec{\chi})$. 
In addition, from the fact that the kernel of 
$S_{L, \gamma}^{\prime \prime}(\nu_{L})$ is spanned by $\vec{\chi}$,  
we infer that $D_{v}U(0, \nu_{L}, 0) \in 
\mathcal{L}({\rm Span} (\vec{\chi})^{\perp}, 
L^2(\mathbb{R}) \times L^2(\mathbb{R}))$ is invertible. 
Then, by the implicit function theorem, there exist two $C^\infty$ functions 
$a(\lambda) \in \mathbb{R}$ and 
$\eta(\lambda) \in {\rm Span} (\vec{\chi})^{\perp}$ 
satisfying $a(0)= \nu_{L}$, $\eta(0)=0$ and 
$U(\eta(\lambda), a(\lambda), \lambda)=0$, that is, 
\begin{equation*}
0 = U(\eta(\lambda), a(\lambda), \lambda) = 
S_{L, \gamma} (a(\lambda)) (\vec{\chi} 
+\eta(\lambda)) + \lambda J^{-1} (\vec{\chi}+\eta(\lambda)),
\end{equation*}
for sufficiently small $\lambda$. 
Differentiating the previous identity with respect to $\lambda$, 
we have
\begin{equation} \label{identity:U'}
\begin{split}
0 = \frac{d  }{d \lambda} U(\eta(\lambda), a(\lambda), \lambda)
& = 
2(\frac{\pi}{L})^2
a(\lambda) a^{\prime}(\lambda)(\vec{\chi} + \eta(\lambda)) 
+S_{L, \gamma}^{\prime \prime}
(a(\lambda)) \eta^{\prime}(\lambda) \\
& \quad + J^{-1}(\vec{\chi}+\eta(\lambda))
+\lambda J^{-1} \eta^{\prime}(\lambda).
\end{split}
\end{equation}
Since $\eta(0) = 0$ and $a(0) = \nu_{L}$, 
for $\lambda=0$, we obtain
\begin{equation}
\label{identity:lambda=0}
0 = 
(\frac{\pi}{L})^2
\nu_{L}
a^{\prime}(0) \vec{\chi} + 
S_{L, \gamma}^{\prime \prime}(\nu_{L})\eta^{\prime}(0) + J^{-1}\vec{\chi}.
\end{equation}
Taking the $L^2$-scalar product of identity \eqref{identity:lambda=0} 
with $\vec{\chi}$ and using the fact that $\vec{\chi} \in 
\mathop{\mathrm{Ker}} S_{L, \gamma}^{\prime \prime}(\nu_{L})$ and 
$J$ is a skew-symmetric operator, we get
\begin{equation*}
0 = 
(\frac{\pi}{L})^2 
\nu_{L}
a^{\prime}(0) \|\vec{\chi}\|_{L_{x}^2(\R) \times L_{x}^{2}(\R)}^{2}.
\end{equation*}
Thus, $a^{\prime}(0)=0$. 
Substituting $a^{\prime}(0) = 0$ into \eqref{identity:lambda=0}, we see that
\begin{equation}
\label{identity;v(0)}
S_{L, \gamma}^{\prime \prime}(\nu_{L}) \eta^{\prime}(0) =- J^{-1}\vec{\chi}.
\end{equation}
Differentiate \eqref{identity:U'} with respect to $\lambda$ again, 
we have 
\begin{equation*}
\begin{split}
0 = \frac{\partial^2 U}{\partial \lambda^2} (\eta(\lambda),\omega(\lambda),\lambda) 
& =  
2 (\frac{\pi}{L})^2 a(\lambda)
a^{\prime \prime}(\lambda)(\vec{\chi} + \eta(\lambda)) 
+ 2 (\frac{\pi}{L})^2
(a^{\prime}(\lambda))^2(\vec{\chi} + \eta(\lambda)) \\
& \quad 
+ 2 (\frac{\pi}{L})^2 
a(\lambda) a^{\prime}(\lambda)
\eta^{\prime}(\lambda)
+ 2(\frac{\pi}{L})^2
a(\lambda) a^{\prime}(\lambda) \eta^{\prime}(\lambda)
+ S_{L, \gamma}^{\prime \prime}
(a(\lambda))\eta^{\prime \prime}(\lambda) \\
& \quad 
+ 2 J^{-1}\eta^{\prime}(\lambda) + \lambda J^{-1}\eta^{\prime \prime}(\lambda). 
\end{split}
\end{equation*}
Since $\eta(0) = 0$, $a(0) = \nu_{L}$ and 
$a^{\prime}(0) = 0$, putting $\lambda=0$, we have
\begin{equation*}
0 = 
2 (\frac{\pi}{L})^2 \nu_L
a^{\prime \prime}(0) \vec{\chi} + 
S_{L, \gamma}^{\prime \prime}(\nu_{L})\eta^{\prime \prime}(0) 
+ 2 J^{-1}\eta^{\prime}(0).  
\end{equation*}
Multiplying the above identity by $\vec{\chi}$ and 
integrating the resulting equation, we have 
\begin{equation*}
0 = 
2 (\frac{\pi}{L})^2 \nu_L
a^{\prime \prime}(0) 
\|\vec{\chi}\|_{L_{x}^{2}(\R) \times L_{x}^{2}(\R)}^{2} 
+ 2 \langle J^{-1} \eta^{\prime}(0), \vec{\chi} \rangle. 
\end{equation*}
Here, we have used the fact that 
$S_{L, \gamma}^{\prime \prime}(\nu_{L})\vec{\chi} = 0$. 
This together with \eqref{identity;v(0)} yields that 
\begin{equation*}
\begin{split}
a^{\prime \prime}(0) 
= - \frac{\langle J^{-1}\eta^{\prime}(0), \vec{\chi} \rangle}
{(\frac{\pi}{L})^2 \nu_L \|\vec{\chi}\|_{L_{x}^{2}(\R) \times L_{x}^{2}(\R)}^{2}} 
& = \frac{\langle J^{-1}\eta^{\prime}(0), JS_{L, \gamma}^{\prime \prime}(\nu_{L})\eta^{\prime}(0) \rangle}
{(\frac{\pi}{L})^2 \nu_L 
\|\vec{\chi}\|_{L_{x}^{2}(\R) \times L_{x}^{2}(\R)}^{2}} \\
& = - 2 \frac{\langle \eta^{\prime}(0), S_{L, \gamma}^{\prime \prime}(\nu_{L})
\eta^{\prime}(0) \rangle}
{(\frac{\pi}{L})^2 \nu_L 
\|\vec{\chi}\|_{L_{x}^{2}(\R) \times L_{x}^{2}(\R)}^{2}} <0.
\end{split}
\end{equation*}
In the last inequality, we have used the fact that 
$\eta^{\prime}(0) \in 
\mathop{\mathrm{Span}} (\vec{\chi})^{\perp}$, 
$\mathop{\mathrm{Ker}} S_{L, \gamma}^{\prime \prime}(\nu_{L})
= \mathop{\mathrm{Span}} (\vec{\chi})$ and $S_{L, \gamma}^{\prime \prime}(\nu_{L})$ 
is non-negative operator. 
Therefore, there exists a sufficiently small $\varepsilon >0$ such that the function $a(\lambda)$ 
on $(0, \varepsilon)$ has the inverse function $\lambda(a)$
defined on a neighborhood of $a(0)=\nu_L$.
In other words, 
$JS_{L, \gamma}^{\prime \prime}(a)$ 
has a positive eigenvalue 
$\lambda(a)$ for $a \in (\lambda(\varepsilon), \nu_{L})$.

Let 
\begin{equation*}
 a_0 := \inf\{ a > 0 \colon \exists \ b \in (a, \nu_{L})
\mbox{ s.t. $JS_{L, \gamma}^{\prime \prime}
(b)$ has a positive eigenvalue}  \}.
\end{equation*}
From the above argument, we see that $a_{0} < \nu_{L}$. We shall show that $a_0 = 0$. Suppose the contrary that $0 < a_0 < \nu_{L}$. We take $\{a_n\}_{n=1}^{\infty} \subset  (a_0, \nu_{L})$ such that 
\[
\lim_{n \to \infty} a_n = a_0.
\]
By the perturbation theory, the sequence $\{a_{n}\}$ satisfies 
either 
\[
\lim_{n \to \infty} \lambda(a_n) = 0 \quad
{\rm or} \quad \lim_{n \to \infty}\lambda(a_n) = \infty.
\]
We claim that 
the second limit cannot hold. 
Indeed, taking $u_n\neq 0$ the associated eigenfunction corresponding to the eigenvalue $\lambda(a_n)$ with 
$\|u_{n}\|_{L_{x}^{2}(\R) \times L_{x}^{2}(\R)} = 1$, 
we see that
\begin{equation} \label{iden-JS}
\begin{split}
\langle J S_{L, \gamma}^{\prime \prime}
(a_n)u_n,u_n \rangle 
&= \langle (L_{L, -, a_{n}})\text{Re}\; u_n, \text{Im}\; u_n \rangle - 
\langle (L_{L, +, a_{n}}) \text{Im}\; u_n, \text{Re}\; u_n \rangle \\
& = -(p-1) \langle (\phi_{\omega, \gamma})^{p-1}_{\omega}\text{Im}\; u_n, \text{Re}\; u_n \rangle.
\end{split}
\end{equation}
Since $J S_{L, \gamma}^{\prime \prime}
(a_n)u_n = \lambda(a_{n}) u_{n}$ and $\|\phi_{\omega, \gamma}\|_{L_{x}^\infty(\R)}\lesssim 1$, it follows from \eqref{iden-JS} that 
\begin{equation*}
\lambda(a_n) \|u_n\|_{L_{x}^2(\R) \times L_{x}^{2}(\R)}^{2} 
= |\langle J S_{L, \gamma}^{\prime \prime}
(a_n)u_n, u_n \rangle| \lesssim 
\|u_n\|_{L_{x}^2(\R) \times L_{x}^{2}(\R)}^{2},
\end{equation*}
which means that $\lambda(a_n)\lesssim 1$.
As a consequence, we obtain that 
\begin{equation*}
\lim_{n \to \infty} \lambda(a_n) = 0. 
\end{equation*}
Next,  
since $-JS_{L, \gamma}^{\prime \prime}
(a_n)u_n=\lambda(a_n)u_n$, we have 
\begin{equation*}
S_{L, \gamma}^{\prime \prime}(a_0)u_n = 
S_{L, \gamma}^{\prime \prime} 
(a_{n}) u_{n} + (a_{0} - a_{n})u_{n} =
- J^{-1} \lambda (a_{n}) u_{n} + (a_{0} - a_{n}) u_{n}.
\end{equation*}
Since $S_{L, \gamma}^{\prime \prime}(a_0)$ 
is invertible 
and the inverse 
$(S_{L, \gamma}^{\prime \prime}(a_0))^{-1}$ is bounded, 
we obtain
\begin{equation*}
\begin{split}
1 = \|u_n\|_{L_{x}^{2}(\R) \times L_{x}^{2}(\R)}&= 
\|(S_{L, \gamma}^{\prime \prime}
(a_0))^{-1} (- J^{-1} 
\lambda (a_{n}) u_{n} + (a_{0} - a_{n}) u_{n})\|_{L_{x}^{2}(\R)} \\
&\lesssim 
\lambda(a_{n}) + |a_{0} - a_{n}|
\to 0 
\end{split}
\end{equation*}
as $n \to \infty$, 
which is a contradiction. 
Thus, we conclude that $a_0 = 0$, which implies that 
for $0 < a < \nu_{L}$, 
$JS_{L, \gamma}^{\prime \prime}
(a)$ 
has a positive eigenvalue.  
Since $\nu_{L} > 1$ for $L >  L_{*}$, 
we see that $JS_{L, \gamma}^{\prime 
\prime}(1)$ has a positive 
eigenvalue. 
This completes the proof. 
\end{proof}

\bibliographystyle{abbrv} %
\bibliography{biblio}

\end{document}